\documentclass[a4paper,10pt]{article}

\usepackage[T1]{fontenc}
\usepackage{lmodern}
\usepackage[utf8]{inputenc}
\usepackage{setspace}
\usepackage{indentfirst}

\usepackage{geometry}
\usepackage[new]{old-arrows}
\usepackage[hidelinks]{hyperref}
\usepackage{titlesec}
\usepackage{microtype}
\titleformat{\subsection}[runin]
  {\normalfont\large\bfseries}{\thesubsection.}{0.3em}{\addperiod}
  \newcommand{\addperiod}[1]{#1.}

\usepackage{amsthm}
\usepackage{amsmath}
\usepackage{amsfonts}
\usepackage{amssymb}
\usepackage{mathdots}
\usepackage{bm}
\usepackage{bbm}
\usepackage{bbold}
\usepackage{enumitem}
\usepackage{nicematrix}
\usepackage{stmaryrd}
\usepackage{tocloft}

\usepackage{graphicx}
\usepackage{xcolor}
\usepackage{color}
\usepackage{tikz-cd}
\usepackage{tikz}
\usetikzlibrary{matrix,arrows,decorations.pathmorphing,mindmap,fit}
\usepackage[usestackEOL]{stackengine}
\usepackage{fancybox}
\usepackage{caption}
\usepackage{dynkin-diagrams}

\newcommand\nnfootnote[1]{%
  \begin{NoHyper}
  \renewcommand\thefootnote{}\footnote{#1}%
  \addtocounter{footnote}{-1}%
  \end{NoHyper}
}
\usepackage{etoolbox}
\usepackage{tikzsymbols}

\usepackage{theoremref}
\newtheorem{prop}{Proposition}[section]
\newtheorem{lem}[prop]{Lemma}

\newtheorem{cor}[prop]{Corollary}
\newtheorem{thm}[prop]{Theorem}

\numberwithin{equation}{section}
\newtheorem{lem-de}[prop]{Lemma-Definition}

\theoremstyle{definition}
\newtheorem{defn}[prop]{Definition}
\newtheorem{de-lem}[prop]{Definition-Lemma}
\newtheorem{example}[prop]{Example}
\newtheorem{rem}[prop]{Remark}

\newtheorem{nota}[prop]{Notation}
\newtheorem{nota-rem}[prop]{Notation-Remark}

\newtheorem{nota-defn}[prop]{Notation-Definition}
\AtEndEnvironment{rem}{\hfill$\diamond$}

\DeclareMathOperator{\SL}{SL}

\DeclareMathOperator{\sign}{sgn}
\DeclareMathOperator{\tr}{tr}

\DeclareMathOperator{\GL}{GL}
\DeclareMathOperator{\diag}{diag} 
\DeclareMathOperator{\supp}{{\rm supp}}

\DeclareMathOperator{\Stab}{Stab}
\DeclareMathOperator{\Jac}{{\rm Jac}}
\DeclareMathOperator{\Rho}{{\rm P}}

\DeclareMathOperator{\A}{{\mathbb A}}
\DeclareMathOperator{\B}{{\mathbb B}}
\newcommand{\Span}{\mathrm{Span}}

\newcommand{\low}{\mathrm{low}}
\newcommand{\Pf}{\mathrm{Pf}}

\newcommand{\RN}[1]{%
  \textup{\uppercase\expandafter{\romannumeral#1}}%
}

\def \ind {{\rm ind}}

\def \CC {{\mathbb C}}
\def \ZZ {{\mathbb Z}}
\def \NN {{\mathbb N}}
\def \TT {{\mathbb T}}

\def \calA  {{\mathcal{A}}}           
\def \calB  {{\mathcal{B}}}
\def \calC  {{\mathcal{C}}}

\def \calI  {{\mathcal{I}}}

\def \calQ  {{\mathcal{Q}}}
\def \calP {{\mathcal{P}}}

\def \calT {{\mathcal{T}}}
\def \calU  {{\mathcal{U}}}

\def \g  {\mathfrak{g}}   
\def \h  {\mathfrak{h}}

\def \n  {\mathfrak{n}}

\def \b  {\mathfrak{b}}
\def \p  {\mathfrak{p}}
\def \r  {\mathfrak{r}}

\def \sl {\mathfrak{sl}}
\def \gl {\mathfrak{gl}}

\def \t  {\mathfrak{t}}
\def \l  {\mathfrak{l}}

\def \hs {\hspace{.2in}}
\def \hand {\hs \mbox{and} \hs}

\def \la {{\langle}}
\def \ra {{\rangle}}
\def \lara {\langle \,, \, \rangle}
\def \ow {\overline{w}}
\def \low {{\rm low}}
\def \sv {{\scriptscriptstyle \vee}}
\def \Ad {{\rm Ad}}
\def \pist {\pi_{\rm st}}
\def \TT {\mathbb{T}}
\def \lrw {\longrightarrow}
\def \Maps {\longmapsto}

\def \bfw {{\bf w}}
\def \bfwz {{\bf w_0}}
\def \vphi {{\varphi}}

\def \sGB {{\scriptscriptstyle G/B}}

\def \pist {\pi_{\rm st}}
\def  \piGB {\pi_{\sGB}}
\def \bbone {{\mathbbm 1}}
\def \wB {{w_\cdot B}}
\def \calCBFZ {\calC_{\rm BFZ}}
\def \gsts {{\g_{\rm st}^{\ast}}}

\def \gog {{\g \oplus \g}}

\def \elw {{\ell(w_0)}}

\begin{document}

\title{Polynomial integrable systems from cluster structures}
\author{Yanpeng Li \and Yu Li \and Jiang-Hua Lu}

\newcommand{\Addresses}{{
  \bigskip
  \footnotesize
  
  \textsc{Department of Mathematics, Sichuan University, No.24 South Section 1, Yihuan Road, Chengdu, China}\par\nopagebreak
  \textit{E-mail address}: \texttt{yanpeng.li@scu.edu.cn}

  \medskip


  \textsc{Department of Mathematics, University of Notre Dame, 255 Hurley Bldg., Notre Dame, IN 46556, USA}\par\nopagebreak
  \textit{E-mail address}: \texttt{yli234@nd.edu}

  \medskip

  \textsc{Department of Mathematics, The University of Hong Kong, Pokfulam Road, Hong Kong, China}\par\nopagebreak
  \textit{E-mail address}: \texttt{jhlu@maths.hku.hk}
}}
\date{}
\maketitle

\nnfootnote{\emph{Keywords:} cluster algebras, integrable systems, Poisson Lie groups, linearizations.}

\begin{abstract}
\noindent 
We present a general framework for constructing polynomial integrable systems on linearizations
of Poisson varieties that
admit log-canonical systems. Our construction is in particular applicable to Poisson varieties with compatible cluster or generalized cluster structures. As examples, we consider a standard complex semi-simple Poisson Lie group
$G$ and a Borel subgroup $B$ of $G$, equipped with the Berenstein-Fomin-Zelevinsky cluster structures;  
the unipotent Lie subgroup $N_w$ of $G$ associated to 
any $w$ in the Weyl group of $G$, equipped with the cluster structure on the corresponding 
Schubert cell 
as first defined by Geiss-Leclerc-Schr\"oer when $G$ is simply-laced; 
and the dual Poisson Lie group
$\GL(n, \CC)^*$ of the standard Poisson Lie group $\GL(n, \CC)$, equipped with the 
Gekhtman-Shapiro-Vainshtein generalized cluster structure.   In each of these four cases, we show that every extended cluster in the respective cluster 
or generalized cluster structure gives rise to at least one polynomial integrable system with respect to the linearization of the Poisson structure at the identity element. For some of the polynomial integrable systems, we show that all their Hamiltonian flows are complete. 

Just as generalized minors on a complex semi-simple Lie group $G$ 
are used to describe certain initial extended clusters in the Berenstein-Fomin-Zelevinsky cluster structure on $G$,
we introduce a special class of homogeneous polynomials, 
called {\it signed generalized minors}, on the Lie algebra $\g$ of $G$, which are then used 
to describe some of the
polynomial integrable systems obtained via our construction. As a further application, 
we use the homogeneous degrees of certain signed generalized minors to give an explicit formula 
for the index  of the Lie algebra of $N_w$  for every $w$ in the Weyl group. 

\end{abstract}

\tableofcontents 
\addtocontents{toc}{\protect\setcounter{tocdepth}{2}}

\section{Introduction and main results}\label{s:intro}
\subsection{Background and motivation}\label{ss:intro}
The theory of cluster algebras \cite{FZ:I, FZ:II, BFZ:III, FZ:IV} is now ubiquitous in mathematics. In this paper, we 
explore a connection between cluster structures and integrable systems on vector spaces with linear Poisson structures.

Recall that a linear Poisson structure on a finite dimensional complex\footnote{In this paper we only consider complex Poisson manifolds.}  vector 
space $V$ is, by definition, a Poisson 
structure $\pi_0$ on $V$ such that the linear functions on $V$ are closed under 
the Poisson bracket $\{\,, \, \}_{\pi_0}$
on the algebra of algebraic functions on $V$. 
Linear Poisson structures on $V$, also referred to as {\it Kirillov-Kostant-Souriau Poisson structures},
are thus in bijection with Lie algebra structures on the dual vector space $V^*$ of $V$. 

Let $\l$ be  an $n$-dimensional complex Lie algebra, and denote by $\pi_{0, \l}$ the linear 
Kirillov-Kostant-Souriau Poisson 
structure on $\l^*$. 
It is well-known (see $\S$\ref{ss:int-sys}) that a set of algebraically independent polynomial functions 
on $\l^*$ with pairwise zero Poisson brackets with respect to $\pi_{0, \l}$ has cardinality at most  
$n - \frac{1}{2}{\rm rk} (\pi_{0, \l})$, 
where ${\rm rk}(\pi_{0, \l})$ is the rank of $\pi_{0, \l}$ (see Definition \ref{defn:int}), 
and such a set is called a {\it polynomial integrable system} on $(\l^*, \pi_{0, \l})$ if its cardinality is equal to $n - \frac{1}{2}{\rm rk} (\pi_{0, \l})$.
The Mishchenko-Fomenko conjecture \cite{M-F:Poi}, 
proved by S. T. Sadetov in \cite{Sadetov}, says that $(\l^*, \pi_{0, \l})$ admits a polynomial integrable system
for every finite dimensional complex\footnote{The Mishchenko-Fomenko conjecture, as proved in 
\cite{Sadetov}, holds for finite dimensional Lie algebras over any field of characteristic $0$.}
Lie algebra $\l$. For a given Lie algebra $\l$, it has long been an interesting problem to 
concretely construct polynomial integrable systems on $(\l^*, \pi_{0, \l})$ (see, for example,
\cite{Vergne:CRAS,Vergne:nilpotent, Trofimov,Thimm, Bolsinov}).

Assume that $(P, \pi)$ is an $n$-dimensional  
complex Poisson manifold, and let $p_0 \in P$ be  such that $\pi(p_0) = 0$. The {\it linearization} of $\pi$ at $p_0$ is then a linear Poisson structure $\pi_0$ on the tangent 
space $T_{p_0}P$ of $P$ at $p_0$ (see $\S$\ref{ss:linPoi}), and we refer to $(T_{p_0}P, \pi_0)$ as the {\it linearization\footnote{In general $(P, \pi)$ is not necessarily {\it linearizable at $p_0$}, i.e.,
$(P, \pi)$ may not be locally isomorphic to $(T_{p_0}P, \pi_0)$.} of $(P, \pi)$ at $p_0$}. 
Our paper starts with the
observation (Lemma \ref{Lem:mainlem}) that if two local holomorphic functions $\vphi$ and $\psi$ on $P$ defined near $p_0$ 
have log-canonical Poisson bracket with respect to
$\pi$, i.e., if
\begin{equation}\label{eq:log-can}
\{\vphi, \psi\}_\pi \in \CC \vphi \psi,
\end{equation}
then their {\it lowest degree terms} $\vphi^\low$ and $\psi^\low$ at $p_0$
(see Definition \ref{defn:phi-low}), which are homogeneous
polynomial functions on $T_{p_0}P$, satisfy\footnote{See Remark \ref{rem:higher} for a  condition 
that is weaker than \eqref{eq:log-can}
but still implies \eqref{eq:low-0}.}
\begin{equation}\label{eq:low-0}
\{\vphi^\low, \, \psi^\low\}_{\pi_0} = 0.
\end{equation} 
For a set\footnote{The bracket $\{\,\cdot ,\ldots, \cdot \, \}$ used in denoting the set 
$\Phi = \{\vphi_1, \ldots, \vphi_n\}$ should not be confused with the Poisson bracket between two functions.
An $n$-tuple
$\Phi = (\vphi_1, \ldots, \vphi_n)$ of functions is sometimes also regarded as a set of functions.} $\Phi = \{\varphi_1, \ldots, \varphi_n\}$   
of holomorphic functions on $P$, if 
elements of $\Phi$ have pairwise log-canonical Poisson brackets with respect to $\pi$ and if
$d\vphi_1 \wedge \cdots \wedge d\vphi_n \neq 0$, we call $\Phi$ a {\it log-canonical system on $(P, \pi)$}.
A log-canonical system on an open neighborhood of $p_0$ in $P$ is called a {\it log-canonical system on $(P, \pi)$ at $p_0$}. 
The above observation leads to the following natural question. 

\medskip
\noindent
{\bf Question.} 
{\rm
Given a log-canonical system $\Phi = \{\vphi_1, \ldots, \vphi_n\}$ on $(P, \pi)$ at $p_0$,
does the set 
\[
\Phi^\low = \{\vphi_1^\low, \ldots, \vphi_n^\low\}
\]
contain a polynomial integrable system on $(T_{p_0}P, \pi_0)$, i.e., a subset of
algebraically independent 
polynomials on $T_{p_0}P$ with cardinality 
$n - \frac{1}{2} {\rm rk}(\pi_0)$? 
}

\medskip
To answer the above question, we present in this paper both a general theory ($\S$\ref{s:int-linear}) and a detailed study of concrete examples from Lie theory
($\S$\ref{s:Lie-prelim}-$\S$\ref{sect:GLndual}). 
We remark that the operation of taking lowest degree terms has been used in other areas of mathematics under different names, such as deformations to the normal cone in intersection theory \cite{Ful} and weighted normal bundles in differential geometry \cite{LM23}.  The connections to integrable systems and cluster theory explored in our paper are new. Applying 
the operation to generalized minors on a complex semi-simple Lie group $G$, 
we are led to a class of homogeneous polynomials on the Lie algebra $\g$ of $G$, 
which we call {\it signed generalized minors} on $\g$, and which we use  
 to describe some of the polynomial integrable systems 
in the Lie theoretical examples
associated to $G$. 
As a further application, we give a formula for the indices of the 
nilpotent Lie sub-algebras of $\g$ associated to elements in the Weyl group of $G$ using the homogeneous degrees of certain
signed generalized minors.

We now give more details of our results. 

\subsection{The general theory}\label{ss:outine-general} As part of our general theory, we introduce a 
sufficient condition, to be referred to as {\it Property $\calI$},
 for a positive answer to the question  in $\S$\ref{ss:intro}.

{\bf Property $\calI$ of log-canonical systems}. 
Let $(P, \pi)$ be an $n$-dimensional
complex Poisson manifold  and let $p_0 \in P$ be such that
$\pi(p_0) = 0$. For a log-canonical system 
$\Phi = \{\vphi_1, \ldots, \vphi_n\}$ on $(P, \pi)$ at $p_0$,
define the meromorphic $n$-form $\mu_\Phi$ on $P$ by
\[
\mu_\Phi = \frac{d\vphi_1 \wedge \cdots \wedge d\vphi_n}{\vphi_1 \cdots \vphi_n} = d(\log \vphi_1) \wedge \cdots \wedge 
d(\log \vphi_n),
\]
which will also be called the {\it log-volume form of $\Phi$}. 
Then $\mu_\Phi$ has a well-defined {\it lowest degree term} $\mu_{\Phi}^\low$ at $p_0$  (see $\S$\ref{ss:phi-low}), 
which is a non-zero homogeneous rational $n$-form on $T_{p_0}P$. Let  
$\deg (\mu_{\Phi}^\low) \in \ZZ$ be the 
homogeneous degree of $\mu_{\Phi}^\low$ (see \eqref{eq:deg-alpha-low}). 
The following is the first main result of the paper.

\medskip   
\noindent
{\bf Theorem A} (Corollary \ref{cor:I} and Theorem \ref{thm:main-0}). {\it For a
log-canonical system $\Phi$ on $(P, \pi)$ at $p_0$, if 
\begin{equation}\label{eq:deg-mu-Phi-r0}
\deg (\mu_{\Phi}^\low) \leq \frac{1}{2} \, {\rm rk}(\pi_0), \hs \mbox{equivalently,} \hs
\deg (\mu_{\Phi}^\low) = \frac{1}{2} \, {\rm rk}(\pi_0),
\end{equation}
then the set $\Phi^\low$ 
contains a polynomial integrable system on $(T_{p_0}P, \pi_0)$. 
}

\medskip
\noindent
{\bf Definition A}.
If \eqref{eq:deg-mu-Phi-r0} holds, we say that the log-canonical system $\Phi$ on $(P, \pi)$ at $p_0$
{\it has Property $\calI$ at $p_0$}.  

\medskip
{\bf Property $\calI$ of compatible cluster structures}.
Examples of Poisson varieties with log-canonical systems have recently emerged in the study of {\it cluster varieties}.
Let $(P, \pi)$ be a smooth rational quasi-affine Poisson variety over $\CC$, and let $\CC(P)$ 
be the field of rational functions on $P$. By a {\it (generalized) cluster structure in $\CC(P)$} we mean
an equivalence class of seeds in $\CC(P)$  under either
the ordinary mutation rule as in \cite{FZ:I, FZ:II} or the generalized mutation rule 
as in \cite{GSV:2012, GSV:2018} (see $\S$\ref{ss:cluster}). 

Assume that
$\calC$ is  a  (generalized) cluster structure in $\CC(P)$ 
that is {\it regular on $P$ and compatible with $\pi$} in the sense 
that all the extended clusters of $\calC$
consist of regular functions\footnote{This condition is weaker than
$\calC$ being a (generalized) cluster structure on $P$.  See  Remark \ref{rem:weaker}.}
on $P$ and are log-canonical with respect to
$\pi$. Let ${\rm Froz}$ be the set of all frozen variables of $\calC$ and let $m = |{\rm Froz}|$. 
Motivated by the examples studied in the second part of the paper, we define a {\it regular $\pi$-compatible frozen variable modification of $\calC$} to be an algebraically independent  subset $\overline{{\rm Froz}}$
of $\CC({\rm Froz})\cap \CC[P]$ of cardinality $m$ 
with certain additional conditions (see Definition \ref{defn:modify} and 
Definition \ref{defn:regular} for detail). Given such a 
$\overline{{\rm Froz}}$, for an extended cluster $\Phi$ of $\calC$ we define 
the {\it modification}
of $\Phi$ by $\overline{{\rm Froz}}$ to be
\[
\overline{\Phi} = (\Phi\backslash {\rm Froz}) \sqcup \overline{{\rm Froz}}.
\]
The following statements are proved in $\S$\ref{ss:cluster}. 

\medskip   
\noindent
{\bf Theorem B} (Lemma-Definition \ref{lem:mu-modify}, Theorem \ref{thm:main-cluster}). {\it Let $(P, \pi)$ be a smooth 
rational quasi-affine Poisson  variety over $\CC$. Let  $\calC$ be a 
 (generalized) cluster structure in  $\CC(P)$ that is {\it regular on $P$ and compatible with $\pi$}, and let  
 $\overline{\rm Froz}$ 
 be any regular $\pi$-compatible frozen variable modification of $\calC$.
 
 1) For every extended cluster $\Phi$ of $\calC$, the modification $\overline{\Phi}$ of ${\Phi}$ is 
 again a log-canonical system on $(P, \pi)$, and 
the meromorphic form $\overline{\mu}_{\calC} \stackrel{{\rm def}}{=} \mu_{\overline{\Phi}}$ on $P$ 
is, up to a sign, independent of the extended cluster $\Phi$ of $\calC$;

2) Suppose that $p_0\in P$ is such that $\pi(p_0) = 0$, and let $\pi_0$ be the linearization of $\pi$ at $p_0$. 
If the lowest degree term $(\overline{\mu}_{\calC})^\low$ of $\overline{\mu}_{\calC}$ at $p_0$ satisfies
\begin{equation}\label{eq:deg-mu-C-0}
\deg \left((\overline{\mu}_{\calC})^\low\right) = \frac{1}{2} {\rm rk}(\pi_0),
\end{equation}
then for every extended cluster $\Phi$ of $\calC$, 
the modification $\overline{\Phi}$ of $\Phi$ 
has Property $\calI$ at $p_0$, so  the set
$(\overline{\Phi})^{\low}$ of lowest degree terms of elements in $\overline{\Phi}$ at $p_0$
contains a polynomial integrable system on $(T_{p_0}P, \pi_0)$.
}

\medskip 
For a smooth Poisson variety $(P, \pi)$ with $p_0 \in P$ such that $\pi(p_0) = 0$ and 
 a (generalized) cluster structure $\calC$ in $\CC(P)$ that is regular
on $P$ and compatible with $\pi$, 
Theorem B says that
Property $\calI$ at $p_0$, or the lack thereof, is shared by all
(the simultaneous frozen variable modifications of) the extended clusters of $\calC$. 

\medskip
\noindent
{\bf Definition B.} Let $(P, \pi)$ be a smooth 
rational quasi-affine Poisson  variety over $\CC$, and let $p_0 \in P$ be such that $\pi(p_0) = 0$.
For a (generalized) cluster structure $\calC$ in $\CC(P)$ that is regular
on $P$ and compatible with $\pi$, we say that  {\it $\calC$ has Property $\calI$ at $p_0$} 
if there exists a frozen variable modification of $\calC$ such that 
\eqref{eq:deg-mu-C-0} holds. 

\medskip
{\bf Property $\calI$ via $\TT$-Pfaffian}.
In $\S$\ref{ss:Pfaffian}, we investigate another setting in which Property $\calI$ at $p_0 \in P$ can be simultaneously 
verified for a class of log-canonical systems on $(P, \pi)$. Namely, 
assume that $\TT$ is a complex torus and $(P, \pi)$ is an $n$-dimensional connected complex $\TT$-Poisson manifold, i.e., $(P, \pi)$ has  an action of $\TT$ by Poisson automorphisms. Assume, in addition, that $(P, \pi)$ has an open dense $\TT$-leaf, i.e., 
an open dense subset $L$ of the form
$L =\cup_{t \in \TT} t S$, where $S$ is a symplectic leaf of $(P, \pi)$. In such a case, one has, up to 
non-zero scalar multiples, 
a well-defined {\it $\TT$-Pfaffian} (see $\S$\ref{ss:Pfaffian} for detail)
\[
{\rm Pf}_\TT(\pi) = \wedge^r \pi \wedge \sigma(x_1) \wedge \cdots \wedge \sigma(x_{n-2r}) \in {\mathfrak{X}}^n(P),
\]
where $r = \frac{1}{2} {\rm rk}(\pi)$,   and the $\sigma(x_j)$'s are certain 
generating vector fields of the $\TT$-action on $P$. 
Define a {\it $\TT$-log-canonical system} on
$(P, \pi)$ to be any log-canonical system $\Phi = \{\vphi_1, \ldots, \vphi_n\}$ on $(P, \pi)$ 
with the additional property that each $\vphi_j$ is a $\TT$-weight vector with respect to the $\TT$-action.

As is with differential forms,  the $\TT$-Pfaffian 
${\rm Pf}_\TT(\pi)$ has a lowest degree term ${\rm Pf}_\TT(\pi)^\low$  at any $p_0 \in P$, which is a polynomial $n$-vector field on $T_{p_0}P$ and has
a homogeneous degree
$\deg \left({\rm Pf}_\TT(\pi)^\low\right) \in \ZZ$  
(see $\S$\ref{ss:phi-low}).
The following statements are proved in $\S$\ref{ss:Pfaffian}.

\medskip   
\noindent
{\bf Theorem C} (Lemma \ref{lem:mu-Pfaff}, Theorem \ref{thm:mu-Pfaff}, and Theorem \ref{thm:TT-fixed}). {\it 
Let $\TT$ be a complex torus and let 
$(P, \pi)$ be a connected complex $\TT$-Poisson manifold with  an open dense $\TT$-leaf. 

1) For every 
$\TT$-log-canonical system $\Phi$ on $(P, \pi)$, taking the lowest degree terms 
at any $p_0 \in P$ one has
\[
\deg \left(\mu_\Phi^\low\right) =-\deg \left({\rm Pf}_\TT(\pi)^\low\right).
\]
Consequently, if  $\pi(p_0) = 0$, and if  
\begin{equation}\label{eq:deg-Pfaffian}
\deg \left({\rm Pf}_\TT(\pi)^\low\right)=-\frac{1}{2} {\rm rk}(\pi_0),
\end{equation}
where $\pi_0$ is the linearization of $\pi$ at $p_0$, then every $\TT$-log-canonical system 
$\Phi$ on $(P, \pi)$ has Property $\calI$ at $p_0$. 

2) Assume that $\TT$-log-canonical systems on $(P, \pi)$ exist. Then the 
identity \eqref{eq:deg-Pfaffian}  automatically holds for every $p_0 \in P$ that is a $\TT$-fixed point and is such that $\pi(p_0)=0$.
Consequently, 
all $\TT$-log-canonical systems have Property $\calI$ at such a $p_0 \in P$. 
}

\subsection{Examples from Lie theory}\label{ss:three-examples}
We consider four examples from Lie theory.

Let $G$ be a connected and simply connected complex semi-simple Lie group
with a fixed pair $(B, B_-)$
of opposite Borel subgroups, and let $\g$, $\b$, and $\b_-$ be the respective Lie algebras of $G$, $B$, 
and $B_-$. The choice of $(B, B_-)$ gives rise to a so-called 
{\it standard multiplicative Poisson structure $\pi_{\rm st}$}, 
which vanishes at the identity element $e \in G$. The linearization of $(G, \pi_{\rm st})$ at $e \in G$ is
\begin{equation}\label{eq:g-pi0}
(\g, \;\pi_{0, \g_{\rm st}^*}),
\end{equation}
where 
$\g_{\rm st}^*$ is a Lie sub-algebra of $\b \times \b_-$ given in \eqref{eq:gsts}, identified with $\g^*$ 
via \eqref{eq:g-gsts}.
On the other hand, the upper cluster structure on the open double Bruhat cell of $G$ defined by
A. Berenstein, S. Fomin, and A. Zelevinsky  in \cite{BFZ:III}, regarded as a cluster
structure in $\CC(G)$ and denoted as $\calC_{\rm BFZ}(G)$, is regular on $G$ (see \cite[Theorem 5.1]{GSV:2012} and \cite{Qin-Yakimov:G}) and compatible with  $\pist$ 
(see \cite[Theorem 4.18]{GSV:book}). We refer to $\calC_{\rm BFZ}(G)$ as the {\it BFZ cluster structure in $\CC(G)$}. 

The Borel subgroup $B$ is a Poisson Lie subgroup of $(G, \pi_{\rm st})$, and we also write $\pi_{\rm st} =
\pi_{\rm st}|_B$. 
Using the Killing form of $\g$ to identify $\b \cong (\b_-)^*$ (see \eqref{eq:bb-pair}), the linearization of 
$(B, \pi_{\rm st})$ at $e \in B$ is 
\begin{equation}\label{eq:b-pi0}
(\b, \;\pi_{0, \b_-}).
\end{equation}
On the other hand, 
let $W$ be the Weyl group of $G$ with respect to $T = B \cap B_-$, and let
$w_0 \in W$ be the longest element. 
The Berenstein-Fomin-Zelevinsky cluster structure on the double Bruhat cell
$G^{e, w_0} = B \cap B_-w_0B_-$, regarded as a cluster structure in $\CC(B) = \CC(G^{e, w_0})$
and denoted as $\calC_{\rm BFZ}(B)$, 
is regular and compatible with $\pi_{\rm st}$ (see Lemma 
 \ref{lem:comp-T-B}). We refer to $\calC_{\rm BFZ}(B)$ as the {\it BFZ cluster structure in $\CC(B)$}. 
 
Consider now the flag variety $G/B$ of $G$
and its decomposition
\[
G/B = \bigsqcup_{w \in W} BwB/B
\]
 into Schubert cells.
 It is well-known \cite{GY09, LM:mixed}  that
the Poisson structure $\pi_{\rm st}$ on $G$ projects to a well-defined Poisson structure $\piGB$ on $G/B$, 
and that  each Schubert cell $BwB/B$ is
a Poisson sub-manifold of $(G/B, \piGB)$ and  $\piGB(w_\cdot B) = 0$. 
We write $\pi_w = \piGB|_{BwB/B}$ for each $w \in W$ and call $\pi_w$ the 
{\it standard Poisson structure on $BwB/B$}.
Let $\n \subset \b$ and $\n_-\subset \b_-$ be the respective nilpotent radicals of $\b$ and $\b_-$. For $w \in W$, 
and for any representative $\dot{w}$ of $w$ in the normalizer subgroup $N_G(T)$ of $T$ in $G$, let
\begin{equation}\label{def:nw-nw}
\n_w = \n \cap {\rm Ad}_{\dot{w}}(\n_-)\hs \mbox{and} \hs
\n_{w, -} = \n_- \cap {\rm Ad}_{\dot{w}}(\n).
\end{equation}
Identifying  $T_{w_\cdot B}(BwB/B) \cong \n_w \cong (\n_{w, -})^*$ (see $\S$\ref{ss:linear-Cw}), the linearization of
$(BwB/B, \pi_w)$ at $w_\cdot B$ is then
\begin{equation}\label{eq:nw-pi0}
(\n_w, \; \pi_{0, \n_{w, -}})
\end{equation}
up to a factor of $-2$ (see Lemma \ref{lem:pi0-nw}). 
On the other hand, the Schubert cell $BwB/B$ carries the so-called {\it standard cluster structure}, denoted as
$\calC^w_{\rm st}$, 
which is compatible with $\pi_w$ and with the $T$-action on $BwB/B$ by left translation (in the sense that all extended 
cluster variables are $T$-weight vectors). 
The cluster structure $\calC^w_{\rm st}$ 
was first introduced by 
C. Geiss, B. Leclerc, and J. Schr\"oer 
\cite{GLS:partial} when $G$ is simply-laced, and the general case follows from  the 
Goodearl-Yakimov theory on {\it symmetric Poisson CGL extensions} in \cite{GY:Poi-CGL} (see 
$\S$\ref{ss:st-cluster-BwB} for detail). 

Finally, consider the dual Poisson Lie group $(\GL(n, \CC)^*, \pi_{\rm st}^*)$ of the 
Poisson Lie group $(\GL(n, \CC), \pi_{\rm st})$, 
where $\GL(n, \CC)^*$ is a certain closed subgroup of $B \times B_-$ and $B$ (resp. $B_-$)
consists of all the upper triangular (resp. lower triangular) matrices in $\GL(n, \CC)$
(see $\S$\ref{ss:GLn-star}).
Let ${\mathbbm{1}}_n$ be the $n \times n$ identity matrix. The Poisson structure $\pi_{\rm st}^*$ vanishes at $({\mathbbm{1}}_n, {\mathbbm{1}}_n) \in \GL(n, \CC)^*$, and the linearization of $\pi_{\rm st}^*$ at $({\mathbbm{1}}_n, {\mathbbm{1}}_n)$ is
\begin{equation}\label{eq:gl-pi0}
({\mathfrak{gl}}(n, \CC), \; \pi_{0, {\mathfrak{gl}}(n, \CC)}),
\end{equation}
where we identify ${\mathfrak{gl}}(n, \CC)\cong {\mathfrak{gl}}(n, \CC)^*$ using the trace form.
On the other hand, M. Gekhtman, M. Shapiro, and A. Vainshtein constructed in \cite{GSV:related} a generalized upper
cluster structure on $\GL(n, \CC)^*$, denoted as $\calC_{\rm GSV}$, which is compatible with the Poisson structure $\pi_{\rm st}^*$ (see 
$\S$\ref{ss:cluster-G-star} for detail). We prove

\medskip
\noindent
{\bf Theorem D}. 
 The following hold:

($\S$\ref{s:GB}) {\it The standard cluster structure $\calC^w_{\rm st}$ on each Schubert cell $BwB/B$ has Property $\calI$ at $w_\cdot B$};

($\S$\ref{s:G}) {\it The  cluster structure $\calC_{\rm BFZ}(G)$ in $\CC(G)$ has Property $\calI$ at $e \in G$;}

($\S$\ref{s:B}) {\it The  cluster structure $\calC_{\rm BFZ}(B)$ in $\CC(B)$ has Property $\calI$ at $e \in B$;}

($\S$\ref{sect:GLndual}) {\it The generalized cluster structure $\calC_{\rm GSV}$ in $\CC(\GL(n, \CC)^*)$ has Property $\calI$ at $({\mathbbm{1}}_n, {\mathbbm{1}}_n) \in \GL(n, \CC)^*$}.

\medskip 
We note that for 
$\calC_{\rm BFZ}(G)$ and $\calC_{\rm BFZ}(B)$ and when $w_0\neq -1$ on the root system of $G$, as well as for
$\calC_{\rm GSV}$, we need to do frozen variable modifications for Property $\calI$ as stated in Theorem D. No frozen variable modification is needed, and Theorem C applies, in the case of Schubert cells.

In each of the cases in Theorem D and for every extended cluster $\Phi$ 
in the respective cluster structures, it thus follows from Theorem B that 
the set $\Phi^\low$ of lowest degree terms at the 
indicated points  contains  
a polynomial integrable system on the linearizations respectively given in \eqref{eq:g-pi0}, \eqref{eq:b-pi0}, 
\eqref{eq:nw-pi0}, and \eqref{eq:gl-pi0}. In each case and for some extended clusters $\Phi$, 
we describe explicit subsets of $\Phi^\low$ which are polynomial integrable 
systems on the linearizations (Theorem \ref{thm:choice-GB}, Theorem \ref{thm:psi}, Theorem \ref{thm:choice-G}, 
Theorem \ref{thm:choice-B}, and Theorem \ref{thm:choice-GLn}). 

While in this paper we only examine  polynomial integrable systems obtained from {\it some}
initial extended clusters, 
it is natural to ask how these integrable systems {\it mutate} as the extended clusters mutate, a project we will carry out in the future. Another interesting future project is to investigate the quantum analogs of the results in this paper. 

We also remark that the integrable systems on $({\mathfrak{gl}}(n, \CC), \; \pi_{0, {\mathfrak{gl}}(n, \CC)})$
that we obtain from the generalized cluster structure $\calC_{\rm GSV}$ in $\CC(\GL(n, \CC)^*)$ are in general different from the celebrated {Gelfand-Zeitlin integrable system} (see Remark \ref{rem:different}). 

\subsection{Signed generalized minors on \texorpdfstring{$\g$}{g}}\label{ss:signed-intro}
Let again $G$ be a connected and simply connected complex semi-simple Lie group with Lie algebra $\g$. 
Initial extended clusters constructed in \cite{BFZ:III} for 
the BFZ
cluster structure $\calC_{\rm BFZ}(G)$  consist of {\it generalized minors} $\Delta_{u\omega_\alpha, v\omega_\alpha}$ on $G$, where $u, v \in W$ and $\omega_\alpha$ is a fundamental weight (see $\S$\ref{ss:minors-G}).  
For such $u, v \in W$ and $\omega_\alpha$, we introduce 
\begin{equation}\label{eq:delta-Delta-intro}
\delta_{u\omega_\alpha, v\omega_\alpha} \stackrel{\rm def}{=} 
\Delta_{u\omega_\alpha, v\omega_\alpha}^\low \in \CC[\g],
\end{equation}
where the lowest degree term is taken  at the identity element $e \in G$. For $G = \SL(n, \CC)$, it is well-known that every $\Delta_{u\omega_\alpha, v\omega_\alpha}(g)$ 
is an ordinary minor of $g \in \SL(n, \CC)$, and we prove in Lemma \ref{lem:GLn-minor} that 
$\delta_{u\omega_\alpha, v\omega_\alpha}(x)$ is an ordinary minor of $x \in \sl(n, \CC)$ but 
possibly with a negative sign.
For an arbitrary $G$, we thus call 
$\delta_{u\omega_\alpha, v\omega_\alpha}$ a {\it signed generalized minor on $\g$}. Each $\delta_{u\omega_\alpha, v\omega_\alpha}$ is, by definition,  a 
homogeneous polynomial on $\g$.

Given the key role played by generalized minors in the BFZ cluster structures, 
we believe that signed generalized 
minors on $\g$ are interesting  on their own and warrant attention.  
We prove in 
Definition-Lemma  \ref{de-lem:delta} that $\delta_{u\omega_\alpha, v\omega_\alpha}$
coincides with the {\it co-degree term},  as defined by Kostant \cite[$\S$2.7]{K13},
of a (linear) matrix coefficient $\widehat{\Delta}_{u\omega_\alpha, v\omega_\alpha} \in (U\g)^*$ in the 
highest weight representation of $\g$ with highest weight $\omega_\alpha$, where $U\g$ is
the universal enveloping algebra of $\g$ 
 (see Definition-Lemma \ref{de-lem:co-deg}). In $\S$\ref{ss:signed-minor} we establish some 
basic properties of signed generalized minors which are parallel to some of the properties
of generalized minors on $G$ proved by Fomin and Zelevinsky in \cite{FZ:double}. 
The explicit polynomial integrable systems on 
\eqref{eq:g-pi0}, \eqref{eq:b-pi0} and 
\eqref{eq:nw-pi0}  coming from the extended clusters in $\calC_{\rm BFZ}(G)$, $\calC_{\rm BFZ}(B)$ and  
$\calC^w_{\rm st}$ defined by (double) reduced words of $w_0$ and  $w \in W$
all consist of signed generalized minors (except in the case when 
$w_0\neq -1$ on the root system of $\g$; see  
Theorem \ref{thm:choice-GB}, Theorem \ref{thm:psi}, Theorem \ref{thm:choice-G}, and 
Theorem \ref{thm:choice-B}).


We now give more details on our results in the case of Schubert cells.

\subsection{Schubert cells and applications to the nilpotent Lie algebra \texorpdfstring{$\n_{w, -}$}{nw-}}\label{ss:app-n-w}
For an $n$-dimensional Lie algebra $\l$, the cardinality of any polynomial integrable system on $(\l^*, \pi_{0, \l})$,  i.e., the integer 
\[
{\rm mag} (\l) \stackrel{\rm def}{=} n-\frac{1}{2} {\rm rk}(\pi_{0, \l}),
\]
is called the {\it magic number}
of $\l$ or of $(\l^*, \pi_{0, \l})$, and the integer 
\[
\ind(\l) \stackrel{\rm def}{=} n-{\rm rk}(\pi_{0, \l}) = 2\;{\rm mag}(\l)-n,
\]
is called the {\it index of $\l$} \cite[11.1.6]{Diximir:book}. To construct polynomial integrable systems on $(\l^*, \pi_{0, \l})$, one
needs to compute ${\rm mag}(\l)$, or, equivalently, $\ind(\l)$. There is an extensive literature on
computing $\ind(\l)$ for various $\l$ (see, for example, 
\cite{B-D:index, DK:index, Panyushev:index, RT:index} and references therein).
For $\l = \g_{\rm st}^*$, or $\b_-$,
or $\gl(n, \CC)$ that appear in the linear Poisson structures in \eqref{eq:g-pi0}, \eqref{eq:b-pi0}, and 
\eqref{eq:gl-pi0}, the indices $\ind(\l)$ are known or easy to compute.  Indeed, one has
(see Proposition
\ref{prop:ind-gsts}, Lemma \ref{lem:ind-bm}, and Lemma \ref{lem:rkBB-})
\[
{\ind}(\g_{\rm st}^*) = r=\dim (\b\cap \b_-), \hs 
{\ind}(\b_-) = \dim {\rm im} (1 + w_0), \hs \mbox{and} \hs
\ind(\gl(n, \CC)) = n.
\]
For the Lie algebra $\n_{w, -} = \n_- \cap {\rm Ad}_{\dot{w}} (\n)$ in \eqref{def:nw-nw} for $w \in W$, however, 
there does not seem to be any general result in the literature on $\ind(\n_{w, -})$ except for the case of $w = w_0$, for which $\n_{w_0, -} = \n_-$ and it is known from \cite{K12} (see Proposition \ref{prop:w0}) that 
$\ind(\n_-) = \dim \ker (1+w_0)$. In this paper we give a formula for  $\ind(\n_w)$ for arbitrary $w \in W$ using signed 
generalized minors. More precisely, for $w \in W$, let
${\rm supp}(w)$ be the set of all simple roots $\alpha$ such that $w\omega_\alpha \neq \omega_\alpha$, and let
\[
d_w = \sum_{\alpha \in {\rm supp}(w)} \deg (\delta_{\omega_\alpha, w\omega_\alpha}),
\]
where $\deg (\delta_{\omega_\alpha, w\omega_\alpha})$ is the homogeneous degree of 
$\delta_{\omega_\alpha, w\omega_\alpha}$.
We prove in Proposition \ref{prop:index-nw-0} that 
\[
{\rm mag}(\n_{w, -}) = d_w, \hs  
  \ind(\n_{w, -}) = 2d_w - \ell(w), \hs \mbox{and} \hs {\rm rk}(\pi_{0, \n_{w,-}}) = 2(\ell(w)-d_w)
\]
for every $w \in W$. We note (see Remark \ref{rem:nw-Casimir}) that the polynomials 
$\{\delta_{\omega_\alpha, w\omega_\alpha}|_{\n_w}: \alpha \in {\rm supp}(w)\}$ on $\n_w \cong (\n_{w, -})^*$
can be identified with the 
lowest degree terms at $w_\cdot B \in BwB/B$ of the frozen variables of the standard cluster structure
$\calC^w_{\rm st}$ on $BwB/B$. 

For $w \in W$ and any reduced word
${\bf w} = (s_{i_1}, \ldots, s_{i_l})$ of $w$, by taking the lowest degree terms of
the extended cluster in $\calC^w_{\rm st}$ associated to $\bfw$ (see $\S$\ref{ss:st-cluster-BwB}), 
we obtain the set
\begin{equation}\label{eq:delta-k-intro}
\{\delta_{\omega_{i_k}, s_{i_1}\cdots s_{i_k}\omega_{i_k}}|_{\n_w} \in \CC[\n_w] \colon k \in [1, l]\}
\end{equation}
of polynomials on $\n_w$ that pairwise Poisson commute with respect to $\pi_{0, \n_{w, -}}$. 
We also explain in $\S$\ref{ss:Thimm-flow} the set in \eqref{eq:delta-k-intro} in the context of Thimm's method
\cite{Guillemin-Sternberg:collective}. We prove in Proposition \ref{prop:quasi-poly-GB} that the Hamiltonian flow of every polynomial in 
\eqref{eq:delta-k-intro} is quasi-polynomial (Definition \ref{defn:quasi-poly}) and thus complete.  
In Theorem \ref{thm:choice-GB} we 
identify a subset $K_\bfw$ of $[1, l]$ of cardinality $d_w$ such that
\[
\{\delta_{\omega_{i_k}, s_{i_1}\cdots s_{i_k}\omega_{i_k}}|_{\n_w}: k \in K_\bfw\}
\]
is algebraically 
independent and thus a
polynomial integrable system on $(\n_w, \pi_{0, \n_{w, -}})$. We will further study these integrable systems, including their mutations and  quantum analogs, in  separate papers.

\subsection{Acknowledgments}
We would like to thank Anton Alekseev, Arkady Berenstein, Sam Evens, Misha Gekhtman, Victor Ginzburg, Joel Kamnitzer, Eckhard Meinrenken, Lenya Rybnikov, Dima Voloshyn, Yuancheng Xie, Xiaomeng Xu and Milen Yakimov for stimulating discussions at various stages of this project.
Yanpeng Li's research has been partially supported by the National Natural Science Foundation of China (No.12201438). Part of this work was done during Yu Li's stay at the Max Planck Institute for Mathematics and visits to the University of Hong Kong and Sichuan University, 
the hospitality of these institutions is gratefully acknowledged.
Jiang-Hua Lu's research has been partially supported by the Research Grants Council of the Hong Kong SAR, 
China (GRF 17306621 and 17306724).

\section{Integrable systems on linearization of Poisson manifolds}\label{s:int-linear}

\subsection{Integrable systems on Poisson manifolds}\label{ss:int-sys}
Let $P$ be a complex manifold and $TP$ its holomorphic tangent bundle.  
Recall \cite{LGPV:book} that a \textit{Poisson structure} on $P$ is a holomorphic section $\pi$ of $\wedge^2 TP$ such that the binary operation
$\{\varphi_1, \varphi_2\}_{\pi} = \pi(d \varphi_1, d \varphi_2)$
on the sheaf of  holomorphic functions on $P$ is a Lie bracket.   
The pair $(P, \pi)$ is referred to as a \textit{complex Poisson manifold} and $\pi$ its {\it Poisson bi-vector field}.
For $p \in P$, let
\[
\pi(p)^\#: \; T^*_pP \longrightarrow T_pP, \;\; ( \pi(p)^\#(\alpha), \, \beta ) = \pi(p)(\alpha, \beta), \hs \alpha, \beta \in T^*_pP.
\]
Then ${\rm im} (\pi(p)^\#) \subset T_pP$ is the tangent space at $p$ to the {\it symplectic leaf of $\pi$ through $p$}, and the \textit{rank} of 
$\pi$ at $p$ is defined as (the even integer)
${\rm rk} (\pi(p)) = \dim ({\rm im}(\pi(p)^\#))$. 
Define the {\it rank of $\pi$ in $P$} as
\[
{\rm rk}(\pi) = {\rm max}\{{\rm rk}(\pi(p)): p \in P\}.
\]

\begin{defn}\label{defn:indep-0}
Let $\Phi =\{\varphi_1, \ldots, \varphi_m\}$ be any finite set of holomorphic functions on $P$.

1) $\Phi$ is said to be \textit{independent} if the $m$-form 
$d \varphi_1 \wedge \cdots \wedge d \varphi_m$ is non-zero;

2) $\Phi$ is said to be {\it $\pi$-involutive} if
$\{\varphi_i, \varphi_j\}_{\pi} = 0$  for all $i,j \in [1,m]$. 
\end{defn}

\begin{lem}\label{lem:m-atmost}
When $P$ is connected, the cardinality of any set of independent and $\pi$-involutive holomorphic functions on $P$ is at most 
$n - \frac{1}{2} {\rm rk}(\pi)$, where $n = \dim P$.
\end{lem}

\begin{proof} Let $\Phi = \{\varphi_1, \ldots, \varphi_m\}$ be a set of independent and $\pi$-involutive holomorphic functions on $P$. As the $m$-form
$d \varphi_1 \wedge \cdots \wedge d \varphi_m$ is holomorphic, the subset  
${\mathcal{U}}_\Phi = \{p \in P: d_p \varphi_1 \wedge \cdots \wedge d_p \varphi_m\neq 0\}$
of $P$ is open and dense in $P$. The statement now follows from  \cite[Proposition 4.12]{AMVan:book}.
\end{proof}

\begin{defn}\label{defn:int}
Let $(P, \pi)$ be a connected complex Poisson manifold of dimension $n$.
The integer 
\[
m=n -\frac{1}{2} {\rm rk}(\pi)
\]
is called the \textit{magic number} of $(P, \pi)$, and a
set of $m$ independent and $\pi$-involutive holomorphic functions on $P$ is called an \textit{integrable system} on $(P, \pi)$. 
\end{defn}

We refer to  \cite{AMVan:book, LGPV:book} for more details on integrable systems. 

\begin{nota-defn}\label{nota:pi0-l}
{\rm
For a finite dimensional Lie algebra $\l$ over $\CC$, denote by $\pi_{0, \l}$ the linear 
Poisson structure, also called the {\it Kirillov-Kostant-Souriau Poisson structure}, on $\l^*$ given by
\[
\{x, y\}_{\pi_{0, \l}} = [x, y], \hs x, \, y \in \l, 
\]
where $x, y \in \l$ are 
regarded as linear functions on $\l^*$. 
For $\xi \in \l^*$, let 
\[
\Stab_{\l}(\xi) = \{x\in \l: (\xi, [x, y]) = 0,\, \forall \, y \in \l\}
\]
be the stabilizer Lie sub-algebra of $\l$ at $\xi$ for the co-adjoint representation of $\l$ on $\l^*$. The integer
\begin{equation}\label{eq:ind-l}
\ind (\l) = \min \{\dim \Stab_{\l}(\xi): \xi \in \l^*\}
\end{equation}
is called the {\it index} of $\l$ (see, for example, \cite[11.1.6]{Diximir:book}). Denote by ${\rm mag}(\l)$
the magic number of $(\l^*, \pi_{0, \l})$ and also call it the {\it magic number of the Lie algebra $\l$}. 
If $\dim \l = n$, then 
\begin{equation}\label{eq:pi0-l-all}
{\rm rk}(\pi_{0, \l}) = n - \ind(\l) \hs \mbox{and} \hs 
{\rm mag}(\l)=n-\frac{1}{2} {\rm rk} (\pi_{0, \l}) = \frac{1}{2}(n + \ind (\l)).
\end{equation}
By a polynomial integrable system on $(\l^*, \pi_{0, \l})$ we mean
an integrable system on $(\l^*, \pi_{0, \l})$ consisting of polynomial functions on $\l^*$. 
}
\end{nota-defn}

\subsection{Linearization of Poisson structures}\label{ss:linPoi}
Let $(P, \pi)$ be a complex Poisson manifold, and suppose 
that $p_0 \in P$ is such that $\pi(p_0) = 0$.  
Then the cotangent space $T^{\ast}_{p_0} P$ has the well-defined Lie bracket $[\, ,\, ]$
such that, for any local holomorphic functions $\vphi_1$ and $\vphi_2$ on $P$ defined near $p_0$, one has
\[
[d_{p_0}\vphi_1, \, d_{p_0}\vphi_2] = d_{p_0} \bigl( \{\varphi_1, \varphi_2\}_{\pi} \bigr).
\]

\begin{defn}\label{defn:linearization}
The induced linear Poisson structure on $T_{p_0}P$, denoted by $\pi_0$,  is
called the linearization of $\pi$ at $p_0$, and the  Poisson manifold $(T_{p_0} P, \pi_0)$ is called the {\it linearization of $(P, \pi)$ at $p_0$.}  
\end{defn}

\subsection{Lowest degree terms of tensor fields at \texorpdfstring{$p_0$}{p0}}\label{ss:phi-low}
Let $P$ be a complex manifold and let $p_0 \in P$.  Let $z = (z_1, \ldots, z_n)$ be a local holomorphic coordinate system on $P$ near $p_0$ with $z(p_0) = 0 \in \CC^n$. 
We write  $(v_1, \ldots, v_n) = (d_{p_0}z_1, \ldots, d_{p_0}z_n)$ 
and regard $(v_1, \ldots, v_n)$ as a linear coordinate system  on $T_{p_0} P$.  

Consider first a non-zero holomorphic function $\varphi$ on $P$ defined near $p_0$ with Taylor expansion
\[
\varphi (z_1, \ldots, z_n) = \sum \limits_{\epsilon \in (\ZZ_{\geq 0})^n} c_{\epsilon} z^{\epsilon}
\]
at $0$, where $c_{\epsilon} \in \CC$ and $z^{\epsilon}=z_1^{\epsilon_1} \ldots z_n^{\epsilon_n}$ for 
$\epsilon = (\epsilon_1, \ldots, \epsilon_n) \in (\ZZ_{\geq 0})^n$.  Let $e$ be the smallest integer such that there exists $\epsilon \in (\ZZ_{\geq 0})^n$ with
$c_{\epsilon} \neq 0$ and $\epsilon_1 + \cdots + \epsilon_n = e$.

\begin{defn}\label{defn:phi-low}
With the notation as above, the {\it lowest degree term of $\varphi$ at $p_0$} is the homogeneous polynomial function 
    $\vphi^\low$ on $T_{p_0} P$ given by
    \[
    \varphi^{\low}(v_1, \ldots, v_n) = \sum \limits_{\substack {\epsilon \in (\ZZ_{\geq 0})^n \\ \epsilon_1 + \cdots + \epsilon_n = e}} c_{\epsilon} v^{\epsilon}. 
    \]
\end{defn}

\begin{prop} \label{Pro:indep}
    The polynomial function $\varphi^{\low}$ on $T_{p_0} P$ is well-defined, i.e., it is independent of the choice of the local holomorphic coordinate system $(z_1, \ldots, z_n)$ on $P$ near $p_0$.
\end{prop}

\begin{proof}
    Suppose that $(z'_1, \ldots, z'_n)$ is another local holomorphic coordinate system on $P$ near $p_0$ with $z'(p_0) = 0$, and let
    $(v'_1, \ldots, v'_n) = (d_{p_0} z_1^\prime, \ldots, d_{p_0}z_n^\prime)$
    be the induced linear coordinate system on $T_{p_0} P$.  Set
   $J_{ji} =  \frac{\partial z_i}{\partial z_j^\prime}(0)$ for $i, j \in [1, n]$.
Then  $v_i = \sum \limits_{j=1}^n J_{ji} v'_j$ for $i \in [1,n]$, and one has
    \begin{align*}
       \vphi 
        = & \bigg( \sum \limits_{\substack{\epsilon \in (\ZZ_{\geq 0})^n \\ \epsilon_1 + \cdots + \epsilon_n = e}} c_{\epsilon} z^{\epsilon} \bigg) + O(z^{e+1}) =
         \bigg( \sum \limits_{\substack{\epsilon \in (\ZZ_{\geq 0})^n \\ \epsilon_1 + \cdots + \epsilon_n = e}} c_{\epsilon} z_1(z'_1, \ldots, z'_n)^{\epsilon_1} \cdots z_n(z'_1, \ldots, z'_n)^{\epsilon_n} \bigg) + O \bigl( (z')^{e+1} \bigr) \\
        = & \bigg( \sum \limits_{\substack{\epsilon \in (\ZZ_{\geq 0})^n \\ \epsilon_1 + \cdots + \epsilon_n = e}} c_{\epsilon} (J_{11} z'_1 + \cdots + J_{n1} z'_n)^{\epsilon_1} \cdots (J_{1n} z'_1 + \cdots + J_{nn} z'_n)^{\epsilon_n} \bigg) + O \bigl( (z')^{e+1} \bigr).
    \end{align*}
    Since the last line is the Taylor expansion of $\varphi$ at $0$ with respect to $(z'_1, \ldots, z'_n)$, it follows that
 \begin{align*}
    \varphi^{\low}(v'_1, \ldots, v'_n) &= \sum \limits_{\substack{\epsilon \in (\ZZ_{\geq 0})^n \\ \epsilon_1 + \cdots + \epsilon_n = e}} c_{\epsilon} (J_{11} v'_1 + \cdots + J_{n1} v'_n)^{\epsilon_1} \cdots (J_{1n} v'_1 + \cdots + J_{nn} v'_n)^{\epsilon_n}\\
     & = \sum \limits_{\substack{\epsilon \in (\ZZ_{\geq 0})^n \\ \epsilon_1 + \ldots + \epsilon_n = e}} c_{\epsilon} 
       v_1^{\epsilon_1} \cdots v_n^{\epsilon_n} =
     \varphi^{\low} (v_1, \ldots, v_n). \tag*{\qedhere}
\end{align*}
\end{proof}

The following fact is proved using the definitions, and we omit the proof.

\begin{lem}\label{lem:comp}
For any local holomorphic diffeomorphism $\sigma$ on $P$ such that $\sigma(p_0) = p_0$ and for any local holomorphic 
function $\vphi$ on $P$ defined near $p_0$, one has $(\vphi \circ \sigma)^\low = \vphi^\low\circ (d_{p_0}\sigma)$.
\end{lem}

\begin{rem}\label{rem:v-z}
If $z = (z_1, \ldots, z_n)$ is a holomorphic local coordinate system near $p_0 \in P$ with $z(p_0) = 0$, we will abuse notation and also denote $z_k^\low = d_{p_0} z_k$ by $z_k$ for $k \in [1, n]$, so that for any
holomorphic function $\vphi$ on $P$ defined near $p_0$, we have 
$\vphi^\low = \varphi^{\low}(z_1, \ldots, z_n)  \in \CC[z_1, \ldots, z_n]$.
\end{rem}

In the rest of the paper we will always make clear the point $p_0 \in P$ at which we take the lowest degree terms, so 
   for ease of notation we do not include $p_0$ in the symbol $\vphi^\low$.  The same convention will be applied when taking the lowest degree terms of tensor fields at $p_0$, which we now explain.

Consider a non-zero holomorphic $k$-form $\alpha$ on $P$ defined near $p_0$ with Taylor expansion 
\begin{equation}\label{eq:alpha}
\alpha 
= \sum \limits_{1 \leq i_1 < \cdots < i_k \leq n} \;\,
\sum \limits_{\epsilon \in (\ZZ_{\geq 0})^n} c_{i_1, \ldots, i_k, \epsilon} z^{\epsilon}\;
d z_{i_1} \wedge \cdots \wedge d z_{i_k}
\end{equation}
at $0$. Let $e$ be the smallest integer such that
there exist $1 \leq i_1 < \cdots < i_k \leq n$ and $\epsilon = (\epsilon_1, \ldots, \epsilon_n) \in (\ZZ_{\geq 0})^n$ with
$c_{i_1, \ldots, i_k, \epsilon} \neq 0$ and $\epsilon_1 + \cdots + \epsilon_n = e$.
Define the {\it lowest degree term of $\alpha$ at $p_0$}  to be
\[
\alpha^{\low} = \sum \limits_{1 \leq i_1 < \cdots < i_k \leq n} 
\;\,
\sum \limits_{\substack{\epsilon \in (\ZZ_{\geq 0})^n \\ \epsilon_1 + \cdots + \epsilon_n = e}} c_{i_1, \ldots, i_k, \epsilon} z^{\epsilon} \;d_{p_0} z_{i_1} \wedge \cdots \wedge d_{p_0} z_{i_k},
\]
where, as in Remark \ref{rem:v-z}, $(z_1, \ldots, z_n)$ is viewed as a linear coordinate system on $T_{p_0}P$,
and 
$\alpha^{\low}$ is regarded as a polynomial $k$-form on $T_{p_0} P$. An argument similar to that in
the proof of Proposition \ref{Pro:indep} shows that $\alpha^{\low}$ is independent of the choice of the coordinate system 
$(z_1, \ldots, z_n)$ near $p_0$ with $z(p_0) = 0$.
Define 
\begin{equation}\label{eq:deg-alpha-low}
\deg (\alpha^{\low}) = e+k.
\end{equation}

Similarly, for a
non-zero $k$-vector field $\partial$, i.e., a holomorphic section of $\wedge^k TP$, defined near $p_0$,  let
\[
\partial = 
 \sum \limits_{1 \leq i_1 < \cdots < i_k \leq n} \;\,
\sum \limits_{\epsilon \in (\ZZ_{\geq 0})^n} c_{i_1, \ldots, i_k, \epsilon} z^{\epsilon}
\frac{\partial}{\partial z_{i_1}} \wedge \cdots \wedge \frac{\partial}{\partial z_{i_k}}
\]
be the Taylor expansion of $\partial$ at $0$. Let $e$ be the smallest integer given in exactly the same way as for 
the $k$-form $\alpha$
in \eqref{eq:alpha}, and
define the {\it lowest degree term of $\partial$ at $p_0$}  to be
\[
\partial^{\low} = \sum \limits_{1 \leq i_1 < \cdots < i_k \leq n} 
\sum \limits_{\substack{\epsilon \in (\ZZ_{\geq 0})^n \\ \epsilon_1 + \cdots + \epsilon_n = e}} 
c_{i_1, \ldots, i_k, \epsilon} z^{\epsilon} \left(\left.\frac{\partial}{\partial z_{i_1}}\right|_{p_0}\right) \wedge \cdots \wedge \left(\left.\frac{\partial}{\partial z_{i_k}}\right|_{p_0}\right),
\]
where $\left.\frac{\partial}{\partial z_1}\right|_{p_0}, \ldots, \left.\frac{\partial}{\partial z_n}\right|_{p_0}$ 
are regarded as constant vector fields on $T_{p_0} P$, so $\partial^{\low}$ is a polynomial $k$-vector field on $T_{p_0} P$, 
which is again independent of the choice of the coordinate system $(z_1, \ldots, z_n)$ near $p_0$ with $z(p_0) = 0$.
Define the \textit{degree} of $\partial^{\low}$ to be $$\deg (\partial^{\low}) = e-k.$$

\begin{example} \label{Ex:lin&low}
    Assume that $\pi$ is a holomorphic Poisson structure on $P$ such that $\pi({p_0}) = 0$, and
    let $\pi_0$ be the linearization of $\pi$ at $p_0$ as in $\S$\ref{ss:linPoi}.  If $\pi_0 \neq 0$, then
    $\pi^{\low} = \pi_0$.  
\end{example}

Meromorphic tensor fields on $P$ also have well-defined lowest degree terms at $p_0$.
Consider first a non-zero meromorphic function $\varphi = a/b$ on $P$, where $a$ and $b$ are
local non-zero holomorphic functions  defined near $p_0$.
Define the {\it lowest degree term of $\varphi$ at $p_0$} to be the rational function on $T_{p_0} P$ given by
\[
\varphi^{\low} = \frac{a^{\low}}{b^{\low}}.
\]
 If $\varphi = c/d$ is another presentation of $\varphi$, then
 taking the lowest degree terms of both sides of $ad=bc$  gives $a^\low/b^\low = c^\low/d^\low$.
  Thus the rational function $\varphi^{\low} $ is independent of the presentation $\varphi =a/b$.

Similarly, let $\alpha$ (resp. $\partial$) be a non-zero meromorphic $k$-form (resp. $k$-vector field), expressed as
\begin{equation}\label{eq:alpha-partial}
\alpha = \frac{\alpha_1}{\varphi} \hs \left(\mbox{resp.} \; \partial = \frac{\partial_1}{\varphi}\right) 
\end{equation}
for some non-zero holomorphic $k$-form $\alpha_1$ (resp. $k$-vector field $\partial_1$)  and non-zero holomorphic function $\varphi$. Define the {\it lowest degree term of $\alpha$ (resp. $\partial$) at $p_0$} to be the rational 
$k$-form (resp. $k$-vector field)
\[
\alpha^{\low} = \frac{\alpha_1^{\low}}{\varphi^{\low}}  \hs 
\left(\text{resp. } \partial^{\low} = \frac{\partial_1^{\low}}{\varphi^{\low}}\right)
\]
 on $T_{p_0}P$. Same arguments as above for meromorphic functions show that 
$\alpha^{\low}$ (resp. $\partial^{\low}$)
is independent of the presentation of $\alpha$ (resp. $\partial$) in \eqref{eq:alpha-partial}.
Define the \textit{degree} of $\alpha^{\low}$ (resp. $\partial^{\low}$) to be
\[
\deg (\alpha^{\low}) = \deg (\alpha_1^{\low}) - \deg (\varphi^{\low}) \quad \quad (\text{resp. } 
\deg (\partial^{\low}) = \deg (\partial_1^{\low}) - \deg (\varphi^{\low})).
\]

\begin{example}\label{ex:dlog-degree}
For any non-constant local holomorphic function $\vphi$ defined near $p_0$, it follows from the definitions that
$\deg ((d\vphi)^\low) \geq \deg (\vphi^\low)$, so 
\[
\deg \left(\left(\frac{d\vphi}{\vphi}\right)^\low\right) = 
\deg \left(\frac{(d\vphi)^\low}{\vphi^\low}\right) \geq   0.   
\]
Moreover, the last inequality is an equality if and only if  $\vphi(p_0) = 0$.
\end{example}

\subsection{\texorpdfstring{$\pi_0$}{pi0}-involutivity}
In this section we assume that $(P, \pi)$ is a complex Poisson manifold and $p_0 \in P$ is such that $\pi(p_0) = 0$.  Let $(T_{p_0} P, \pi_0)$ be the linearization of $(P, \pi)$ at $p_0$.

\begin{lem} \label{Lem:mainlem}
    Let $\varphi, \psi$ be non-zero local holomorphic functions on $P$ defined near $p_0$. If 
   $\{\varphi, \psi\}_{\pi} =\lambda \varphi \psi$  for some $\lambda \in \CC$, then 
   $\{\varphi^{\low}, \psi^{\low}\}_{\pi_0} = 0$.  
\end{lem}

\begin{proof}
    We assume that $\pi_0 \neq 0$, for the conclusion is trivial otherwise.  
    Let $(z_1, \ldots, z_n)$ be a local holomorphic coordinate system on $P$ near $p_0$ with $z(p_0) = 0$. 
    It follows from $\pi^{\low} = \pi_0$  and $\{\varphi, \psi\}_{\pi} =\lambda \varphi \psi$ that 
\begin{equation}\label{eq:phi-phi}
\{\varphi^{\low}, \psi^{\low}\}_{\pi_0} + O \bigl( z^{\deg (\varphi^{\low}) + \deg (\psi^{\low})} \bigr)=
\lambda \varphi^{\low} \psi^{\low} + O \bigl( z^{\deg (\varphi^{\low}) + 
    \deg (\psi^{\low}) +1} \bigr).
\end{equation}
If $\{\varphi^{\low}, \psi^{\low}\}_{\pi_0} \neq 0$, then the lowest degree of the terms of  the left hand side
of \eqref{eq:phi-phi} is $\deg (\varphi^{\low}) + \deg (\psi^{\low})-1$, a contradiction. Thus 
   $\{\varphi^{\low}, \psi^{\low}\}_{\pi_0} = 0$.
\end{proof}

\begin{rem}\label{rem:higher}
The assumption that $\{\vphi, \psi\}_{\pi} = \lambda \vphi \psi$ in Lemma \ref{Lem:mainlem} can be relaxed in various ways.  For example, the same proof as that of Lemma \ref{Lem:mainlem} shows that one still has $\{\varphi^{\low}, \psi^{\low}\}_{\pi_0} =0$
if $\{\varphi, \psi\}_\pi = f(\varphi, \psi)$, where $f \in \CC[x, y]$ and is divisible by $xy$ in $\CC[x,y]$. 
\end{rem}

\begin{defn}\label{defn:system-0} 
1) A set $\Phi$ of holomorphic functions on $P$ is said to be {\it log-canonical} with respect to $\pi$ if 
$\{\vphi, \psi\}_\pi \in \CC \vphi \psi$ for all $\vphi, \psi \in \Phi$;

2) A set of $n = \dim P$ holomorphic functions on $P$ that is independent and 
log-canonical with respect to $\pi$ is called a {\it log-canonical system on $(P, \pi)$};

3) A log-canonical system on 
a neighborhood of $p_0$ is called a {\it log-canonical system on $(P, \pi)$ at $p_0$}.
\end{defn}

In view of Lemma \ref{Lem:mainlem}, every log-canonical system $\Phi$ on $(P, \pi)$ at $p_0$ 
 gives rise to
a $\pi_0$-involutive set $\Phi^\low = \{\vphi^\low: \vphi \in \Phi\}$ of homogeneous polynomials on $T_{p_0}P$. It is then natural to ask for the maximal
number of independent elements contained in $\Phi^\low$. We now address this question in $\S$\ref{ss:indep}.

\subsection{Independence}\label{ss:indep}
Let $P$ be a complex manifold of dimension $n$ and let $p_0 \in P$. In this section, the lowest degree terms of functions or forms 
will be taken at $p_0$.

Recall that a set $\Phi = \{\varphi_1, \ldots, \varphi_n\}$ of local holomorphic functions on $P$ near $p_0$
is said to be independent if 
$d \varphi_1 \wedge \cdots \wedge d \varphi_n\neq 0$.
For such a $\Phi$, introduce 
 the meromorphic $n$-form
\[
\mu_{\Phi} = \frac{d \varphi_1 \wedge \cdots \wedge d \varphi_n}{\varphi_1 \cdots \varphi_n} =
d(\log \vphi_1) \wedge \cdots \wedge d(\log \vphi_n)
\]
on $P$ near $p_0$, and we also call $\mu_\Phi$ the {\it log-volume form associated to $\Phi$}.  Recall by definition that
$$\deg (\mu_{\Phi}^{\low}) = \deg \bigl( (d \varphi_1 \wedge \cdots \wedge d \varphi_n)^{\low} \bigr) - \deg \bigl( (\varphi_1 \cdots \varphi_n)^{\low} \bigr).$$
By Example \ref{ex:dlog-degree}, $\deg (\mu_{\Phi}^{\low})\geq 0$, so we always have 
$n - \deg (\mu_{\Phi}^{\low}) \leq n$.

\begin{prop} \label{prop:F-low}
For any independent set $\Phi=\{\varphi_1, \ldots, \varphi_n\}$ of local holomorphic functions on $P$ near $p_0$, 
the set 
$\Phi^{\low}= \{\varphi_1^\low, \ldots, \varphi_n^\low\}$ contains at least $n - \deg (\mu_{\Phi}^{\low})$ independent polynomials on $T_{p_0}P$.
\end{prop}

\begin{proof}
 Set $n' = n - \deg (\mu_{\Phi}^{\low})$.   Let $z=(z_1, \ldots, z_n)$ be a local holomorphic coordinate system on $P$ near $p_0$ such that
    $z(p_0) = 0$. Consider the matrix $K = (K^1, \ldots, K^n)$ of polynomials in $z$, where 
\[
K^i = \left(\frac{\partial (\varphi_i^{\low})}{\partial z_1}, \ldots, \frac{\partial (\varphi_i^{\low})}{\partial z_n}\right)^T,
\hs i \in [1, n],
\]
and $T$ stands for transpose. 
We show that the rank of $K$ 
is at least $n'$. Note that 
\[
d\vphi_1 \wedge \cdots \wedge d\vphi_n = \det (J) \,dz_1 \wedge \cdots \wedge dz_n,
\]
where $J=(J^1, \ldots, J^n)$ with 
$J^i = \left(\frac{\partial \varphi_i}{\partial z_1}, \ldots, \frac{\partial \varphi_i}{\partial z_n}\right)^T$
for $i \in [1, n]$.
Set
  $d = \deg ((\det(J))^\low)$ and $d_i = 
\deg (\varphi_i^{\low})$ for $i \in [1, n]$. Then
$\deg (\mu_{\Phi}^{\low}) = n + d -(d_1+\cdots + d_n)$, so 
$n' =d_1+\cdots + d_n -d$. Writing
$J^i = K^i + L^i$ for $i \in [1, n]$, 
then for $c \in \CC^{\times}$, 
    \begin{align*}
\det(J)(cz) 
        &=  \det (K^1(cz) + L^1(cz), \ldots, K^n(cz) + L^n(cz)) \\
         & =\det (c^{d_1-1} K^1(z) + O(c^{d_1}), \ldots, 
         c^{d_n-1} K^n(z) + O(c^{d_n})) \\
         & =c^{-n+d_1+ \cdots + d_n} \det (K^1(z) + O(c), \ldots, K^n(z) + O(c)).
    \end{align*}
Suppose that the rank of $K$ is less than $n'$. Then all 
    $n' \times n'$ minors of $K$ are zero. It follows that every monomial term of the Taylor expansion of 
    $\det(J)(cz)$ in $c$ has degree at least 
\[
(-n+d_1+ \cdots + d_n) + (n-n'+1) = d+1,
\]
contradicting the definition of $d = \deg ((\det(J))^{\low})$.  Thus the rank of $K$ is at least $n'$.
\end{proof}

\begin{rem}\label{rem:no-meaning}
In Proposition \ref{prop:F-low}, the integer
$n - \deg (\mu_{\Phi}^{\low})$ may be negative, in which case Proposition \ref{prop:F-low} does not give
any information on the least number of independent elements in $\Phi^\low$.  
Even when $n - \deg (\mu_{\Phi}^{\low})$ is non-negative, 
the set $\Phi^{\low}$ may contain more than
$n - \deg (\mu_{\Phi}^{\low})$ independent elements. For example, consider
$\CC^2$ with $\Phi = \{z_1, 1+z_2^{k+1}\}$, where
$k$ is any non-negative integer. Then 
\[
\mu_\Phi = (k+1)\frac{z_2^k}{z_1(1+z_2^{k+1})} dz_1 \wedge dz_2.
\]
Taking lowest degree terms at $p_0 =(0, 0)$, we have  $\Phi^\low = \{z_1, 1\}$, while
$n - \deg (\mu_{\Phi}^{\low}) = 2-(k+1) = 1-k$.
Thus Proposition \ref{prop:F-low} is meaningful in this example only when $k = 0$.
\end{rem}

Let  now $\pi$ be a holomorphic Poisson structure on $P$, and assume that  $\pi({p_0}) = 0$.
Let $(T_{p_0} P, \pi_0)$ be the linearization of $(P, \pi)$ at $p_0$.

\begin{cor}\label{cor:I}
For every log-canonical system $\Phi$ on $(P, \pi)$ at $p_0$, one has
\[
\deg (\mu_\Phi^\low) \geq \frac{1}{2}{\rm rk}(\pi_0).
\]
Consequently,
$\deg (\mu_{\Phi}^\low) \leq \frac{1}{2}{\rm rk} (\pi_0)$ if and only if  
$\deg (\mu_{\Phi}^\low) = \frac{1}{2}{\rm rk} (\pi_0)$.
\end{cor}

\begin{proof}
By Lemma \ref{Lem:mainlem} and Proposition \ref{prop:F-low}, the set $\Phi^\low$ contains at least $n - \deg (\mu_\Phi^\low)$ 
independent  polynomials on $T_{p_0}P$ that are $\pi_0$-involutive. By Lemma
\ref{lem:m-atmost}, one has
$n - \deg (\mu_\Phi^\low) \leq n - \frac{1}{2}{\rm rk}(\pi_0)$, so
$\deg (\mu_\Phi^\low) \geq \frac{1}{2}{\rm rk}(\pi_0)$. The rest of Corollary \ref{cor:I} now follows. 
\end{proof}

\begin{defn}\label{defn:I} 
A log-canonical system $\Phi$ on $(P, \pi)$ at $p_0$ is said to have {\it Property $\mathcal I$} 
at $p_0$ if    
\[
    \deg (\mu_{\Phi}^\low) = \frac{1}{2}{\rm rk} (\pi_0).   
\]
\end{defn}

We can now prove the first main result of our paper.

\begin{thm} \label{thm:main-0}
    Let $\Phi = (\vphi_1, \ldots, \vphi_n)$ be a log-canonical system on $(P, \pi)$ at $p_0$.  If $\Phi$ has Property $\mathcal I$ at $p_0$, then $\Phi^{\low} = (\vphi_1^\low, \ldots, \vphi_n^\low)$ contains a polynomial integrable system on $(T_{p_0} P, \pi_0)$.
\end{thm}

\begin{proof}
 By Proposition \ref{prop:F-low}, Property $\calI$ of $\Phi$ implies that
$\Phi^\low$ contains 
$n - \deg (\mu_\Phi^\low) = n-\frac{1}{2}{\rm rk}(\pi_0)$
independent and $\pi_0$-involutive polynomials on $T_{p_0}P$, which form 
an integrable system on $(T_{p_0}P, \pi_0)$.
\end{proof}

\subsection{Integrable systems from compatible (generalized) cluster structures}\label{ss:cluster}
Following \cite{FZ:I, BFZ:III, GSV:2012, GSV:2018}, we now review some concepts on
cluster structures.

Let $P$ be an $n$-dimensional irreducible rational quasi-affine variety over $\CC$, and let $\CC[P] \subset \CC(P)$
be respectively the algebra of regular functions on $P$ and the field of rational functions on $P$.
A {\it seed in $\CC(P)$} is a triple  $\Sigma = (\Phi, {\rm ex}, M)$, where

1) $\Phi = (\vphi_1, \ldots, \vphi_n)$ is a free transcendental basis of $\CC(P)$ over $\CC$,

2) ${\rm ex}$ is a subset of $[1, n]$, called the {\it exchange set},  and

3) $M = (m_{i, j})$ is an integral matrix with rows indexed by $i \in [1, n]$ and columns by 
$j \in {\rm ex}$ and such that the submatrix 
$M_{{\rm ex} \times {\rm ex}} := (m_{i, j})_{i, j \in {\rm ex}}$ of $M$ is skew-symmetrizable \cite{BFZ:III}.

Given a seed $\Sigma = (\Phi, {\rm ex}, M)$ in $\CC(P)$, the $n$-tuple $\Phi$ is called 
the   {\it extended cluster} of $\Sigma$, each 
$\vphi_j$, for $j \in {\rm ex}$, is 
called a {\it cluster variable} of $\Sigma$, and each $\vphi_j$, for $j \in [1, n]\backslash {\rm ex}$, is
called a {\it frozen variable} of $\Sigma$. 
For $k \in {\rm ex}$, the 
{\it mutation of $\Sigma$ in the direction of $k$} is the seed 
\begin{equation}\label{eq:Sigma-mutation}
\mu_k(\Sigma) = \left(\Phi^\prime = (\vphi_1, \ldots, \vphi_{k-1}, \vphi_k^\prime, \vphi_{k+1}, \ldots, \vphi_n), \,
{\rm ex}, \,M^\prime=( m_{i, j}^\prime)\right),
\end{equation}
where 
\begin{align}\label{eq:phik}
    \vphi_k^\prime &= \frac{\prod \limits_{m_{j, k}>0} \vphi_j^{m_{j, k}} + \prod \limits_{m_{j, k}<0} \vphi_j^{-m_{j, k}}}{\vphi_k},  \\ 
    m_{i,j}^\prime &= \begin{cases} -m_{i,j}, & \hs i = k, \; \mbox{or} \; j = k, \vspace{.05in}\\
m_{i,j} + \frac{1}{2} (|m_{i,k}| m_{k, j}+m_{i,k}|m_{k, j}|), & \hs \mbox{otherwise}.\end{cases}\nonumber
\end{align}
Note that, as $m_{k, k} = 0$,  $\vphi_k \vphi_k^\prime$ is equal to the sum of two monomials in $\Phi\backslash \{\vphi_k\} =
\Phi^\prime \backslash \{\vphi_k^\prime\}$. Two seeds in $\CC(P)$ related by a sequence of mutations are said to be {\it mutation equivalent}.

In \cite{GSV:2012, GSV:2018}, the authors introduced {\it generalized mutation rule} of seeds in $\CC(P)$,
for which a one-step mutation in the direction $k \in {\rm ex}$ still replaces $\Phi$ with $\Phi^\prime$ 
 by only changing the variable $\vphi_k\in \Phi$ to a new variable $\vphi_k^\prime \in \Phi^\prime$ as in 
 \eqref{eq:Sigma-mutation}.  We do not need the detailed formula for $\vphi_k^\prime$, and we only need to note that 
in the generalized seed mutation rule, the product $\vphi_k \vphi_k^\prime$ may now be a sum of {\it more than two} monomials in $\Phi\backslash\{\vphi_k\}=\Phi^\prime\backslash\{\vphi_k^\prime\}$. See \cite[$\S$2.1]{GSV:2018} for details.

Although the generalized seed mutation rule covers the 
(ordinary) mutation rule in \eqref{eq:phik} as special cases, 
one may want to distinguish the two especially in examples. 
In the rest of $\S$\ref{ss:cluster}, we put the prefix ``generalized'' in 
parentheses in the definitions and statements to indicate that they are valid for the generalized 
mutation rule of seeds but the word ``generalized'' should be dropped when the mutation rule is the ordinary one.

\begin{defn} 
By a {\it (generalized) cluster
structure in $\CC(P)$} we mean a (generalized) mutation equivalence class of seeds in $\CC(P)$. For such a 
$\calC$, the extended cluster of any 
seed in $\calC$ is called an {\it extended cluster of $\calC$}, a  cluster variable in any seed in $\calC$ is called a 
{\it cluster variable of $\calC$}, and the common set of 
frozen variables in all the seeds in $\calC$, denoted as ${\rm Froz}(\calC)$,
is called the set of {\it frozen variables of $\calC$}. 
\end{defn}

Let $\calC$ be any (generalized) cluster
structure in $\CC(P)$. Associated to each 
extended cluster $\Phi = (\vphi_1, \ldots, \vphi_n)$ of $\calC$, we then have the log-volume form 
\[
\mu_\Phi = \frac{d\vphi_1 \wedge \cdots \wedge d\vphi_n}{\vphi_1 \cdots \vphi_n} = d(\log \vphi_1) \wedge \cdots \wedge 
d(\log \vphi_n).
\]
We now have the following key observation.

\begin{lem-de}\label{lem:same-volume}
For any irreducible rational quasi-affine variety $P$ over $\CC$, and for any (generalized) cluster
structure $\calC$ in $\CC(P)$, all the extended clusters of $\calC$ have, 
up to a sign, the same log-volume form, which will be denoted by $\mu_\calC$ and called the {\it log-volume form of $\calC$}.
\end{lem-de}

\begin{proof}
We only need to consider the one-step mutation 
\[
  \Phi = (\vphi_1,  \ldots, \vphi_n) \rightsquigarrow 
  \Phi^\prime = (\vphi_1, \ldots, \vphi_{k-1}, \vphi_k^\prime, \vphi_{k+1}, \ldots, \vphi_n) 
\]
for any $k \in {\rm ex}$.  
Let  $\alpha = d(\log \vphi_1) \wedge \cdots d(\log \vphi_{k-1}) \wedge d(\log \vphi_{k+1}) \wedge \cdots \wedge d(\log \vphi_n)$. 
Since
$\vphi_k \vphi_k^\prime$ is a sum of  monomials in $\Phi \setminus \{\vphi_k\}$,  we have
\[
 \mu_{\Phi^\prime}  =(-1)^{n-k}  \alpha \wedge d(\log \vphi_k^\prime) =(-1)^{n-k} 
 \alpha \wedge d(-\log \vphi_k )\\
 =   -\mu_{\Phi}.  \tag*{\qedhere}
\]
\end{proof} 


Not all (generalized) cluster structures that we will consider later in this paper satisfy Property $\mathcal{I}$.  It is sometimes helpful to modify the original extended cluster $\Phi$ to form a new system $\overline{\Phi}$ for which the degree of the log-volume form is reduced: $\deg(\mu_{\overline{\Phi}}) < \deg(\mu_{\Phi})$.  This motivates the following

\begin{defn}\label{defn:modify}
Let $\calC$ be a (generalized) cluster
structure in $\CC(P)$, and let $\CC({\rm Froz}(\calC))$ be the subfield of $\CC(P)$ generated by 
${\rm Froz}(\calC)$. By a {\it frozen variable modification of $\calC$} we mean an independent subset 
$\overline{{\rm Froz}(\calC)}$ of $\CC({\rm Froz}(\calC))$ 
 that has the same cardinality as ${\rm Froz}(\calC)$. 
The {\it modification} of an extended cluster 
$\Phi$  of $\calC$ by $\overline{{\rm Froz}(\calC)}$ is defined as
\begin{equation}\label{eq:overline-Phi}
\overline{\Phi} = (\Phi \backslash {\rm Froz}(\calC)) \sqcup \overline{{\rm Froz}(\calC)}.
\end{equation}
\end{defn}

\begin{lem-de}\label{lem:mu-modify} For any (generalized) cluster structure $\calC$ in $\CC(P)$
and for any frozen variable modification
$\overline{{\rm Froz}(\calC)}$ of $\calC$, the meromorphic form $\overline{\mu}_{\calC}$ on $P$ given by
\begin{equation}\label{eq:mu-C-modify}
\overline{\mu}_{\calC} = \mu_{\overline{\Phi}},
\end{equation}
where $\Phi$ is any extended cluster of $\calC$ and $\overline{\Phi}$ is given in \eqref{eq:overline-Phi},
is, up to a sign, independent of $\Phi$.  We call $\overline{\mu}_{\calC}$ the {\it modified log-volume form of $\calC$
by $\overline{{\rm Froz}(\calC)}$}.
\end{lem-de}

\begin{proof}
Let ${\rm Froz}(\calC) = \{f_1, \ldots, f_m\}$ and $\overline{{\rm Froz}(\calC)} = 
\{\overline{f}_1, \ldots, \overline{f}_m\}$. Then there exists a non-zero $F \in \CC(P)$ such that 
$d\overline{f}_1 \wedge \cdots \wedge d\overline{f}_m = Fdf_1 \wedge \cdots \wedge df_m$. Thus $\mu_{\overline{\Phi}} = F \mu_{{\Phi}}$ for any extended cluster $\Phi$ of $\calC$. The statement now follows from 
Lemma-Definition \ref{lem:same-volume}.
\end{proof}

Following \cite{GSV:2012, GSV:2018}, we now state some (Poisson geometric) properties of cluster structures needed for the considerations in this paper.
Recall first that if $(P, \pi)$ is an irreducible rational quasi-affine Poisson variety, the Poisson bracket on 
$\CC[P]$ extends to one on $\CC(P)$. A set of elements of $\CC(P)$ is said to be log-canonical with respect to 
$\pi$ if its elements have pairwise log-canonical Poisson brackets with respect to $\pi$.

\begin{defn}\label{defn:regular} 
Let $P$ be a smooth rational quasi-affine variety over $\CC$. 

1) A (generalized) cluster structure
$\calC$ in $\CC(P)$ is said to be {\it regular on $P$} if all cluster variables and frozen variables of $\calC$ are regular functions on $P$; A frozen variable modification $\overline{{\rm Froz}(\calC)}$ of $\calC$ is said to be {\it regular} if 
$\overline{{\rm Froz}(\calC)} \subset \CC[P]$;

2) Given a Poisson structure $\pi$ on $P$, a (generalized) cluster structure
$\calC$  in $\CC(P)$ is said to be {\it compatible with  $\pi$} if every extended cluster of $\calC$ is log-canonical with respect to $\pi$; A frozen variable modification $\overline{{\rm Froz}(\calC)}$ of $\calC$ is said to be
{\it $\pi$-compatible} if every $\overline{\vphi} \in \overline{{\rm Froz}(\calC)}$ is of the form
\begin{equation}\label{eq:cjMj}
\overline{\vphi} = c_{\overline{\vphi}}M_{\overline{\vphi}}
\end{equation}
for some Casimir function $c_{\overline{\vphi}} \in \CC(P)$ with respect to $\pi$ and some monomial 
$M_{\overline{\vphi}}$ in ${\rm Froz}(\calC)$. 
\end{defn}

We can now prove our main result on polynomial integrable systems and cluster structures.

\begin{thm}\label{thm:main-cluster} Let $(P, \pi)$ be a smooth rational quasi-affine Poisson variety over $\CC$. 
Let $\calC$ be a (generalized) cluster structure 
in $\CC(P)$ that is regular on $P$ and
compatible with $\pi$, and let $\overline{{\rm Froz}(\calC)}$ be 
a regular $\pi$-compatible frozen variable modification of $\calC$. 

1) The modification $\overline{\Phi}$ in \eqref{eq:overline-Phi} of every extended cluster $\Phi$ of $\calC$ 
is a log-canonical system on $(P, \pi)$.

2) Suppose that $p_0\in P$ is such that $\pi(p_0) = 0$, and let
$\pi_0$ be the linearization of $\pi$ at $p_0$. If
\begin{equation}\label{eq:deg-mu-C-1}
\deg \left((\overline{\mu}_{\calC})^\low\right) = \frac{1}{2} {\rm rk}(\pi_0),
\end{equation}
where $\overline{\mu}_{\calC}$ is the modified log-volume form of $\calC$ by $\overline{{\rm Froz}(\calC)}$ and
$(\overline{\mu}_{\calC})^\low$ is taken at $p_0$,
then for every extended cluster $\Phi$ of $\calC$, the modification $\overline{\Phi}$ of 
$\Phi$ 
has Property $\calI$ at $p_0$, so the set
$(\overline{\Phi})^{\low}$ of lowest degree terms of elements in $\overline{\Phi}$ at $p_0$
contains a polynomial integrable system on $(T_{p_0}P, \pi_0)$.
\end{thm}

\begin{proof}
1) Let $\Phi$ be any extended cluster of $\calC$. As $\Phi \subset \CC[P]$ and 
$\overline{{\rm Froz}(\calC)} \subset \CC[P]$, we
have $\overline{\Phi} \subset \CC[P]$. The condition in \eqref{eq:cjMj} on each 
$\overline{\vphi} \in \overline{{\rm Froz}(\calC)}$ implies that $\overline{\Phi}$ is log-canonical with respect to $\pi$.

2) is an immediate consequence of  Lemma-Definition \ref{lem:mu-modify} and Theorem \ref{thm:main-0}.
\end{proof}

\begin{defn}\label{defn:C-I}
In the setting of Theorem \ref{thm:main-cluster}, when \eqref{eq:deg-mu-C-1} holds, we say that  {\it $\calC$ has Property $\calI$ at $p_0$ after the frozen variable modification $\overline{{\rm Froz}(\calC)}$}.

\end{defn}

Given a smooth rational quasi-affine Poisson variety $(P, \pi)$ and a cluster structure $\calC$ that is regular on $P$ and 
compatible with $\pi$, Theorem \ref{thm:main-cluster} then says that Property $\calI$ at $p_0$ is a property of  the pair
$(\calC, \overline{{\rm Froz}(\calC)})$ rather than one of its individual extended clusters. In particular, 
one (modified) extended cluster of $\calC$ has Property $\calI$ at $p_0$ if and only if {\it every} (modified) extended cluster of $\calC$ has that property. 
We also say that 
{\it $\calC$ has Property $\calI$ at $p_0$} if it does after some regular $\pi$-compatible
frozen variable modification. In $\S$\ref{s:GB}-$\S$\ref{sect:GLndual}, we will give four examples of (generalized) cluster structures from Lie theory 
that have Property $\calI$ at certain points of interest on the respective varieties.

\begin{rem}\label{rem:weaker}
Let again $P$ be an irreducible rational quasi-affine variety over $\CC$. To accurately quote the literature, we follow 
\cite{GSV:book, GSV:2018} and clarify the notion of (generalized) cluster and upper cluster structures {\it
on the variety $P$}. 
Let $\calC$ be a (generalized) cluster structure  in $\CC(P)$, and 
let ${\rm inv}\subset {\rm Froz}(\calC)$.

1) The
{\it cluster algebra associated to $\calC$ and ${\rm inv}$} is the sub-algebra $\calA(\calC; {\rm inv})$ of $\CC(P)$
 generated over $\CC$ by all the cluster variables of $\calC$, all the frozen variables of $\calC$, and 
all $\vphi^{-1}$ for $\vphi \in {\rm inv}$; The
{\it upper cluster algebra associated to $\calC$ and ${\rm inv}$} is 
the intersection
\[
{\calU}(\calC; {\rm inv}) = \bigcap_{\Sigma \in \calC} L(\Sigma; {\rm inv}),
\]
where $L(\Sigma; {\rm inv}) = 
\CC[\vphi_1, \ldots, \vphi_n][\vphi_i^{- 1}: i \in {\rm ex}][\vphi^{-1}: \vphi \in {\rm inv}]$
for a seed $\Sigma \in \calC$ with extended cluster  $(\vphi_1, \ldots, \vphi_n)$.
By the strong Laurent phenomenon of cluster algebras of geometric type \cite[Proposition 11.2]{FZ:II} (and
\cite{Chekhov-Shapiro} for the generalized case), one always has
\[
\calA(\calC; {\rm inv}) \subset \calU(\calC; {\rm inv}).
\]

When ${\rm Froz}(\calC) \subset \CC[P]$, let
${\rm inv}(P)$ be
the set of all $\vphi \in {\rm Froz}(\calC)$ that vanish nowhere on $P$.

2) A {\it (generalized) cluster structure on $P$} is a (generalized) cluster structure $\calC$  in 
$\CC(P)$ such that ${\rm Froz}(\calC) \subset \CC[P]$ and 
$\calA(\calC, {\rm inv}(P)) = \CC[P]$;
A {\it (generalized) upper cluster structure on $P$} is a (generalized) cluster structure $\calC$
in $\CC(P)$ such that ${\rm Froz}(\calC) \subset \CC[P]$ and 
$\calU(\calC, {\rm inv}(P)) = \CC[P]$.

3) A (generalized) cluster  or upper cluster structure on $P$ is in particular
regular on $P$.
On the other hand, a free generating set $\Phi$ of $\CC(P)$ over $\CC$ consisting of regular functions on $P$
defines a seed $\Sigma$ in $\CC(P)$ with exchange set ${\rm ex} = \emptyset$, so the (generalized) 
mutation equivalence class of 
 $\Sigma$, composed of $\Sigma$ only, is regular on $P$ but may not be a (generalized) cluster structure on $P$.
\end{rem}

\subsection{\texorpdfstring{$\TT$}{T}-Pfaffians and polynomial integrable systems}\label{ss:Pfaffian}
In this section, we look at another class of Poisson manifolds on which Property $\calI$ of certain families of
log-canonical systems may be checked simultaneously. 

Let $\TT$ be a complex torus. By a {\it $\TT$-Poisson manifold} we mean a 
complex Poisson manifold  $(P, \pi)$ with a holomorphic $\TT$-action by Poisson automorphisms.

\begin{defn}\label{defn:TT-system}
Let $(P, \pi)$ be an $n$ dimensional $\TT$-Poisson manifold.  By a {\it $\TT$-log-canonical system} on $(P, \pi)$ we mean a 
log-canonical system 
$\Phi = \{\varphi_1, \ldots, \varphi_n\}$ 
on $(P, \pi)$ (recall Definition \ref{defn:system-0}) such that every $\vphi_i$ is a $\TT$-weight vector for the induced $\TT$-action on the algebra of holomorphic functions on $P$. 
\end{defn}

\begin{rem}\label{rem:T-comp}
Suppose that $(P, \pi)$ is a smooth rational quasi-affine $\TT$-Poisson variety with  a 
cluster structure $\calC$ in $\CC(P)$ that is regular on $P$ and compatible with $\pi$. If $\calC$ is 
also compatible with the $\TT$-action in the sense that every extended cluster of $\calC$ consists of $\TT$-weight vectors, 
then every extended cluster of $\calC$ is a $\TT$-log-canonical system on $(P, \pi)$. We will look at such examples in 
$\S$\ref{s:GB} and $\S$\ref{s:G}.
\end{rem}

Let $(P, \pi)$ be a $\TT$-Poisson manifold. 
Let $\t$ be the Lie algebra of $\TT$, and let $\sigma: \t \to \mathfrak{X}^1(P)$ be the induced action of $\t$ on $P$,
where $\mathfrak X^k(P)$ for $1 \leq k \leq \dim P$ is the space of holomorphic $k$-vector fields on $P$.
For $p \in P$, we also write
\[
\sigma_p: \; \t \lrw T_p P, \;\; \sigma_p(x) = \sigma(x)(p).
\]
As shown in \cite{Lu:Lie-Duke}, one has the Lie algebroid $A_\sigma = (P \times \t) \bowtie T^{\ast} P$ over $P$, which has $(P \times \t) \oplus T^{\ast} P$ as the underlying vector bundle and $-\sigma + \pi^\#$ as the anchor map.  Following \cite[$\S$2.1]{LM:T-leaves}, the {\it orbits} of the Lie algebroid $A_\sigma$ in $P$ are called the {\it $\TT$-orbits of symplectic leaves}, or simply {\it $\TT$-leaves} of $(P, \pi)$.  By definition, the tangent space to the $\TT$-leaf of $(P, \pi)$ through $p \in P$ is ${\rm im}(\sigma_p) +{\rm im}(\pi(p)^\#)$.

Let $L \subset P$ be a $\TT$-leaf of $(P, \pi)$. For any symplectic leaf $S$ of $(P, \pi)$ through some point in $L$, it is shown in \cite[Lemma 2.3]{LM:T-leaves} that $S \subset L$ and that the map
$\TT \times S \rightarrow L, (t, p) \mapsto t \cdot p$,
is a surjective submersion.  In particular,  $\pi$ has constant rank in $L$. By \cite[Lemma 2.5]{LM:T-leaves}, the {\it leaf stabilizer} of $\t$ at $p \in L$, defined as 
$\t_p = \{x \in \t: \sigma_p(x) \in {\rm im}(\pi(p)^\#)\}$,
depends only on $L$ and not on $p \in L$.  Set $\t_L = \t_p$ for any $p \in L$.  By the definition of $\t_L$, one has the vector space isomorphism
\[
\t/\t_L \lrw T_p L/{\rm im}(\pi(p)^\#), \;\; x+\t_L \Maps \sigma_p(x) + {\rm im}(\pi(p)^\#)
\]
for any $p \in L$. It then follows that $\dim \t - \dim \t_L = \dim L -{\rm rk}(\pi|_L)$.

\begin{lem}\label{lem:rk-L-P}
For any connected $\TT$-Poisson manifold $(P, \pi)$ with an open dense $\TT$-leaf $L$, one has
\[
{\rm rk}(\pi|_L) = {\rm rk}(\pi).
\]
\end{lem}

\begin{proof} Let $2r ={\rm rk}(\pi|_L)$.  Then 
$\wedge^r \pi \neq 0 \in {\mathfrak{X}}^{2r}(P)$ and  
$\wedge^{r+1}\pi = 0 \in {\mathfrak{X}}^{2r+2}(P)$.
Thus, $2r= {\rm rk}(\pi)$.
\end{proof}

Let $(P, \pi)$ be a connected $\TT$-Poisson manifold with an open dense $\TT$-leaf $L$, and let 
$2r = {\rm rk}(\pi)$.  Let $\t_L^\prime$ be any subspace of $\t$ complementary to $\t_L$
and $x_1, \ldots, x_{n-2r}$ any basis of $\t_L^\prime$, where $n = \dim P$. Then 
\begin{equation}\label{eq:eta-L}
{\rm Pf}_\TT(\pi) = \wedge^r \pi \wedge \sigma(x_1) \wedge \cdots \wedge \sigma(x_{n-2r}) \in {\mathfrak{X}}^n(P)
\end{equation}
is an anti-canonical section of $P$ that is nowhere vanishing on $L$ and is, up to non-zero scalar multiples, independent of the choices of $\t_L^\prime$ and of the basis $x_1, \ldots, x_{n-2r}$ of $\t_L^\prime$. 

\begin{defn} 
Assuming that $P$ is connected and
$(P, \pi)$ has an open dense $\TT$-leaf, the anti-canonical section ${\rm Pf}_\TT(\pi)$ of $P$ in \eqref{eq:eta-L} is called a {\it $\TT$-Pfaffian of the Poisson structure $\pi$ on $P$}, which is unique up to non-zero scalar multiples (see \cite[Remark 2.7]{LM:T-leaves}).  
\end{defn}

\begin{lem}\label{lem:mu-Pfaff} Let $(P, \pi)$ be a connected $\TT$-Poisson manifold with an open dense $\TT$-leaf.
For any $\TT$-log-canonical system $\Phi = \{\varphi_1, \ldots, \varphi_n\}$ on $(P, \pi)$, one has
\begin{equation}\label{eq:Pf-Phi}
\langle {\rm Pf}_\TT(\pi), \; d\vphi_1 \wedge \cdots \wedge d\vphi_n \rangle = c \,\vphi_1 \cdots \vphi_n
\end{equation}
for some $c \in \CC^\times$. Consequently, for any $p_0 \in P$, taking lowest degree terms at $p_0$, one has
\[
    \deg (\mu_{\Phi}^\low) = - \deg \bigl( {\rm Pf}_\TT(\pi)^\low \bigr).
 \]
\end{lem}

\begin{proof}
Let $U$ be an open subset of $P$ on which $(\vphi_1, \ldots, \vphi_n)$ is a local holomorphic coordinate system
and $\vphi_1 \vphi_2\cdots \vphi_n$ is nowhere zero.  
Consider the holomorphic vector fields on $U$ given by
$\partial_{\vphi_i} = \vphi_i \frac{\partial}{\partial \vphi_i}$ for $i \in [1, n]$.
Let $\{\varphi_i, \varphi_j\}_{\pi} = \lambda_{i,j} \varphi_i \varphi_j$, where $\lambda_{i,j} \in \CC$
for $i, j \in [1, n]$. Then on $U$ one has 
\[
\pi = \sum_{1 \leq i < j \leq n} \lambda_{i,j} \partial_{\vphi_i} \wedge \partial_{\vphi_j}, 
\]
and for every $x \in \t$ one has 
\[
\sigma(x) = \sum_{i=1}^n \sigma(x)(\vphi_i) \frac{\partial}{\partial \vphi_i} =
\sum_{i=1}^n \chi_{\vphi_i}(x) \vphi_i \frac{\partial}{\partial \vphi_i} = 
\sum_{i=1}^n \chi_{\vphi_i}(x) \partial_{\vphi_i},
\]
where $\chi_{\vphi_i}$ is the $\TT$-weight of $\vphi_i$ regarded as an element of $\t^{\ast}$.  Thus
there exists $c \in \CC^{\times}$ such that
\[
{\rm Pf}_\TT(\pi) = c \,\partial_{\vphi_1} \wedge \cdots \wedge \partial_{\vphi_n} = c \,\vphi_1\cdots \vphi_n 
\frac{\partial}{\partial \vphi_1} \wedge \cdots \wedge \frac{\partial}{\partial \vphi_n}
\]
on $U$,  so  \eqref{eq:Pf-Phi} holds on $U$. As $P$ is connected, 
\eqref{eq:Pf-Phi} holds everywhere on $P$.
Let now $p_0 \in P$, and let 
$z =(z_1, \ldots, z_n)$ be a local holomorphic coordinate system on $P$ near $p_0$ with $z(p_0) = 0$. Write
    \[
    {\rm Pf}_{\TT}(\pi) = P(z) \frac{\partial}{\partial z_1} \wedge \cdots \wedge \frac{\partial}{\partial z_n}
    \hs \mbox{and} \hs J(z) = \det \bigg( \frac{\partial(\varphi_1, \ldots, \varphi_n)}{\partial(z_1, \ldots, z_n)} \bigg).
\]
Then $P J= c\, \varphi_1 \cdots \varphi_n$ by \eqref{eq:Pf-Phi}.
Taking the lowest degree terms at $p_0$ gives
$P^{\low} J^{\low} = c \,(\varphi_1 \cdots \varphi_n)^{\low}$.
It then follows from 
$\mu_{\Phi} = (\varphi_1 \cdots \varphi_n)^{-1}d \varphi_1 \wedge \cdots \wedge d \varphi_n = J(z)
(\varphi_1 \cdots \varphi_n)^{-1}dz_1 \wedge \cdots\wedge dz_n$ that
\[
        \deg (\mu_{\Phi}^\low) =  \deg (J^\low) + n - \deg \bigl( (\varphi_1 \cdots \varphi_n)^{\low} \bigr) =
        -\deg (P^\low) + n =
 - \deg \bigl( {\rm Pf}_{\TT}(\pi)^{\low} \bigr). \tag*{\qedhere}
 \]
\end{proof}

We now have the following immediate consequence of Theorem \ref{thm:main-0} and Lemma \ref{lem:mu-Pfaff}.

\begin{thm}\label{thm:mu-Pfaff}
Suppose $(P, \pi)$ is a connected $\TT$-Poisson manifold with an open dense $\TT$-leaf. Let
$p_0 \in P$ be such that $\pi(p_0) = 0$, and let $\pi_0$ be the linearization of $\pi$ at $p_0$. If 
\begin{equation}\label{eq:Pfaff-I}
\deg \bigl( {\rm Pf}_{\TT}(\pi)^{\low} \bigr) = -\frac{1}{2} {\rm rk}(\pi_0),
\end{equation}
where ${\rm Pf}_{\TT}(\pi)^{\low}$ is taken at $p_0$, then every $\TT$-log-canonical system 
has Property $\calI$ at $p_0$.
\end{thm}

We now describe a setting in which \eqref{eq:Pfaff-I} always holds.

\begin{prop}\label{prop:Pfaff-atleast}
Let $(P, \pi)$ be a connected $\TT$-Poisson manifold with an open dense $\TT$-leaf, and let $p_0 \in P$ be such that
$\pi(p_0) = 0$.  Let $\pi_0$ be the linearization of $\pi$ at $p_0$ and let ${\rm Pf}_{\TT}(\pi)^{\low}$
be taken at $p_0$.

1) If $(P, \pi)$ admits a $\TT$-log-canonical system, then 
$\deg \bigl( {\rm Pf}_{\TT}(\pi)^{\low} \bigr) \leq  -\frac{1}{2} {\rm rk}(\pi_0)$;

2) If $p_0 \in P$ is a $\TT$-fixed point, then 
$\deg \bigl( {\rm Pf}_{\TT}(\pi)^{\low} \bigr) \geq -\frac{1}{2} {\rm rk}(\pi_0)$.
\end{prop}

\begin{proof}
1) Let $\Phi$ be any $\TT$-log-canonical system on $(P, \pi)$. By Corollary \ref{cor:I} and Lemma \ref{lem:mu-Pfaff},
\[
\deg \bigl( {\rm Pf}_{\TT}(\pi)^{\low} \bigr) = -\deg \left(\mu_\Phi^\low\right) \leq  -\frac{1}{2} {\rm rk}(\pi_0). 
\]

2) Let $r = \frac{1}{2}{\rm rk}(\pi)$, and let $x_1, \ldots, x_{n-2r} \in \t$ be such that
${\rm Pf}_\TT(\pi) = \wedge^r \pi \wedge \sigma(x_1) \wedge \cdots \wedge \sigma(x_{n-2r})$
as in \eqref{eq:eta-L}. 
Let $z = (z_1, \ldots, z_n)$ be any local holomorphic coordinate system near $p_0$ with $z(p_0) = 0$. 
Since $p_0$ is a $\TT$-fixed point, 
we have $\sigma(x_j)(p_0) = 0$ and thus 
\begin{equation}\label{eq:sigma-x}
\deg (\sigma(x_j)^\low) \geq 0, \hs j \in [1, n-2r].
\end{equation}
Let $r_0 = \frac{1}{2}{\rm rk}(\pi_0)$. If $\pi_0 = 0$, then  
either $\pi = 0$, or $\pi \neq 0$ and $\deg (\pi^\low) \geq 0$, so 2) holds. Assume that $\pi_0 \neq 0$.
Then $\pi_0 = \pi^\low$, and it follows from $\wedge^{r_0} \pi_0 \neq 0$ that $\wedge^{r_0} \pi \neq 0$, so $r \geq r_0$.
Write $\pi = \pi_0 + \pi^\prime$, where 
\[
\pi^\prime = \sum_{1 \leq i < j \leq n} a_{i,j}(z) \frac{\partial}{\partial z_i} \wedge \frac{\partial}{\partial z_j}
\]
and each $a_{i, j}$ has lowest degree at least $2$ at $z = 0$. We then have
\begin{align*}
\wedge^r \pi  = \wedge^r (\pi_0 + \pi^\prime) = \sum_{k=0}^{r} \binom{r}{k}
(\wedge^k \pi_0) \wedge (\wedge^{r-k} \pi^\prime)
= \sum_{k=0}^{r_0} \binom{r}{k}
(\wedge^k \pi_0) \wedge (\wedge^{r-k} \pi^\prime).
\end{align*}
For every $k \in [0, r_0]$, the $2r$-vector field $(\wedge^k \pi_0) \wedge (\wedge^{r-k} \pi^\prime)$ either is $0$ or,
when non-zero, has lowest degree at least $-k\geq -r_0$ at $z =0$. It follows that $(\wedge^r \pi)^\low$ has degree at least 
$-r_0$. By \eqref{eq:sigma-x}, 
\[
\deg \bigl( {\rm Pf}_{\TT}(\pi)^{\low} \bigr) \geq -r_0 = -\frac{1}{2} {\rm rk}(\pi_0). \tag*{\qedhere}
\]
\end{proof}

We now have the following immediate consequence of Theorem \ref{thm:mu-Pfaff} and Proposition \ref{prop:Pfaff-atleast}.

\begin{thm}\label{thm:TT-fixed}
Let $(P, \pi)$ be a connected $\TT$-Poisson manifold with an open dense $\TT$-leaf, 
and assume that $(P, \pi)$ admits $\TT$-log-canonical systems. Let $p_0 \in P$ be any $\TT$-fixed point such that
$\pi(p_0) = 0$. Then 
\[
\deg \bigl( {\rm Pf}_{\TT}(\pi)^{\low} \bigr)  = -\frac{1}{2} {\rm rk}(\pi_0),
\]
where $\pi_0$ be the linearization of $\pi$ at $p_0$ and ${\rm Pf}_{\TT}(\pi)^{\low}$ is taken at $p_0$.
Consequently,  all $\TT$-log-canonical systems 
on $(P, \pi)$ have Property $\calI$ at $p_0$.
\end{thm}

\begin{cor}\label{cor:pi-T-I}
Let $(P, \pi)$ be a smooth rational quasi-affine $\TT$-Poisson variety with an open dense $\TT$-leaf, and 
let $p_0 \in P$ be any $\TT$-fixed point such that
$\pi(p_0) = 0$. Then every cluster structure $\calC$
in $\CC(P)$ that is regular on $P$  and compatible with both $\pi$ and 
the $\TT$-action (see Remark \ref{rem:T-comp}) has Property $\calI$ at $p_0$. 
\end{cor}

\begin{proof}
As every extended cluster of $\calC$ is a $\TT$-log-canonical system 
on $(P, \pi)$, Corollary \ref{cor:pi-T-I} follows directly from Theorem \ref{thm:TT-fixed}.
\end{proof}

\section{Lie theory preliminaries}\label{s:Lie-prelim}
\subsection{Notation}\label{ss:Lie-nota}
Throughout $\S$\ref{s:Lie-prelim}, we fix a connected and simply connected complex semi-simple Lie group $G$ with Lie algebra $\mathfrak{g}$ and a 
choice of opposite Borel subgroups $(B,B_-)$ of $G$ with 
respective Lie algebras $\b$ and $\b_-$. Let $N\subset B$ and $N_-\subset B_-$ be their unipotent radicals
with respective Lie algebras $\n$ and $\n_-$.  Let $T=B\cap B_-$, 
a maximal torus of $G$, and let $\mathfrak{t}$ be
the Lie algebra of $T$.
Let $R\subset \mathfrak{t}^*$ and $R_+\subset R$ be the respective sets of roots in $\mathfrak{g}$ and in $\mathfrak{n}$
 with respect to $T$, and let 
\[
\mathfrak{g} = \t + \n + \n_- = \mathfrak{t}+ \sum_{\alpha \in R_+} \mathfrak{g}_{\alpha} + \sum_{\alpha \in R_+} \mathfrak{g}_{-\alpha}
\]
be the corresponding root decomposition of $\g$,
where $\g_{\alpha}$ for $\alpha \in R$ is the root space of $\alpha$.

Let $\Gamma \subset R_+$ be the set of all simple roots. 
Let $r = \dim \t$, and fix a listing of elements in $\Gamma$ as $\alpha_1, \ldots, \alpha_r$. 
Denote the linear pairing between $\t$ and $\t^*$ by $(\, , \, )$.
For each $i \in [1, r]$, fix root vectors $e_i \in \g_{\alpha_i}$ and $e_{-i} \in \g_{-\alpha_i}$ 
such that $[e_i, e_{-i}] = \alpha_i^\sv$, where
$\alpha_i^\sv$ is the unique element in $[\g_{\alpha_i}, \g_{-\alpha_i}] \subset \t$ 
such that $(\alpha_i^\vee, \alpha_i) = 2$. 
Abusing notation, for $i \in [1, r]$ we also denote by  $\alpha_i^\vee:  \CC^\times \to T$ and 
$e_i: \CC \to N$ and $e_{-i}: \CC \to N_-$ the  group homomorphisms respectively given by
\[
\alpha_i^\vee(e^a) = \exp(a\alpha_i^\vee), \hs
e_i(a)=\exp(ae_i)\in N, \quad \mbox{and} \quad  e_{-i}(a)=\exp(ae_{-i})\in N_-, \quad a \in \CC.
\]
For $i, j \in [1, r]$, the Cartan integer $a_{i, j}\in \ZZ$ is defined as
\[
a_{i, j} = (\alpha_i^\vee, \; \alpha_j).
\]

Let $W = N_G(T)/T$ be the Weyl group, where $N_G(T)$ is the normalizer subgroup of $T$ in $G$. For 
$\alpha \in R$, let $s_\alpha \in W$ be the reflection defined by $\alpha$ and set $s_i = s_{\alpha_i}$
for $i \in [1, r]$. For each $i \in [1, r]$, 
the element 
\[
\overline{s_i} =e_i(-1)e_{-i}(1)e_i(-1)
\]
is then a representative of $s_i \in W$ in $N_G(T)$. 
Let $\ell: W \to \ZZ_{\geq 0}$ be the length function on $W$. For $w \in W$ and $l = \ell(w)$,
using any reduced decomposition $w = s_{i_1}s_{i_2}\cdots s_{i_l}$ one obtains  the representative
\[
\overline{w} = \overline{s_{i_1}} \; \overline{s_{i_2}} \;\cdots\; \overline{s_{i_l}} \in N_G(T)
\]
of $w$ which is 
independent of the choice of the reduced decomposition \cite{FZ:double}. For $w \in W$ and $t \in T$, let
\[
t^w = \ow^{\, -1} t \ow \in T.
\]
Recall also \cite{Tits} that the Tits group is the subgroup $\calT$ of $N_G(T)$ generated by $\{\overline{s_i}: i \in [1, r]\}$.

Let $\{\omega_i: i \in [1, r]\}\subset \t^*$ be the fundamental weights, i.e., 
the basis of $\t^*$ dual to the basis $\{\alpha_i^\sv: i \in [1, r]\}$ of $\t$, and let
$\calP = \sum_{i=1}^r \ZZ \omega_i \subset \t^*$ be the weight lattice. 
We identify $\calP$ with the character lattice of $T$, and for $\lambda \in \calP$ denote by 
$T \to \CC^\times, t \mapsto t^\lambda,$
the corresponding group homomorphism.
  Let $\calQ^\vee =\sum_{i=1}^r \ZZ \alpha_i^\vee\subset \t$ be the co-root
lattice, and for $\beta \in R_+$, let $\beta^\vee \in \calQ^\vee$ be the corresponding co-root.
Let $\rho^\vee =\frac{1}{2}\sum_{\beta \in R_+} \beta^\vee\in \t$.
For $w \in W$, introduce
\begin{equation}\label{eq:tw}
t_w = \overline{w^{-1}}\, \ow \in T \cap \calT.
\end{equation}
We give a proof of the following fact as we could not find a reference in the literature.

\begin{lem}\label{lem:tw}
For any $w \in W$ and $\lambda \in \calP$, one has 
$(t_w)^{\lambda} = (-1)^{(\rho^\vee, \, \lambda-w\lambda)} \in \{1, -1\}$.
\end{lem}

\begin{proof}
Let $w = s_{i_1}\cdots s_{i_l}$ be any reduced decomposition of $w$. As $\overline{s_i}^2 = \alpha_i^\vee(-1)$ for $i \in [1, r]$, one has
\[
t_w = \overline{s_{i_l}\cdots s_{i_2}} \, \overline{s_{i_1}}^2 \, \overline{s_{i_2}\cdots s_{i_l}} = 
\overline{s_{i_l}\cdots s_{i_2}} \; \overline{s_{i_2}\cdots s_{i_l}} \,(\alpha^\vee_{i_1}(-1))^{s_{i_2}\cdots s_{i_l}} =\cdots =
\prod_{k=1}^l( \alpha_{i_k}^\vee(-1))^{s_{i_{k+1}} \cdots s_{i_l}}.
\]
Using the fact that $\{s_{i_l}\cdots s_{i_{k+1}} \alpha_{i_k}: k \in [1, l]\} = R_+ \cap w^{-1}(-R_+)$,  for every
$\lambda \in \calP$ one has
\begin{align*}
(t_w)^{\lambda} &= \prod_{k=1}^l (-1)^{(s_{i_l}\cdots s_{i_{k+1}} \alpha_{i_k}^\vee, \, \lambda)}
=(-1)^{\left(\sum_{\beta \in R_+\cap w^{-1}(-R_+)}\beta^\vee, \; \lambda\right)}\\
&= (-1)^{(\rho^\vee-w^{-1}\rho^\vee, \, \lambda)}=(-1)^{(\rho^\vee, \, \lambda-w\lambda)}.
\end{align*}
Since $\lambda-w\lambda$ lies in the root lattice, and since $(\rho^\vee, \alpha) =1$ for every simple root
$\alpha$, one has 
$(\rho^\vee, \, \lambda-w\lambda) \in \ZZ$, and thus
$(t_w)^{\lambda} = \pm 1$.
\end{proof}

\begin{rem}\label{rem:w0-center}
Let $w_0$ be the longest element in $W$. It follows from $\rho^\vee-w_0\rho^\vee = 2\rho^\vee$ that
for every simple root $\alpha$ one has $(t_{w_0})^{\alpha} = (-1)^{2(\rho^\vee, \alpha)} = 1$. 
One thus arrives at the well-known fact (see, for example, \cite[Section 3]{Adam-He:lifts}) that 
$\overline{w_0}^{\, 2} \in T$ lies in the center of $G$.
\end{rem}

Let again $w_0 \in W$ be the longest element.
For $w \in W$, one  has $\overline{w_0} =\overline{w^{-1}} \, \overline{ww_0}$, and it follows that
\begin{equation}\label{eq:ww0}
\ow \, \overline{w_0} = \ow \, \left(\overline{w^{-1}}\, \ow \right) \, \ow^{\, -1}\,  \overline{ww_0}
= (t_w)^{w^{-1}} \,  \overline{ww_0}\, =  \overline{ww_0}\, (t_w)^{w_0}.
\end{equation}
Define $w^* = w_0ww_0$ for $w \in W$ and note that
$\overline{w_0} = \overline{w^*} \, \overline{(w^*)^{-1}w_0} = \overline{w^*} \, \overline{w_0w^{-1}} =
\overline{w^*} \, \overline{w_0} \, \overline{w}^{\, -1}$. 
Thus
\begin{equation}\label{eq:www}
\overline{w^*} = \overline{w_0} \, \overline{w} \, \overline{w_0}^{\,-1} = \overline{w_0}^{\, -1} \, 
\overline{w} \, \overline{w_0},\hs w \in W,
\end{equation}
where in the second identity we used the fact  that $\overline{w_0}^{\, 2}$ lies in the center of $G$.

\subsection{Generalized minors}\label{ss:minors-G}
For any finite dimensional representation $V$ of $G$, if $\eta \in \calP$ is a weight of $V$, denote by 
$V_\eta$ the $\eta$-weight subspace of $V$. For $i \in [1, r]$, let $i^* \in [1, r]$ be such that 
$\omega_{i^*} = -w_0\omega_i$. 

Let $i \in [1, r]$ and let $V(\omega_i)$ be the
highest weight representation of $G$ with highest weight $\omega_i$. Then the dual space $V(\omega_i)^*$ of $V(\omega_i)$,
equipped with the representation of $G$ such that
\[
(g\zeta, \, g\nu) = (\zeta, \, \nu), \hs \;\;  \zeta \in V(\omega_i)^*, \, \nu \in V(\omega_i),\, g \in G,
\]
is a highest weight representation of $G$ with highest weight $-w_0\omega_i = \omega_{i^*}$, where $(\, , \, )$ denotes the pairing
between $V(\omega_i)^*$ and $V(\omega_i)$.
Let $\nu_{\omega_i} \in V(\omega_i)_{\omega_i}$ be any (non-zero)
highest weight vector in $V(\omega_i)$. Then there is a unique lowest weight vector 
$\zeta_{w_0\omega_{i^*}} \in (V(\omega_{i})^*)_{w_0\omega_{i^*}} = (V(\omega_{i})^*)_{-\omega_{i}}$ in $(V(\omega_i))^*$
such that 
\begin{equation}\label{eq:zeta-nu}
(\zeta_{w_0\omega_{i^*}}, \; \nu_{\omega_i}) = 1.
\end{equation}
Let $u \in W$. Then
$\overline{u} \,\zeta_{w_0\omega_{i^*}} \in V(\omega_i)^*$ 
is the unique element in $(V(\omega_{i})^*)_{uw_0\omega_{i^*}} = (V(\omega_{i})^*)_{-u\omega_{i}}$ such that
$(\overline{u} \,\zeta_{w_0\omega_{i^*}}, \; \overline{u}\nu_{\omega_i}) = 1$.
For $u, v \in W$ and $i \in [1, r]$, define $\Delta_{u\omega_i, v\omega_i} \in \CC[G]$ by \cite[Definition 1.4]{FZ:double}
\[
\Delta_{u\omega_i, v\omega_i}(g) = (\overline{u}\, \zeta_{w_0\omega_{i^*}}, \; g\overline{v}\nu_{\omega_i})  =
(\zeta_{w_0\omega_{i^*}}, \; \overline{u}^{\, -1} g\overline{v}\nu_{\omega_i}) =
\Delta_{\omega_i, \omega_i}(\overline{u}^{\, -1} g \overline{v}), \hs g \in G.
\]
The definition of $\Delta_{u\omega_i, v\omega_i}$ depends only on the weights $u\omega_i$ and $v\omega_i$ of
$V(\omega_i)$ and not on 
the choice of the highest weight 
vector $\nu_{\omega_i}$ (see \cite[Proposition 2.3]{FZ:double} and \cite[Lemma 6.1 and Definition 6.2]{MR:flag}).



The following facts are well-known (\cite[Proposition 2.2 and Theorem 1.17]{FZ:double}, \cite[Remark 7.5]{BZ:tensor}).

\begin{lem}\label{lem:gig} 1) Let $i, i' \in [1, r]$, $\delta \in W\omega_i$, and $i'\neq i$. Then for 
any $a \in \CC$ and $g \in G$,
\begin{align}\label{eq:gqq}
&\Delta_{\delta, \,\omega_i}(g \,e_{i}(a) \overline{s_i})
= a \,\Delta_{\delta, \,\omega_i}(g) + \Delta_{\delta, \,s_i\omega_i}(g)
\hs \mbox{and} \hs 
\Delta_{\delta, \,\omega_i}(g \,e_{i'}(a) \overline{s_{i'}}) = \Delta_{\delta, \, \omega_i}(g),\\
\label{eq:gp}
&\Delta_{\delta, \, \omega_i}(ge_{-i}(a))= \Delta_{\delta, \, \omega_i}(g) + a\Delta_{\delta, \, s_i\omega_i}(g)
\hs \mbox{and} \hs 
\Delta_{\delta, \, \omega_i}(ge_{-i'}(a))= \Delta_{\delta, \, \omega_i}(g),\\
\label{eq:qqg}
&\Delta_{\omega_i, \, \delta}(e_{i}(a) \overline{s_i} \,g)
= a \,\Delta_{\omega_i, \,\delta}(g) - \Delta_{s_i\omega_i, \, \delta}(g)
\hs \mbox{and} \hs 
\Delta_{\omega_i, \, \delta}(e_{i'}(a) \overline{s_{i'}}\,g) = \Delta_{\omega_i, \, \delta}(g),\\
\label{eq:pg}
&\Delta_{\omega_i, \, \delta}(e_{i}(a) \,g)
= \Delta_{\omega_i, \,\delta}(g) +a \Delta_{s_i\omega_i, \, \delta}(g)
\hs \mbox{and} \hs 
\Delta_{\omega_i, \, \delta}(e_{i'}(a) \,g) = \Delta_{\omega_i, \, \delta}(g).
\end{align}

2) If $u, v \in W$ and $i \in [1, r]$ are such that
$\ell(us_i) = \ell(u) +1$ and $\ell(vs_i) = \ell(v)+1$, then
\begin{equation}\label{eq:master}
\Delta_{u\omega_i, \, v\omega_i} \Delta_{us_i\omega_i, \, vs_i\omega_i} =
\Delta_{us_i\omega_i, \, v\omega_i} \Delta_{u\omega_i, \, vs_i\omega_i} +\prod_{j \neq i} \Delta_{u\omega_j, \, v\omega_j}^{-a_{j, i}}.
\end{equation} 
\end{lem}


\begin{lem}\label{lem:w0-g}
For every $i \in [1, r]$, $u, v \in W$,  and all $g \in G$, one has
\begin{equation}\label{eq:w0-g}
\Delta_{u\omega_i, v\omega_i}(g) 
= \Delta_{v^*\omega_{i^*}, u^*\omega_{i^*}}(\overline{w_0}\, g^{-1}\, \overline{w_0}^{\, -1})
=(-1)^{(\rho^\vee, \, u\omega_i-v\omega_i)}\Delta_{vw_0 \omega_{i^*}, \, uw_0\omega_{i^*}}(g^{-1}).
\end{equation}
\end{lem}

\begin{proof}
Note first that 
$\overline{w_0}\,\zeta_{w_0\omega_{i^*}} \in V(\omega_i)^*$ is a highest weight vector of highest weight $\omega_{i^*}$, $\overline{w_0}\, \nu_{\omega_i} \in V(\omega_i)$ is a lowest weight vector of lowest weight $w_0 \omega_i = -\omega_{i^*}$, and $(\overline{w_0}\,\nu_{\omega_i}, \, 
\overline{w_0}\, \zeta_{w_0\omega_{i^*}}) = 1$. Using \eqref{eq:www}, one has
\begin{align*}
\Delta_{u\omega_i,\, v\omega_i}(g)& = (\overline{u} \zeta_{w_0\omega_{i^*}}, \; g \overline{v}\nu_{\omega_i}) = 
(g\overline{v} \, \overline{w_0}^{\, -1} (\overline{w_0} \,\nu_{\omega_i}), \; \overline{u}\,  \overline{w_0}^{\, -1} (\overline{w_0}\,\zeta_{w_0\omega_{i^*}}))\\
& = (\overline{w_0}\, g\overline{v} \, \overline{w_0}^{\, -1} (\overline{w_0} \nu_{\omega_i}), \; \overline{w_0}\,\overline{u}\,  \overline{w_0}^{\, -1} (\overline{w_0}\,\zeta_{w_0\omega_{i^*}}))= 
(\overline{w_0}\, g \, \overline{w_0}^{\, -1}\, \overline{v^*}\, (\overline{w_0} \nu_{\omega_i}), \; 
\overline{u^*}\, (\overline{w_0} \,\zeta_{w_0\omega_{i^*}}))\\
&=(\overline{v^*}\, (\overline{w_0} \,\nu_{\omega_i}), \; \overline{w_0}\, g^{-1} \, \overline{w_0}^{\, -1}\,
\overline{u^*}\, (\overline{w_0} \,\zeta_{w_0\omega_{i^*}}))=
\Delta_{v^* \omega_{i^*}, \, u^*\omega_{i^*}}(\overline{w_0}\, g^{-1} \, \overline{w_0}^{\, -1}).
\end{align*}
Similarly, using \eqref{eq:ww0} and the fact that 
$(\overline{w_0}^{\, -1}\,\nu_{\omega_i}, \, \overline{w_0}^{\, -1}\, \zeta_{w_0\omega_{i^*}}) = 1$, one also has
\begin{align*}
\Delta_{u\omega_i,\, v\omega_i}(g)& = (\overline{u}\, \zeta_{w_0\omega_{i^*}}, \; g\, \overline{v}\nu_{\omega_i}) = 
(g \,\overline{v}\, \overline{w_0}\, (\overline{w_0}^{\, -1}\nu_{\omega_i}), \; 
\overline{u} \, \overline{w_0}\, (\overline{w_0}^{\, -1}\zeta_{w_0\omega_{i^*}}))\\
& = 
t_u^{-\omega_{i}} t_v^{\omega_i} (\overline{vw_0} \,  (\overline{w_0}^{\, -1}\nu_{\omega_i}), \;
g^{-1} \overline{uw_0}\, (\overline{w_0}^{\, -1}\zeta_{w_0\omega_{i^*}}))\\
&=t_u^{-\omega_{i}} t_v^{\omega_{i}} \Delta_{vw_0 \omega_{i^*}, \, uw_0\omega_{i^*}}(g^{-1})
=(-1)^{(\rho^\vee, \, u\omega_i-v\omega_i)}
\Delta_{vw_0 \omega_{i^*}, \, uw_0\omega_{i^*}}(g^{-1}). \tag*{\qedhere}
\end{align*}
\end{proof}

\begin{rem}\label{rem:inv}
The first identity in \eqref{eq:w0-g} can also be derived from \cite[(2.15) and (2.25)]{FZ:double}. Compare also the last expression in \eqref{eq:w0-g} with the identity $\Delta_{u\omega_i, v\omega_i}(g) 
 =\Delta_{vw_0 \omega_{i^*}, \, uw_0\omega_{i^*}}(g^\iota)$ from \cite[(2.25)]{FZ:double},
where $\iota$ is the 
involutive anti-automorphism on $G$ given in \cite[Section 2.1]{FZ:double}. Note that $\iota$ preserves the totally positive part of 
$G$ defined by Lusztig (see \cite{FZ:double} for more detail) while the map $g \mapsto g^{-1}$ does not.
\end{rem}

\subsection{The signed generalized minors \texorpdfstring{$\delta_{u\omega_i, v\omega_i}$}{duv} on \texorpdfstring{$\g$}{g}}\label{ss:signed-minor}
Assume first that $P$ is any complex Lie group with Lie algebra $\p$, and let $e \in P$ be the identity element. 
Let $V$ be a finite dimensional representation of $P$ over $\CC$. For $v \in V$ and $\xi \in V^*$,  we then have the 
{\it matrix coefficient}
$\vphi_{\xi, v}$, defined as the 
holomorphic function  on $P$ given by
\[
\vphi_{\xi, v}(p) = (\xi, \; p v), \hs p \in P.
\]
Consider also the linear function $\widehat{\vphi}_{\xi, v}$ on the universal enveloping algebra
$U\p$ of $\p$ given by 
\[
\widehat{\vphi}_{\xi, v}(X) = (\xi, \; Xv), \hs X \in U\p.
\]
As $\vphi_{\xi, v}$ is holomorphic, $\vphi_{\xi, v} \neq 0$ as a function on $P$ 
if and only if $\widehat{\vphi}_{\xi, v}\neq 0$ as a linear function on $U\p$.
Let
\[
\CC = U_{0}\p \subset \CC \oplus \p = U_1 \p \subset U_2 \p \subset \cdots \subset  U\p
\]
be the standard filtration of $U\p$. The following definition is due to Kostant \cite[$\S$2.7]{K13}.

\begin{de-lem}\label{de-lem:co-deg}
For a non-zero linear function $f$ on $U\p$, define
the {\it co-degree} of $f$, denoted by ${\rm cdeg}(f)$, to be 
the smallest integer $d \geq 0$ such that $f|_{U_d \p} \neq 0$. If $d = {\rm cdeg}(f)$, define
the {\it co-degree term} of $f$ to be the non-zero
homogeneous polynomial function $f_{(d)}$ on $\p$ of degree $d$ given by
\[
f_{(d)}(x) =\frac{1}{d!}f(x^d), \hs x \in \p.
\]
If $(e_1, \ldots, e_n)$ is a basis of $\p$, by writing $x =x_1e_1 + \cdots + x_n e_n \in \p$ and using the
definition of $d = {\rm cdeg}(f)$, we see that $f_{(d)} \in \CC[\p]$ is explicitly given by
\[
f_{(d)}(x) = \sum_{j_1, \ldots, j_n \geq 0, \,j_1 + \cdots +j_n = d} 
\frac{1}{j_1! \cdots j_n!} x_1^{j_1} \cdots x_n^{j_n} f(e_1^{j_1} \cdots e_n^{j_n}).
\]
\end{de-lem}

\begin{lem}\label{lem:co-low}
Assume that $v \in V$ and $\xi \in V^*$ are such that the function ${\vphi}_{\xi, v}$ on $P$ is non-zero, and let
$\vphi_{\xi, v}^\low \in \CC[\p]$ be the  
lowest degree term of $\vphi_{\xi, v}$ at the identity element $e \in P$. Let
$d = {\rm cdeg} (\widehat{\vphi}_{\xi, v})$. Then
\[
\vphi_{\xi, v}^\low  = (\widehat{\vphi}_{\xi, v})_{(d)} \in \CC[\p].
\]
\end{lem}

\begin{proof}
Let $(e_1, \ldots, e_n)$ be a basis of $\p$.   Then the map 
\[
\CC^n \ni x=(x_1, \ldots, x_n) \longmapsto \exp(x_1e_1) \exp(x_2e_2) \cdots \exp(x_ne_n) \in P
\]
defines a local holomorphic coordinate system on $P$ at $e$ with $x(e)=0$.
By the definition of $d = {\rm cdeg}(\widehat\vphi_{\xi, v})$, the lowest degree term of the Taylor expansion of
$\vphi_{\xi, v}$ in $x$ at $0$  is precisely 
the polynomial $(\widehat{\vphi}_{\xi, v})_{(d)}(x)$.
\end{proof}

We now return to the  simply connected complex semi-simple Lie group $G$ and continue 
with the notation in $\S$\ref{ss:Lie-nota}.
The following statement follows directly from Lemma \ref{lem:co-low}.

\begin{de-lem}\label{de-lem:delta} 
{\rm
Let $i \in [1, r]$ and let $\nu_{\omega_i} \in V(\omega_i)_{\omega_i}$ and $\zeta_{w_0\omega_{i^*}} \in (V(\omega_i)^*)_{w_0\omega_{i^*}}$ be as in \eqref{eq:zeta-nu}. For $u, v \in W$, define the linear 
function  $\widehat{\Delta}_{{u\omega_i, v\omega_i}}$ on $U\g$ by
\[
\widehat{\Delta}_{{u\omega_i, v\omega_i}} (X) = (\overline{u}\hspace{.009in}\zeta_{w_0\omega_{i^*}}, \; X\hspace{.009in}\overline{v}\hspace{.009in}\nu_{\omega_i}), \hs X \in U\g,
\]
and define $\delta_{u\omega_i, \, v\omega_i} \in \CC[\g]$ to be 
the co-degree term of $\widehat{\Delta}_{{u\omega_i, v\omega_i}}$ as in Definition-Lemma \ref{de-lem:co-deg}. Then
\[
\delta_{u\omega_i, \, v\omega_i} = \Delta_{u \omega_i, \, v\omega_i}^\low \in \CC[\g],
\]
where the right hand side denotes the lowest degree term of $\Delta_{u\omega_i, \, v\omega_i} \in \CC[G]$ at the 
identity element $e \in G$. For $x \in \g$, we call $\delta_{u\omega_i, \, v\omega_i}(x)$ a 
{\it signed generalized minor of $x$}.
}
\end{de-lem}

\begin{rem}\label{rm:signed}
The term signed generalized minor comes from 
the case of $\SL(n, \CC)$ (see Lemma \ref{lem:GLn-minor}).
\end{rem}

Note that for $i \in [1, r]$, one has
\[
(x \zeta, \, \nu) + (\zeta, \, x\nu) = 0, \hs \forall\; x \in \g, \, \zeta \in V(\omega_i)^*, \, \nu \in V(\omega_i).
\]
Replacing the map $G \to G, g \mapsto g^{-1}$ by the antipode $S: U\g \to U\g$ of the Hopf algebra $U\g$ and noting
that $S$ preserves the filtration of $U\g$  and 
$S(x^d) = (-1)^d x^d$ for all $x \in \g$ and $d \in \ZZ_{\geq 0}$, we can modify 
the proof of Lemma \ref{lem:w0-g} 
and get the following analog of \eqref{eq:w0-g}.

\begin{lem}\label{lem:delta-w0}
For any $i \in [1, r]$ and $u, v \in W$, the three homogeneous polynomials $\delta_{u\omega_i, v\omega_i}, 
\delta_{v^*\omega_{i^*}, u^*\omega_{i^*}}$, and $\delta_{vw_0\omega_{i^*}, uw_0\omega_{i^*}}$ on $\g$ have the same 
degree, and, denoting their degree by $d$, one has
\[
\delta_{u\omega_i, v\omega_i}(x) = (-1)^d \delta_{v^*\omega_{i^*}, u^*\omega_{i^*}}({\rm Ad}_{\overline{w_0}} x)
= (-1)^{d+(\rho^\vee, \, u\omega_i-v\omega_i)}\delta_{vw_0\omega_{i^*}, uw_0\omega_{i^*}}(x), \hs x \in \g.
\]
\end{lem} 

Generalized minors of the form $\Delta_{\omega_i, w\omega_i}$ or 
$\Delta_{w\omega_i, \omega_i}$, where $i \in [1, r]$ and $w \in W$,
are
called {\it flag minors}.  We now prove some properties of 
the corresponding  $\delta_{\omega_i, w\omega_i}\in \CC[\g]$ to be used in $\S$\ref{s:GB}. 
For $w \in W$, set
\begin{equation}
\n_w = \n \cap \Ad_{\ow} \n^-,  \hs  \n_w^\prime = \n \cap \Ad_{\ow} \n, \hs \mbox{and} \hs
N_w = N \cap \ow N_-\ow^{\, -1},
\end{equation}
and denote the projection $\g \to \n_w$ given by the decomposition $\g = \b_- +\n_w^\prime + \n_w$ by
\begin{equation}\label{eq:pr-nw}
{\rm pr}_{\n_w}:\;\; \g \longrightarrow \n_w, \;\; x \longmapsto x_{\n_w}.
\end{equation}

\begin{lem}\label{lem:delta-1}
For any $i \in [1, r]$ and $w \in W$, one has $\delta_{\omega_i, w\omega_i}(x) = 
\delta_{\omega_i, w\omega_i}(x_{\n_w})$ for all $x \in \g$, and
\[
\delta_{\omega_i, w\omega_i}|_{\n_w} = \left(\Delta_{\omega_i, w\omega_i}|_{N_w}\right)^\low,
\]
where $\left(\Delta_{\omega_i, w\omega_i}|_{N_w}\right)^\low$ is the lowest degree term of 
$\Delta_{\omega_i, w\omega_i}|_{N_w}
\in \CC[N_w]$ at the identity element  $e \in N_w$.
\end{lem}

\begin{proof} Let $d = {\rm cdeg}(\widehat{\Delta}_{\omega_i, w\omega_i})$. Then $d = 0$ if and only if $w\omega_i
= \omega_i$, and in this case $\delta_{\omega_i, w\omega_i} \in \CC[\g]$ is the constant function with value $1$
and there is nothing to prove. We thus assume that $\omega_i \neq w\omega_i$ and thus $d \geq 1$.

Let $x \in \g$ and write $x = x^\prime +x_{\n_w}$, where $x' \in \b_- +\n_w^\prime$. By the Poincar\'e-Birkhoff-Witt theorem, 
\[
x^d = (x_{\n_w})^d + X + X' \in U_d(\g),
\]
where $X \in \b_- U_{d-1}\g + (U_{d-1}\g) \,\n_w^\prime$, and $X' \in U_{d-1}\g$. 
By the definition of $\widehat{\Delta}_{\omega_i, w\omega_i} \in (U\g)^*$ and the fact that $\n_w' \overline{w} \nu_{\omega_i} = 0$, 
\[
\widehat{\Delta}_{\omega_i, w\omega_i}((x_0+x_-)Y + Zx') = \omega_i(x_0)\widehat{\Delta}_{\omega_i, w\omega_i}(Y), 
\hs \forall\; x_0 \in \t, \, x_- \in \n_-, \, x' \in \n_w^\prime, \, Y, \, Z \in U\g.
\]
It follows from $\widehat{\Delta}_{\omega_i, w\omega_i}|_{U_{d-1}\g} = 0$ that 
$\widehat{\Delta}_{\omega_i, w\omega_i}(x^d) 
=\widehat{\Delta}_{\omega_i, w\omega_i}((x_{\n_w})^d)$. Thus $\delta_{\omega_i, w\omega_i}(x) = 
\delta_{\omega_i, w\omega_i}(x_{\n_w})$.

Set $N_w^\prime = N \cap \ow N \ow^{\, -1}$. Writing every $g \in B_-N$ uniquely as 
$g = t n_- n_w n_w^\prime$ with $t \in T$, $n_- \in N_-$, $n_w \in N_w$ and $n_w^\prime \in N_w^\prime$, we have
$\Delta_{\omega_i, \, w\omega_i}(g) = t^{\omega_i} \Delta_{\omega_i, \, w\omega_i}(n_w)$. 
It follows that
\[
\left(\Delta_{\omega_i, w\omega_i}|_{N_w}\right)^\low = \left(\Delta_{\omega_i, \, w\omega_i}^\low\right)|_{\n_w} =
\delta_{\omega_i, w\omega_i}|_{\n_w} \in \CC[\n_w]. \tag*{\qedhere}
\]
\end{proof}

For $w \in W$ and $i \in [1, r]$, consider now the restriction $\delta_{\omega_i, w\omega_i}|_{\n_w} \in \CC[\n_w]$ of $\delta_{\omega_i, w\omega_i}$ to
$\n_w$. Let 
\[
N_{w, -} = N_- \cap \ow N \ow^{\, -1} \hs \mbox{and} \hs \n_{w,-} = \n_- \cap {\rm Ad}_{\overline{w}} \n,
\]
and let $\lara_\g$ be any symmetric and non-degenerate invariant bilinear form on $\g$. Then
$\b_- +\n_w^\prime$, being the annihilator of 
$\n_{w, -}$ in $\g$ with respect to $\lara_\g$, is invariant under the adjoint action of $N_{w, -}$ on $\g$. 
One thus has the left action of $N_{w, -}$ on $\n_w$ given by
\begin{equation}\label{eq:Nwm-nw}
N_{w, -} \times \n_w \longrightarrow \n_w, \;\; (g, \, x) = (\Ad_g x)_{\n_w},
\end{equation}
which, under the identification $\n_w \cong \n_{w, -}^*$ via the bilinear form $\lara_\g$,
becomes the co-adjoint action of $N_{w, -}$ on $\n_{w, -}^*$.
Let $\CC[\n_w]^{N_{w, -}}$ be the algebra of $N_{w, -}$-invariant polynomial functions on $\n_w$.

\begin{lem}\label{lem:Nwm-inv}
For any $i \in [1, r]$ and $w \in W$, one has $\delta_{\omega_i, w\omega_i}|_{\n_w} \in \CC[\n_w]^{N_{w, -}}$.
\end{lem}

\begin{proof}
The adjoint action of $G$ on $\g$ extends to an action of $G$ on $U\g$ which we
denote as $X \mapsto {\rm Ad}_g (X) = g X g^{-1}$ for $g \in G$ and $X \in U\g$. It follows from the definitions of
$\widehat{\Delta}_{\omega_i, , w\omega_i}$ and $N_{w, -}$ that 
\begin{equation}\label{eq:gXg}
\widehat{\Delta}_{\omega_i, \, w\omega_i}(gXg^{-1}) = \widehat{\Delta}_{\omega_i, \, w\omega_i}(X), \hs \forall \, g \in N_{w, -}, \, 
X \in U\g.
\end{equation}
Let now $x \in \n_w$ and $g \in N_{w, -}$,  and 
let $d = {\rm cdeg}(\widehat{\Delta}_{\omega_i, w\omega_i})$. By Lemma \ref{lem:delta-1} and \eqref{eq:gXg}, we have
\[
\delta_{\omega_i, w\omega_i}(({\rm Ad}_g x)_{\n_w}) = \delta_{\omega_i, w\omega_i}({\rm Ad}_g x)
=\frac{1}{d!} \widehat{\Delta}_{\omega_i, \, w\omega_i}(gx^dg^{-1})=\frac{1}{d!} \widehat{\Delta}_{\omega_i, \, w\omega_i}(x^d)
=\delta_{\omega_i, w\omega_i}(x).  \tag*{\qedhere}
\]
\end{proof}

For subsets $I, J$ of $[1,n]$ with $|I| = |J|$ and any $n \times n$ matrix $g$, we write $\Delta_{I,J}(g)$ for the determinant of the submatrix of $g$ whose rows are indexed by $i \in I$ and columns by $j \in J$, and it is understood that
$\Delta_{I, J} = 1$ if $I = J = \emptyset$. Consider $G = \GL(n,\CC)$ or $G = \SL(n,\CC)$, where $n \geq 2$. Denote by $\g$ either $\gl(n,\CC)$ or $\sl(n,\CC)$.  Let $\delta_{I, J} =\Delta_{I,J}^\low \in \CC[\g]$ be  the lowest degree term of $\Delta_{I,J}$ at the identity matrix $\bbone_n \in G$.

\begin{lem} \label{lem:GLn-minor}
For $G = \GL(n,\CC)$ or $G = \SL(n,\CC)$ with $n \geq 2$, and for $I, J \subset [1,n]$ with $|I| = |J|$, one has
\[
\delta_{I, J} = \Delta_{I,J}^\low = (-1)^{\sign(\sigma_I)+\sign(\sigma_J)} \Delta_{I \setminus K, J \setminus K},
\]
where $K=I \cap J$, and, writing $K=\{k_1 < \cdots < k_m\}$, $I\setminus K=\{i_1 < \cdots < i_l\}$, $J\setminus K=\{j_1 < \cdots < j_l\}$, $\sigma_I$ (resp. $\sigma_J$) is the permutation of $I$ (resp. $J$) sending its elements listed in increasing order to the ordered set
\[
\{k_1,\ldots,k_m,i_1,\ldots,i_l\}\hs \Big(\text{resp.~}\{k_1,\ldots,k_m,j_1,\ldots,j_l\}\Big).
\]
\end{lem}

\begin{proof}
    Let $x \in \gl(n,\CC)$.  The entries $x_{i,j}$ of $x$ form a local holomorphic coordinate system on $\GL(n,\CC)$ near $\bbone_n \in \GL(n,\CC)$ via the map
    \[
    x \longmapsto \exp(x) = \bbone_n + x + \cdots.
    \]
    Since the diagonal $1$'s of $\bbone_n$ are of degree $0$ and the entries $x_{i,j}$ of $x$ are of degree $1$, the more diagonal entries of $\exp(x)$ a term of $\Delta_{I,J}(\exp(x))$ contains, the lower the degree of its lowest degree term.  The assertion for $\GL(n, \CC)$ follows.  The proof of the assertion for $\SL(n, \CC)$ is analogous.
\end{proof}

\subsection{Kostant's cascade of roots and the homogeneous degree of \texorpdfstring{$\delta_{\omega_i, w_0\omega_i}$}{d(e,w0)}
}\label{ss:cascade} 
Consider now the special case of
$w = w_0$.
By \eqref{eq:Nwm-nw}  we have the left $N_-$-action on $\n$ by 
\begin{equation}\label{eq:N-minus-n}
g \cdot x = ({\rm Ad}_g x)_\n, \hs g \in N_-, \, x \in \n,
\end{equation}
where $y \mapsto y_{\n}$ for $y \in \g$ denotes the projection $\g = \b_- + \n \to \n$. 
For every $i \in [1, r]$, we then have 
\[
\delta_{\omega_i, w_0\omega_i}(x) = \delta_{\omega_i, w_0\omega_i}(x_{\n}), \hs x \in \g,
\]
and
$\delta_{\omega_i, w_0\omega_i}|_{\n} \in \CC[\n]^{N_-}$ (see Lemma \ref{lem:delta-1} and Lemma \ref{lem:Nwm-inv}). 
Set,
for $i \in [1, r]$,
\begin{equation}\label{eq:di-defn}
d_i = \mbox{homogeneous degree of}\; (\delta_{\omega_i, w_0\omega_i} \in \CC[\g]) = 
\mbox{homogeneous degree of} \; 
(\delta_{\omega_i, w_0\omega_i}|_{\n} \in \CC[\n]).
\end{equation}
By Lemma \ref{lem:delta-w0} and the fact that $w_0 \rho^\vee = -\rho^\vee$, for $i \in [1, r]$ one has
\begin{equation}\label{eq:delta-di}
\delta_{\omega_i, \,w_0\omega_i}(x) = (-1)^{d_i} \delta_{w_0\omega_{i^*}, \,\omega_{i^*}}({\rm Ad}_{\overline{w_0}} x)
= (-1)^{d_i + (2\rho^\vee, \,\omega_i)}\delta_{\omega_{i^*}, \,w_0\omega_{i^*}}(x), \hs x \in \g.
\end{equation}
To compute the numbers $d_i$ for $i \in [1, r]$, and to 
compute the index of several Lie algebras, we now recall 
Kostant's cascade $\calB$ of roots \cite{K12}, iteratively defined as follows.

For each connected component of the 
Dynkin diagram of $\g$, put in $\calB$ the highest root $\beta$ of the Lie sub-algebra of $\g$ defined by that connected 
component and remove all simple roots  that are not orthogonal to $\beta$. Do the same to the resulting 
Dynkin sub-diagram and iterate the procedure until it cannot be performed any further. By 
\cite[Theorem 1.8]{K12}, $\calB$ is a maximal set of strongly orthogonal positive roots, i.e., $\beta + \beta'$ and $\beta -\beta^\prime$ are not roots for any $\beta, \beta' \in \calB$.
By \cite[Propostion 1.10]{K12}, $\calB$ is a basis for $\ker (1+w_0) \subset \t^*$. 
Choose root vectors $e_{\pm \beta} \in \g_{\pm \beta}$ 
for each $\beta \in \calB$, and let 
\begin{equation}\label{eq:rrr}
\r = \sum_{\beta \in \calB} \CC  e_\beta \subset \n, \hs \r^\times = \sum_{\beta \in \calB} \CC^\times  e_\beta \subset \r,
\hs \mbox{and}  \hs 
\r_- = \sum_{\beta \in \calB} \CC e_{-\beta}  \subset \n_-.
\end{equation}
Let $R_- =\exp(\r_-)\subset N_-$ and 
$X = N_-\cdot \r^\times \subset \n$, where $N_-$ acts on $\n$ as in \eqref{eq:N-minus-n}.

\begin{lem}\label{lem:Kostant} \cite[Theorem 2.3 and Theorem 2.8]{K12}
 The set $X$ is 
Zariski open in $\n$, and every $N_-$-orbit in $X$ passes through exactly one point in $\r^\times$. Moreover,
the stabilizer subgroup of $N_-$ at every $x \in \r^\times$ is $R_-$. 
\end{lem}

Recall that  $\calQ^\vee =\sum_{\alpha \in \Gamma} \ZZ \alpha^\vee \subset \t$ is the co-root lattice. Set
\begin{equation}\label{eq:kappa-0}
\kappa^\vee = \sum_{\beta \in \calB} \beta^\vee \in \calQ^\vee.
\end{equation}

\begin{lem}\label{lem:di}
For any $i \in [1, r]$, one has
$d_i = (\kappa^\vee, \, \omega_i)$, and
$\delta_{\omega_i, w_0\omega_i}= (-1)^{(2\rho^\vee-\kappa^\vee, \, \omega_i)} \delta_{\omega_{i^*}, w_0\omega_{i^*}} \in \CC[\g]$.
\end{lem}

\begin{proof}
Let $\vphi \in \CC[\n]^{N_-}$ be homogeneous with $\deg(\vphi) \geq 0$ and assume that $\vphi$ is also a $T$-weight vector with $T$-weight $\chi_\vphi$.  It is known that (\cite[Theorem 3.5]{K12}) $\chi_{\vphi}$ is in the $\ZZ$-span of $\mathcal B$.   Write $\chi_{\vphi} = \sum_{\beta \in \mathcal B} [\chi_{\vphi}: \beta] \beta$.  By (\cite[Theorem 3.11]{K12}),
\begin{equation}\label{eq:deg-chi}
\deg(\vphi) = \sum_{\beta \in \calB} [\chi_\vphi:\beta].
\end{equation}
Consider $\vphi =\delta_{\omega_i, w_0\omega_i}|_{\n}  \in \CC[\n]^{N_-}$, whose $T$-weight is $\omega_i -w_0\omega_i = \omega_i + \omega_{i^*}$. 
We have
$2 [\omega_i - w_0 \omega_i: \beta] = (\beta^\vee, \, \omega_i - w_0 \omega_i) =2(\beta^\vee, \,\omega_i)$,
so
 $d_i= \sum_{\beta \in \calB} (\beta^\vee, \omega_i)=(\kappa^\vee, \, \omega_i)$. The second assertion follows from 
\eqref{eq:delta-di}.
\end{proof}

\begin{rem}\label{rem:2rho}
We have seen that under the identification $\n \cong \n_-^*$ via any 
$\lara_\g$, 
the action of $N_-$ on $\n$ becomes the co-adjoint action of $N_-$ on $\n_-^*$, so
the algebra $\CC[\n]^{N_-}$ can be identified with the Poisson center of $\CC[\n_-^*]$ with the Poisson bracket coming from the
Kirillov-Kostant-Souriau Poisson structure $\pi_{0, \n_-}$ on $\n_-^*$.  
Kostant \cite[Theorem 3.6]{K12} proved that the Poisson center of $\CC[\n_-^*]$ is a polynomial ring in $|\mathcal B|$ generators, although he did not explicitly construct those generators.  Lemma \ref{lem:Nwm-inv} leads to an explicit presentation of them, thus strengthening Kostant's result.  In fact,  by Lemma \ref{lem:delta-w0} 
and for each $i \in [1, r]$,
the two polynomials
$\delta_{\omega_i, w_0 \omega_i}|_{\n}$ and $\delta_{\omega_{i^*}, w_0 \omega_{i^*}}|_{\n}$ 
are either the same or differ by a negative sign. 
One can prove that, for each $i \in [1,r]$, $\omega_i + \omega_{i^*} \in \Span_{\ZZ} \{\beta \colon \beta \in \mathcal B\}$ is either primitive, in the sense that the coefficients form a primitive vector in $\ZZ^{\mathcal B}$, or $2$ times a primitive element.  In the latter case, $\delta_{\omega_i, w_0 \omega_i}|_{\n}$ is the square of a Poisson central polynomial, to be denoted as $\sqrt{\delta_{\omega_i, w_0 \omega_i}|_{\n}}$, on $\n_-^*$.  It is then possible to prove that any set that consists of exactly one element of $\{\delta_{\omega_i, w_0 \omega_i}|_{\n}, \delta_{\omega_{i^*}, w_0 \omega_{i^*}}|_{\n}\}$ such that $\omega_i + \omega_{i^*}$ is primitive, and exactly one element of $\{\sqrt{\delta_{\omega_i, w_0 \omega_i}|_{\n}}, \sqrt{\delta_{\omega_{i^*}, w_0 \omega_{i^*}}|_{\n}}\}$ such that $\omega_i + \omega_{i^*}$ is not primitive, is a set of free generators of the Poisson center of $\CC[\n_-^*]$.  Details will be given in a separate paper.
\end{rem}

\subsection{The standard Poisson Lie group \texorpdfstring{$(G, \pist)$}{G}}\label{ss:pist}
Let $\langle\, , \, \rangle_\g$ be the symmetric and non-degenerate invariant bilinear form on $\g$ such that the induced non-degenerate symmetric bilinear form $\langle\,, \, \rangle$ on $\t^*$ has the 
property\footnote{This choice of $\langle\,, \, \rangle_\g$ is to ensure that when $G$ is simple
the Poisson structure on $G$ we define in this paper coincides with the one used in \cite{EL:BS}.} that
$\langle \alpha, \alpha \rangle =2$ for the short roots in $R_+$. Let $\{h_i: i \in [1, r]\}$ be an orthonormal
basis of $\t$ with respect to $\lara_\g|_\t$, and choose, for each $\alpha \in R_+$, 
root vectors $e_\alpha \in \g_\alpha$ and $e_{-\alpha} \in \g_{-\alpha}$ such that
$\alpha ([e_\alpha, e_{-\alpha}]) = 2$. One then has the quasi-triangular $r$-matrix \cite{Etingof-Schiffman, LM:mixed} 
\[
r_{\rm st} = \sum_{i=1}^r h_i \otimes h_i +\sum_{\alpha \in R_+} \langle \alpha, \, \alpha \rangle e_{-\alpha} \otimes e_\alpha
\in \g \otimes \g
\]
on $\g$ whose skew-symmetric part is given by 
\[
\Lambda_{\rm st} = \sum_{\alpha \in R_+} \frac{\langle \alpha, \alpha \rangle}{2} (e_{-\alpha} \otimes e_\alpha -
e_\alpha \otimes e_{-\alpha}) \in \wedge^2 \g.
\]
It is well-known \cite{dr:Hamil, Etingof-Schiffman} that $r_{\rm st}$ is 
independent of the choice of $\{h_i: i \in [1, r]\} \sqcup \{(e_\alpha, e_{-\alpha}): \alpha \in R_+\}$, and the bi-vector field
$\pist$ on $G$ given by
\[
\pist(g) = L_g r_{\rm st} - R_g r_{\rm st} = L_g \Lambda_{\rm st} - R_g \Lambda_{\rm st},
\]
is Poisson, 
where $L_g$ and $R_g$ for $g \in G$ denote the respective left and right translation on $G$ by $g$. Moreover, $\pist$ is 
multiplicative in the sense that the map
$(G, \pist) \times (G, \pist) \rightarrow (G, \pist), (g_1, g_2) \mapsto g_1g_2$,
is Poisson. The pair $(G, \pist)$ is also referred to as a {\it standard complex semi-simple Poisson Lie group}
\cite{LM:mixed}. 
See \cite{dr:Hamil, Etingof-Schiffman, LM:mixed} for the general theory on Poisson Lie groups.

\begin{example}\label{ex:SLn-GLn}
{\rm For $G=\SL(n, \CC) = \{X = (x_{i,j})_{i, j \in [1,n]}: \det X = 1\}$ and the 
standard choices of $B \subset G$ and $B_- \subset G$ consisting respectively of upper triangular
and lower triangular matrices in $G$, the Poisson structure $\pist$ on $\SL(n, \CC)$ is given, up to a non-zero scalar multiple,
by 
\[
\{x_{i,j}, \, x_{p,q}\}_{\pist} =({\rm sign}(p-i) + {\rm sign}(q-j)) x_{i,q}x_{p,j}, \hs i, j, p,q \in [1, n].
\]
The same formulas also define a Poisson structure on $\GL(n, \CC)$, which is 
also called the {\it standard multiplicative
Poisson structure} on $\GL(n, \CC)$.
}
\end{example}

Coming back to $(G, \pist)$, as ${\rm Ad}_t \Lambda_{\rm st} = \Lambda_{\rm st}$ for every $t \in T$,
the Poisson structure $\pist$ is invariant under the $T$-action by left (as well as  right) translation,
and the $T$-leaf
decomposition of $(G, \pist)$ is  given by \cite{KZ:leaves}
\[
G = \bigsqcup_{u, v \in W} G^{u, v},
\]
where $G^{u, v} = BuB \cap B_-vB_-$ is called \cite{FZ:double} the {\it double Bruhat cell} of $G$ associated to $u, v$. 
In particular, $(G, \pist)$ has an open $T$-leaf $G^{w_0, w_0}$, where again $w_0$ is the longest element of $W$.

Note that $\pist(e) = 0$, where $e$ is the identity element of $G$. To see the linearization of $\pist$ at $e$,
consider the direct product Lie algebra 
$\g \oplus \g$ with 
the non-degenerate invariant symmetric 
bilinear form\footnote{See \cite[$\S$3.2]{LM:T-leaves} for the explanation of the factor $\frac{1}{2}$.} 
\begin{equation}\label{eq:lara-gog}
\langle (x_1, x_2), \;(y_1, y_2) \rangle_{\g \oplus \g} 
=\frac{1}{2}( \langle x_1, y_1 \rangle_{\g} - \langle x_2, y_2 \rangle_{\g}), \hs x_1, x_2, y_1, y_2 \in \g.
\end{equation}
Let $\g_{\rm diag}
=\{(x, x): x \in \g\}$ be the diagonal Lie sub-algebra of $\g \oplus \g$, and identify $\g \cong \g_{\rm diag}$. Let 
\begin{equation}\label{eq:gsts}
\g_{\rm st}^*=\{(\xi_0 + \xi_+, \, -\xi_0 + \xi_-): \, \xi_0 \in \t,\; \xi_+ \in \n, \;\xi_- \in \n_-\} \subset \b \times \b_-.
\end{equation}
Then 
$((\g \oplus \g, \langle ~, ~ \rangle_{\g \oplus \g}), \g_{\rm diag}, \gsts)$ is a Manin triple \cite{dr:Hamil, LM:mixed}
associated to the Poisson Lie group $(G, \pist)$, and, under the identification 
$\g \cong  (\g_{\rm st}^{\ast})^{\ast}$ via 
\begin{equation}\label{eq:g-gsts}
\g \longrightarrow (\g_{\rm st}^{\ast})^{\ast}, ~ x \longmapsto \langle (x,x), - \rangle_{\g \oplus \g},
\end{equation}
the linearization $\pi_0$ of $\pist$ at $e$ coincides with the Kirillov-Kostant-Souriau Poisson structure 
$\pi_{0, \gsts}$ on $(\gsts)^*$ defined by 
the Lie algebra $\gsts$ (see \cite{dr:Hamil, LM:mixed, LM:T-leaves} for detail). We thus write
\[
(\g, \, \pi_0) \cong ((\gsts)^*, \; \pi_{0, \gsts}).
\]
Recall that $r = \dim \t$. 

\begin{prop}\label{prop:ind-gsts}
One has  $\ind(\gsts) = r$. Equivalently, ${\rm rk}(\pi_{0, \gsts}) = 2\ell(w_0)$.
\end{prop}

\begin{proof}
By \cite[Proposition 2.4]{KZ:leaves}, the rank of $\pist$ in the open double Bruhat cell 
$G^{w_0, w_0}$ in $G$ is $2\ell(w_0)$. By Lemma \ref{lem:rk-L-P},   
 the rank of $\pist$ in $G$ is also $2\ell(w_0)$. 
As $\pi_{0, \gsts}$ is the linearization of $\pist$ at $e$, we have ${\rm rk}(\pi_{0, \gsts}) \leq 2\ell(w_0)$, so
$\ind(\g_{\rm st}^*) \geq  r$. We now prove that 
$\ind(\g_{\rm st}^*) = r$ by finding an $x \in \g$ such that
$\dim ({\rm Stab}_{\g_{\rm st}^*}(x)) = r$, where recall that ${\rm Stab}_{\g_{\rm st}^*}(x)$ is the stabilizer Lie sub-algebra of
$\g_{\rm st}^*$ at $x$ for the co-adjoint representation of $\g_{\rm st}^*$ on $\g \cong (\g_{\rm st}^*)^*$.
By the ad-invariance of the bilinear form 
$\langle ~, ~ \rangle_{\g \oplus \g}$, for any $x \in \g$, we have
\begin{equation}\label{eq:stab-x}
{\rm Stab}_{\gsts}(x) = \{(\xi_1, \xi_2) \in \gsts: [(\xi_1, \xi_2), (x, x)] \in \gsts\}.
\end{equation}
Consider Kostant's cascade $\mathcal B \subseteq R_+$. 
Choose again root vectors
$e_{\pm \beta}$ for $\pm \beta$ and  let 
\begin{equation}\label{eq:e-pm}
    e_+ = \sum \limits_{\beta \in \mathcal B} e_{\beta}\in \n \quad \mbox{and} \quad  
    e_- = \sum \limits_{\beta \in \mathcal B} e_{-\beta} \in \n_-.
\end{equation}
We would deduce that $\dim (\Stab_{\g_{\rm st}^{\ast}}(e_+ + e_-)) = r$, and thus $\ind(\gsts) = r$, once we prove that
\begin{equation}\label{eq:stab-ee}
 \Stab_{\g_{\rm st}^{\ast}}(e_+ + e_-) =\{(\xi_0, -\xi_0):\, \xi_0 \in \t, \beta(\xi_0) = 0, \forall \, \beta \in \calB\}
 + {\rm Span}_{\CC}\{(e_\beta, e_{-\beta}): \beta\in \calB\}.
 \end{equation}
To prove \eqref{eq:stab-ee}, let $\xi_0 \in \t, \, \xi_+ \in \n, \, 
\xi_- \in \n_-$, and let $(\xi_1, \xi_2) = (\xi_0 + \xi_+, -\xi_0 +\xi_-)$. Since
\[
[(\xi_1, \xi_2), \; (e_++e_-, \, e_++e_-)] = ([\xi_1, e_+] + [\xi_1, e_-], \; [\xi_2, e_+] + [\xi_2, e_-]),
\]
by \eqref{eq:stab-x}, we have $(\xi_1, \xi_2) \in \Stab_{\g_{\rm st}^{\ast}}(e_+ + e_-)$ if and only if 
\[
[\xi_0 +\xi_+, \;e_-]\in \b, \hs [-\xi_0 + \xi_-,\; e_+]\in \b_-, \hs  [\xi_+, \;e_-] + [\xi_-,\; e_+] \in \n + \n_-.
\]
Write $\xi_+ = \sum \limits_{\alpha \in R_+} c_{\alpha} e_{\alpha}$ for some $c_{\alpha} \in \CC$. Then
\[
    [\xi_0 +\xi_+,\;  e_-] = - \sum \limits_{\beta \in \mathcal B} \beta(\xi_0) e_{- \beta} + 
    \sum \limits_{\substack {\alpha \in R_+ \\ \beta \in \mathcal B}} c_{\alpha} [e_{\alpha}, e_{- \beta}].
\]
As $\alpha -\beta \neq -\beta^\prime$ for any $\alpha \in R_+$ and
$\beta, \beta' \in \calB$, we have $[\xi_0 +\xi_+,\;  e_-] \in \b$ if and only if $\beta(\xi_0) = 0$ for every $\beta \in \calB$
and $[\xi_+,\;  e_-] \in \b$, while the latter,  by \cite[Theorem 2.3]{K12}, is equivalent to
$\xi_+ =\sum_{\beta \in \calB} c_{\beta} e_\beta$.
Similarly,  $[-\xi_0 + \xi_-,\; e_+]\in \b_-$ if and only if $\beta(\xi_0) = 0$ for every $\beta \in \calB$ and
$\xi_- = \sum_{\beta \in \calB} c_{-\beta} e_{-\beta}$ for some $c_{-\beta} \in \CC$ for each $\beta \in \calB$. 
Again using the fact that roots in $\calB$ are strongly orthogonal, we have 
\[
[\xi_+, \;e_-] + [\xi_-,\; e_+] = \sum_{\beta \in \calB}(c_\beta-c_{-\beta}) [e_\beta, \, e_{-\beta}],
\]
so $[\xi_+, \,e_-] + [\xi_-,\, e_+]\in \n + \n_-$ if and only if $c_\beta = c_{-\beta}$ for
every $\beta \in \calB$. We have thus proved   \eqref{eq:stab-ee}. 
\end{proof}

\section{Schubert cells in the flag variety \texorpdfstring{$G/B$}{G/B}} \label{s:GB}
\subsection{The standard Poisson structure on \texorpdfstring{$BwB/B$}{BwB/B}}\label{ss:piGB}
Let $G$ be a connected and simply connected complex semi-simple Lie group, equipped with the 
standard multiplicative Poisson structure $\pi_{\rm st}$ as $\S$\ref{ss:pist}, 
and let the notation be as in $\S$\ref{s:Lie-prelim}. 
 As $B$ is a Poisson Lie subgroup of $(G, \pist)$,
the Poisson structure $\pist$ on $G$ 
projects to a well-defined Poisson structure on the flag variety $G/B$, which we denote by $\piGB$ and call it the 
{\it standard Poisson structure on $G/B$}. Note that
$\piGB$ is $T$-invariant, where $T$ acts on $G/B$ by 
left translation. 
By \cite{GY09, LM:T-leaves}, the $T$-leaf decomposition of $(G/B, \piGB)$ is 
\[
G/B = \bigsqcup_{v, w \in W, v \leq w} R^w_v,
\]
where $\leq$ is the Bruhat order on $W$, and $R^w_v =(B_-vB/B) \cap (BwB/B)$, for $v \leq w$, is
called an {\it open Richardson variety}. 
It follows in particular that for each $w \in W$, the {\it Schubert cell} 
$BwB/B \subset G/B$
is a $T$-invariant Poisson sub-manifold of $(G/B, \piGB)$. For $w \in W$, set
\[
\pi_w = \piGB|_{BwB/B}.
\]
Then $(BwB/B, \pi_w)$ is a $T$-Poisson variety with a closed 
$T$-leaf $R^w_w$ consisting of the single point $\wB$ and 
an open $T$-leaf $R^w_e$, where $e$ again denotes the
identity element of $W$. 
In particular,
\[
\pi_w(w_\cdot B)=0.
\]
We will refer to $\pi_w$ as the {\it standard Poisson structure} on the Schubert cell
$BwB/B$ and denote by $\pi_{w, 0}$ the linearization of $\pi_w$ at $w_\cdot B \in BwB/B$.
As $w_\cdot B \in BwB/B$ is a $T$-fixed point, we know by 
Theorem \ref{thm:TT-fixed} that all $T$-log-canonical systems on $(BwB/B, \pi_w)$, if exist, 
have Property $\calI$ at $\wB$.

In the rest of $\S$\ref{s:GB}, we show that $T$-log-canonical systems on $(BwB/B, \pi_w)$ exist 
for every $w \in W$, and we review the {\it standard cluster structure} $\calC^w_{\rm st}$ on  $BwB/B$.  We then
study in details some integrable systems 
on $(T_{w_\cdot B}(BwB/B), \pi_{w, 0})$ coming from certain extended clusters of $\calC^w_{\rm st}$.
We first set up some notation.

Fix $w \in W$. Recall that 
$N_w = N \cap \ow N_- \ow^{\, -1}$. The unique decomposition $BwB = N_w \ow B$ 
 (see \cite[Proposition 2.9]{FZ:double}) then gives
 the isomorphism
\begin{equation}\label{eq:Nw-BwB}
N_w \longrightarrow BwB/B:\;\;   g \longmapsto g \ow_\cdot B.
\end{equation}

Let $w = s_{i_1} \cdots s_{i_l}$ be any reduced decomposition, where $l = \ell(w)$, and let 
$\bfw = (s_{i_1}, \ldots, s_{i_l})$, called a reduced word of $w$. 
For  $z = (z_1, \ldots, z_l)\in \CC^l$, set
\begin{equation}\label{eq:g-bfw}
g_\bfw(z) = 
e_{i_1}(z_1)\overline{s_{i_1}}\,e_{i_2}(z_2)\overline{s_{i_2}}\,\cdots \,e_{i_l}(z_l)\overline{s_{i_l}}
\in N_w \ow = N \ow \cap \ow N_-.
\end{equation}
One then has the {\it Bott-Samelson parametrizations} (\cite[Proposition 2.11]{FZ:double} and  \cite{EL:BS})
\begin{equation}\label{eq:BS-para}
{\mathbb{C}}^l \longrightarrow N_w: \;\; 
z \longmapsto g_\bfw(z) \ow^{\, -1} \hs \mbox{and} \hs 
{\mathbb{C}}^l \longrightarrow BwB/B:\;\;  z  \longmapsto g_\bfw(z)_\cdot B.
\end{equation}
The resulting coordinates $(z_1, \ldots, z_l)$ on $N_w$ (resp. on $BwB/B$) will be called  
the {\it Bott-Samelson coordinates} on $N_w$ (resp. on $BwB/B$) associated to $\bfw$.
Let $T$ act on $\CC[N_w]$ and $\CC[BwB/B]$ respectively by 
\begin{align}\label{eq:T-on-phi-1}
(t\cdot \varphi)(g) &= \varphi(tgt^{-1}), \hs t \in T, \; \varphi \in \CC[N_w], \; g \in N_w, \\
\label{eq:T-on-phi-0}
(t\cdot \varphi)(g_\cdot B)& = \varphi(tg \ow_\cdot B), \hs t \in T, \; \varphi \in \CC[BwB/B], \; g \in B.
\end{align}
Then  
$z_k \in \CC[N_w]$ (resp. $z_k \in\CC[BwB/B]$) is a $T$-weight vector with $T$-weight 
\begin{equation}\label{eq:beta-bfw-k}
\beta_{\bfw, k} =s_{i_1}\cdots s_{i_{k-1}} \alpha_{i_k} \in R_+, \hs k \in [1, n].
\end{equation}
We will also regard $\pi_w$ as a Poisson structure on $N_w$ via the isomorphism $N_w \to BwB/B$ in \eqref{eq:Nw-BwB}
and still denote it as $\pi_w$. For a reduced word $\bfw$ of $w$, the induced Poisson
structure on $\CC^l$ via the Bott-Samelson parametrizations in \eqref{eq:BS-para} will also be denoted as $\pi_{w}$.

For any reduced word $\bfw$ of $w$, the Poisson brackets 
$\{z_j, z_k\}_{\pi_w}$ between the
Bott-Samelson coordinates $(z_1, \ldots, z_l)$
are computed in \cite{EL:BS}.  More precisely, for any $1 \leq j < k \leq l$, one has 
\begin{equation}\label{eq:zjk-piw}
\{z_j, \, z_k\}_{\pi_w} = -\langle \beta_{\bfw, j}, \, \beta_{\bfw, k} \rangle z_j z_k + f_{j, k},
\end{equation}
where $f_{j, k} \in \CC[z_{j+1}, \ldots, z_{k-1}]$ (we set $\CC[z_{j+1}, \ldots, z_{k-1}] = 0$ when $k = j+1$),
and the explicit formula for $f_{j, k}$ are given in \cite[Theorem 5.15]{EL:BS} 
in terms of certain root strings and structure constants of $\g$. 

For $\h \in \t$,  let $\partial_h$ be the derivation of $\CC[z_1, \ldots, z_l]$ such that
$\partial_h (z_j) = \beta_{\bfw, j}(h) z_j$ for $j \in [1, l]$.
For  $k \in [1, l]$, let
$h_k \in \t$ be such that $\lambda(h_k) = \langle \lambda, \, \beta_{\bfw, k}\rangle$ for $\lambda \in \h^*$, and
let $\delta_k$ be the derivation on $\CC[z_1, \ldots, z_{k-1}]$ given by
\[
\delta_k(z_j) = -f_{j, k} \in \CC[z_{j+1}, \ldots, z_{k-1}] \subset \CC[z_1, \ldots, z_{k-1}].
\] 
Then \eqref{eq:zjk-piw}  can be re-written as
(see also 
\cite[Theorem 5.12]{EL:BS}) 
\begin{equation}\label{eq:zk-a}
\{z_k, \, a\}_{\pi_w}  =z_k\partial_{h_k} (a) + \delta_k(a), \hs k \in [1, l],\, \; a \in \CC[z_1, \ldots, z_{k-1}],
\hs k \in [2, l],
\end{equation}
which, by \cite[Definition 6.1]{GY:Poi-CGL},  says that 
the Poisson algebra $(\CC[BwB/B], \{\, ,\, \}_{\pi_w})$ is now presented as a 
{\it symmetric Poisson CGL extension} in the Bott-Samelson coordinates 
$(z_1, \ldots, z_l)$
(see also \cite[Proposition 4.26]{L} for the case of $w = w_0$).

\begin{rem}\label{rem:delta-k}

For $k \in [1, l]$, let $w_{k-1} = s_{i_1} \cdots s_{i_{k-1}}$, and regard $(z_1, \ldots, z_{k-1})$
as Bott-Samelson coordinates on $Bw_{k-1}B/B$ defined by the reduced word 
${\bf w}_{k-1} = (s_{i_1}, \ldots, s_{i_{k-1}})$ of $w_{k-1}$. 
The derivation $\delta_k$ on $\CC[z_1,\ldots, z_{k-1}]$, regarded as a vector field on 
$Bw_{k-1}B/B$, has the following interpretation 
given in \cite[(92)]{EL:BS}:  parameterizing both $Bw_{k-1}B/B$ and $B_-\backslash B_-w_{k-1}B_-$
by $N\overline{w_{k-1}} \cap \overline{w_{k-1}}\, N_-$, one has the 
isomorphism 
\[
\nu_k: \;\; Bw_{k-1}B/B \longrightarrow B_-\backslash B_-w_{k-1}B_-, \;\; g_\cdot B \longmapsto B_-{}_\cdot g, \hs
g \in N \,\overline{w_{k-1}} \cap \overline{w_{k-1}}\, N_-.
\]
Then $\delta_k$ is the vector field on $Bw_{k-1}B/B$ generating the induced right action of $B_-$ on
$Bw_{k-1}B/B$ in the direction of $-\langle \alpha_{i_k}, \alpha_{i_k}\rangle e_{-i_k} \in \n_-$, i.e., 
\begin{equation}\label{eq:delta-k-vector-field}
\delta_k(g_\cdot B) = -\langle \alpha_{i_k}, \alpha_{i_k}\rangle  
\frac{d}{dc}|_{c=0} \nu_k^{-1} \left(B_-{}_\cdot g \,
e_{-i_k}(c)\right)
\in T_{g_\cdot B} (Bw_{k-1}B/B), 
\end{equation}
for $g \in N \,\overline{w_{k-1}} \cap \overline{w_{k-1}}\, N_-$.
\end{rem}


\subsection{The standard cluster structure \texorpdfstring{$\calC^w_{\rm st}$}{Cwst} on \texorpdfstring{$BwB/B$}{BwB/B} and the \texorpdfstring{$T$}{T}-Pfaffian of \texorpdfstring{$\pi_w$}{piw}}
\label{ss:st-cluster-BwB}
Let again $w \in W$ and let $\bfw = (s_{i_1}, \ldots, s_{i_l})$ be a reduced word of $w$. 
For $k \in [1, l]$, 
denote by $\vphi_{\bfw, k}$ both the function on $N_w$ and 
the function on $BwB/B$, given by
\begin{equation}\label{eq:vphi-bfw-k}
\vphi_{\bfw, k}(g) = \vphi_{\bfw, k} (g\wB) = \Delta_{\omega_{i_k}, \, s_{i_1} \cdots s_{i_k} \omega_{i_k}}(g), \hs g \in N_w.
\end{equation}
 By further abusing notation, for $k \in [1, l]$, we also set
 \begin{equation}\label{eq:vphi-k-z}
\vphi_{\bfw, k}(z) = \vphi_{\bfw, k} ((e_{i_1}(z_1)\overline{s_{i_1}}\,\cdots\, e_{i_l}(z_l)\overline{s_{i_l}})_\cdot B) \in \CC[z_1, \ldots, z_l].
 \end{equation}
For $k \in [1,l]$ and $z = (z_1, \ldots, z_l) \in\CC^l$, since 
$e_{i_{k+1}}(z_{k+1})\overline{s_{i_{k+1}}}\cdots e_{i_l}(z_l)\overline{s_{i_l}}\in N\, \overline{s_{i_{k+1}}\cdots s_{i_l}}$, one has
\begin{equation}\label{eq:useful}
\vphi_{\bfw, k}(z) 
=\Delta_{\omega_{i_k}, \, s_{i_1}\cdots s_{i_k}\omega_{i_k}} (e_{i_1}(z_1)\overline{s_{i_1}}\,\cdots \,e_{i_l}(z_l)\overline{s_{i_l}}\,\ow^{\, -1})\\
=\Delta_{\omega_{i_k}, \, \omega_{i_k}} (e_{i_1}(z_1)\overline{s_{i_1}}\,
\cdots \,e_{i_k}(z_k)\overline{s_{i_k}}).
\end{equation}
For $k \in [1, l]$, set $k^- = \max\{j \in [1,k-1]: s_{i_j} = s_{i_k}\}$
if $\{j \in [1,k-1]: s_{i_j} = s_{i_k}\}\neq \emptyset$, and $k^-=-\infty$ and $\vphi_{\bfw, k^-}=1$ otherwise.
By \eqref{eq:gqq}, for each $k \in [1, l]$, one has
\begin{equation}\label{eq:phik-CGL}
\vphi_{\bfw, k}(z) = z_k \vphi_{\bfw, k^-}+
\Delta_{\omega_{i_k}, \, s_{i_k}\omega_{i_k}} (e_{i_1}(z_1)\overline{s_{i_1}}\,\cdots\,
e_{i_{k-1}}(z_{k-1})\overline{s_{i_{k-1}}}).
\end{equation}
Moreover,  $\vphi_{\bfw, k} \in \CC[BwB/B]$ is a $T$-weight vector with
$T$-weight  (see also \cite[Lemma 2.18]{LMi:Kostant})
\[
\chi_{\vphi_{\bfw, k}}=\sum_{j \in [1, k], \,i_j = i_k} \beta_{\bfw, j} = \omega_{i_k}-s_{i_1}\cdots s_{i_k}\omega_{i_k}.
\]
We also note that \eqref{eq:phik-CGL} can be re-written as 
\begin{equation}\label{eq:phi-k-delta-k}
\vphi_{\bfw, k} = z_k \vphi_{\bfw, k^-} -
\frac{\delta_k(\vphi_{\bfw, k^-})}{\langle \alpha_{i_k}, \alpha_{i_k}\rangle},
\hs k \in [1, l].
\end{equation}
Indeed, for $c \in \CC$ and $(z_1, \ldots, z_{k-1}) \in \CC^{k-1}$, we have (see Remark \ref{rem:delta-k})
\[
e_{i_1}(z_1)\overline{s_{i_1}} \cdots e_{i_{k-1}}(z_{k-1})\overline{s_{i_{k-1}}}  \, e_{-i_k}(c) =
n_-(c) \, e_{i_1}(z_1(c))\overline{s_{i_1}} \cdots e_{i_{k-1}}(z_{k-1}(c))\overline{s_{i_{k-1}}}
\]
for some $n_-(c) \in N_-$ and $z_1(c), \ldots, z_{k-1}(c) \in \CC$, 
and by \eqref{eq:delta-k-vector-field} and \eqref{eq:gp}, we have
\begin{align*}
-\frac{\delta_k(\vphi_{\bfw, k^-})}{\langle \alpha_{i_k}, \alpha_{i_k}\rangle}
&= \frac{d}{dc}|_{c=0} 
\Delta_{\omega_{i_k}, \omega_{i_k}}(e_{i_1}(z_1(c))\overline{s_{i_1}} 
\cdots e_{i_{k-1}}(z_{k-1}(c))\overline{s_{i_{k-1}}})\\
& = \frac{d}{dc}|_{c=0}\Delta_{\omega_{i_k}, \omega_{i_k}} 
(e_{i_1}(z_1)\overline{s_{i_1}} \cdots e_{i_{k-1}}(z_{k-1})\overline{s_{i_{k-1}}}  \, e_{-i_k}(c)) \\
& = \Delta_{\omega_{i_k}, \, s_{i_k}\omega_{i_k}} (e_{i_1}(z_1)\overline{s_{i_1}}\,e_{i_2}(z_2)\overline{s_{i_2}}\,\cdots\,
e_{i_{k-1}}(z_{k-1})\overline{s_{i_{k-1}}}).
\end{align*}
Taking $a = \vphi_{\bfw, k^-}$ in \eqref{eq:zk-a}, we get the following recursive formula for the $\vphi_{\bfw, k}$'s 
using $\{\, , \, \}_{\pi_w}$:
\begin{equation}\label{eq:vphik-recursive}
\langle \alpha_{i_k}, \, \alpha_{i_k}\rangle \vphi_{\bfw, k} = \{\vphi_{\bfw, k^-}, \, z_k\}_{\pi_w} +
\langle s_{i_1}\cdots s_{i_{k-1}}\alpha_{i_k}, \;\omega_{i_k}-s_{i_1}\cdots s_{i_k}\omega_{i_k}\rangle z_k\vphi_{\bfw, k^-}.
\end{equation}

\begin{lem}\label{lem:Phiw-T-log}
For any $w \in W$ with $l = \ell(w)$ and any reduced word $\bfw$ of $w$, the set 
\[
\Phi_\bfw =  \{\vphi_{\bfw, k} \in \CC[BwB/B]: k \in [1, l]\}
\]
is a $T$-log-canonical system on $(BwB/B, \pi_w)$.
\end{lem}

\begin{proof} By \eqref{eq:phik-CGL}, 
the set $\Phi_{\bfw}$ is independent, and we already know that each $\vphi_{\bfw, k}$ is a $T$-weight vector. 
By \cite[Corollary 2.15]{LMi:Kostant}, for
$1 \leq j \leq k \leq l$ and for any $f \in \CC[z_1, \ldots, z_j] \subset \CC[z_1, \ldots, z_l]$ which is a $T$-weight vector with $T$-weight $\chi_f$, one has 
\begin{equation}\label{eq:f-vphi}
\{f, \, \vphi_{\bfw, k}\}_{\pi_w} = \langle \chi_f, \,  \omega_{i_k}+s_{i_1}\cdots s_{i_k}\omega_{i_k}\rangle f\vphi_{\bfw, k}.
\end{equation}
In particular, for all $1 \leq j < k \leq l$ one has
\[
\{\vphi_{\bfw, j}, \, \vphi_{\bfw, k}\}_{\pi_w} = \lambda_{j,k} \vphi_{\bfw, j} \vphi_{\bfw, k},
\]
where $\lambda_{j,k} = \langle \omega_{i_j}-s_{i_1}\cdots s_{i_j}\omega_{i_j}, \;
\omega_{i_k}+s_{i_1}\cdots s_{i_k}\omega_{i_k}\rangle$. Thus $\Phi_\bfw$ is $T$-log-canonical with respect to $\pi_w$.
\end{proof}

\begin{cor}\label{cor:Cw-all}
For every $w \in W$, $T$-log-canonical systems  on $(BwB/B, \pi_w)$ exist and all of them have Property $\calI$
at $\wB$.
\end{cor}

\begin{proof}
The statement is a direct consequence of Lemma \ref{lem:Phiw-T-log} and Theorem \ref{thm:TT-fixed}.
\end{proof}
Consider now the sequence of Poisson sub-algebras of 
$(\CC[z_1, \ldots, z_l],\, \{\, , \, \}_{\pi_w})$
\[
R_k = \CC[z_1, \ldots, z_k] \subset R_l = \CC[z_1, \ldots, z_l], \hs k \in [1, l].
\]
The sequence 
$(\vphi_{\bfw, 1}, \ldots, \vphi_{\bfw, l}\}$ of elements in $\CC[z_1, \ldots, z_l]$ then has the 
following properties:

1) for each $k \in [1, l]$, $\vphi_{\bfw, k}$ is a $T$-weight vector and a prime element in $R_k$ which is not contained in $R_{k-1}$;

2) for each $k \in [1, l]$, $\vphi_{\bfw, k}$ a Poisson element of $R_k$ in the sense that 
$\{\varphi_{\bfw, k}, R_k\}_{\pi_w} \subset \varphi_{\bfw, k} R_k$, as seen by taking $f = z_j$ in \eqref{eq:f-vphi} for all $j \in [1, k]$.

In the 
terminology of \cite[Theorem 5.5]{GY:Poi-CGL}, $(\vphi_{\bfw, 1}, \ldots, \vphi_{\bfw, l})$ is the sequence of
{\it homogeneous Poisson prime elements} associated to the symmetric Poisson CGL extension 
$(\CC[z_1, \ldots, z_l], \, \{\, , \, \}_{\pi_w})$. 
One can further show \cite{Lu-Mi:matrices} that the symmetric Poisson CGL extension 
$(\CC[z_1, \ldots, z_l], \, \{\, , \, \}_{\pi_w})$ is normal in the sense of \cite[Chapter 9]{GY:Poi-CGL}.
Set  ${\rm ex} = \{k \in [1, l]: \, k^+ \neq +\infty\}$, where 
similar to the definition of $k^-$ for $k \in [1, l]$, 
\begin{equation}\label{eq:k-plus}
k^+ = \begin{cases} \min\{j \in [k+1,l]: s_{i_j} = s_{i_k}\}, & \hs \mbox{if}\;\; \{j \in [k+1,l]: s_{i_j} = s_{i_k}\}\neq \emptyset,\\
+\infty, & \hs {\rm otherwise}.\end{cases}
\end{equation}
By \cite[Theorem 11.1]{GY:Poi-CGL}, there is a unique mutation matrix $M$, with rows 
indexed by $i \in [1, l]$ and columns by $k \in {\rm ex}$, such that the mutation equivalence class 
$[(\Phi_\bfw, \, {\rm ex}, \, M)]$ 
of seeds in $\CC(BwB/B)$ is compatible with both the
Poisson structure $\pi_w$ (Definition \ref{defn:regular}) and the $T$-action (Remark \ref{rem:T-comp}), and
that $[(\Phi_\bfw, \, {\rm ex}, \, M)]$ is both a cluster and and upper cluster structure on $BwB/B$. 
One also knows (see \cite[Paragraph after Theorem 1.1]{ShenWeng:DBS}) that $[(\Phi_\bfw, \, {\rm ex}, \, M)]$ is independent of the choice of the reduced word $\bfw$ of $w$. We set
\[
\calC^w_{\rm st} = [(\Phi_\bfw, \, {\rm ex}, \, M)]
\]
using any reduced word $\bfw$ of $w$ and refer to $\calC^w_{\rm st}$  as the  {\it standard cluster structure on $BwB/B$}. 
Thus all the extended clusters of $\calC^w_{\rm st}$ are
$T$-log-canonical systems on $(BwB/B, \pi_w)$. By \cite[Theorem 11.1]{GY:Poi-CGL}, all the Bott-Samelson coordinates
associated to all reduced words of $w$ are cluster or frozen variables of $\calC^w_{\rm st}$.

\begin{rem}\label{rem:GLS-KM}
When $G$ is a 
Kac-Moody group of simply-laced type,
the
cluster structure $\calC_{\rm st}^w$ was first studied  in \cite{GLS:partial}. See also 
\cite{ShenWeng:DBS}.
\end{rem}

We now have the following immediate consequence of Corollary \ref{cor:pi-T-I}.

\begin{thm}\label{thm:main-GB}
For every $w \in W$,  the standard
cluster structure $\calC_{\rm st}^w$ on $BwB/B$ has, without any frozen variable modification, Property $\calI$ at $\wB$.
\end{thm}

 As ${\rm ex} = \{k \in [1, l]: k^+ \in [1, l]\}$, the set of frozen variables of $\calC_{\rm st}^w$ is
\begin{equation}\label{eq:frozen-0}
{\rm Froz}(\calC_{\rm st}^w) =\{\vphi_{\bfw, k}: k^+=+\infty\} = \{\vphi_{\bfw, k}: k \in [1, l], k \neq j^- \,
\mbox{for any}\; j \in [1, l]\}.
\end{equation}
To describe ${\rm Froz}(\calC_{\rm st}^w)$ without using reduced words of $w$ nor any listing 
of the set $\Gamma$ of simple roots, let
\[
{\rm supp}(w) = \{\alpha \in \Gamma: s_\alpha \leq w\} = \{\alpha \in \Gamma: w\omega_{\alpha} \neq \omega_\alpha\},
\]
where $\leq$ is the Bruhat order on $W$ and 
 $\omega_\alpha$ for $\alpha \in \Gamma$ is the fundamental weight corresponding to $\alpha$. For $\alpha \in {\rm supp}(w)$, 
denote by $\vphi_{w, \alpha}$ 
both the  function on
$N_w$ and the function on $BwB/B$ given by
\begin{equation}\label{eq:vphi-w-alpha}
\vphi_{w, \alpha}(g) = \vphi_{w, \alpha}(g\wB) = \Delta_{\omega_\alpha, \, w\omega_\alpha}(g), \hs
g \in N_w.
\end{equation}
If $\bfw = (s_{i_1}, \ldots, s_{i_l})$ is a reduced word of $w$, and if 
$k_\alpha = {\rm max}\{k \in [1, l]: \alpha_{i_k} = \alpha\}$ for $\alpha \in {\rm supp}(w)$, 
then
$\vphi_{\bfw, k_\alpha} = \vphi_{w, \alpha} \in \CC[BwB/B]$ by Lemma \ref{lem:gig}. As
$\{k_\alpha: \alpha \in {\rm supp}(w)\} = [1, l]\backslash {\rm ex}$, we have
\begin{equation}\label{eq:frozen}
{\rm Froz}(\calC_{\rm st}^w) = \{\vphi_{w, \alpha} \in \CC[BwB/B]: \;\alpha \in {\rm supp}(w)\}.
\end{equation}

The following statement will be used in $\S$\ref{ss:Thimm-flow}.

\begin{lem} \label{lem:low-of-frozen}
    For any $w \in W$ and $\alpha \in \supp(w)$, the lowest degree term 
    $\vphi_{w, \alpha}^\low$ of $\vphi_{w, \alpha}$ at $\wB$ is a Casimir function on the linearization 
    $(T_{w_\cdot B}(BwB/B), \pi_{w, 0})$ of 
    $(BwB/B, \pi_w)$ at $\wB$.
\end{lem}

\begin{proof} Take any reduced word of $w$ and
let $(z_1, \ldots, z_k)$ be the corresponding Bott-Samelson coordinates on $BwB/B$.
Taking $f = z_k$ in \eqref{eq:f-vphi} for every $k \in [1, l]$, one has 
$\{z_k, \, \vphi_{w, \alpha}\}_{\pi_w} \in \CC\vphi_{w, \alpha} z_k$. It follows that 
$\{\vphi_{w,\alpha}^\low, z_k^\low\}_{\pi_{w, 0}} = 0$.
\end{proof}

We now turn to the log-volume form $\mu_{\calC^w_{\rm st}}$, defined up to the multiplication by $-1$ 
(see Lemma-Definition \ref{lem:same-volume}), and the 
$T$-Pfaffian of $\pi_w$.

\begin{cor} \label{cor:TPf}
Let $w \in W$ with $l = \ell(w)$ and let $(z_1, \ldots, z_l)$ be the Bott-Samelson coordinates on $BwB/B$ associated to any reduced 
word ${\bf w}$ of $w$. Then, up to multiplication by non-zero scalars, one has
\begin{equation}\label{eq:mu-phi-zzz}
\mu_{\calC^w_{\rm st}} = 
\left( \prod_{\alpha \in {\rm supp}(w)} \vphi_{w, \alpha}^{-1} \right) dz_1 \wedge \cdots \wedge dz_l, \;\;\; \hs
\Pf_T(\pi_w) =\left( \prod_{\alpha \in {\rm supp}(w)} \vphi_{w, \alpha} \right)\frac{\partial}{\partial z_1} \wedge \cdots \wedge \frac{\partial}{\partial z_l}.
\end{equation}
\end{cor}

\begin{proof} 
    By \eqref{eq:phik-CGL},  the Jacobi matrix 
$\frac{\partial (\varphi_{\bfw,1}, \ldots, \varphi_{\bfw, l})}{\partial (z_1, \ldots, z_l)}$
is upper triangular with $\varphi_{\bfw, k^-}$ as its $(k,k)$-entry for $k \in [1, l]$ 
(recall that $\varphi_{\bfw, k^-} = 1$ if $k^- =-\infty$).  The identity for $\mu_{\calC^w_{\rm st}}$
follows from Lemma-Definition \ref{lem:same-volume} and \eqref{eq:frozen}, and that for $\Pf_T(\pi_w)$
follows from Lemma \ref{lem:mu-Pfaff}.
\end{proof}



\subsection{The linearization  of \texorpdfstring{$(BwB/B, \pi_w)$}{(Bwb/B,piw)} at \texorpdfstring{$\wB$}{wB}}\label{ss:linear-Cw} Fix $w \in W$. Set again 
\[
\n_w = \n \cap \Ad_{\ow}\n_- \hand \n_{w, -} = \n_- \cap \Ad_{\ow}\n,
\]
and recall that $N_w = N \cap \ow N_- \ow^{\, -1}$. The isomorphism $N_w \to BwB/B$ in \eqref{eq:Nw-BwB} gives the identification
\begin{equation}\label{eq:nw-tangent}
\n_w \cong T_{w_\cdot B}(BwB/B).
\end{equation}

\begin{nota}\label{nota:nw-pi0} Recall that $\pi_{w, 0}$ denotes the linearization of 
$\pi_w$  at $\wB$. 
Using the identification in \eqref{eq:nw-tangent}, we also regard $\pi_{w, 0}$ as 
 a linear Poisson structure on $\n_w$ and still denote it by $\pi_{w, 0}$.  
\end{nota}

On the other hand, using the non-degenerate symmetric bilinear form 
$\lara_\g$ on $\g$ fixed at the beginning of $\S$\ref{ss:pist}, we have  the vector space isomorphism
\[
\zeta_w:\;\; \n_w \longrightarrow \n_{w, -}^*, \;\; \zeta_w (x) = \langle x, \, -\rangle_\g.
\]
Let $\pi_{0, \n_{w, -}}$ be the Kirillov-Kostant-Souriau  Poisson structure on 
$\n_{w, -}^*$ defined by the Lie algebra $\n_{w, -}$.

\begin{lem}\label{lem:pi0-nw} One has the Poisson isomorphism  
$\zeta_w: (\n_w, \, \pi_{w, 0})  \rightarrow (\n_{w, -}^*, \,  -2\pi_{0, \n_{w, -}})$.
\end{lem}

\begin{proof}
We first identify the Lie bialgebra \cite{LM:mixed} of the Poisson Lie subgroup $(B, \pist)$ of $(G, \pist)$.
With respect to the non-degenerate pairing  between $\g_{\rm diag}$ and $\gsts$ via 
the symmetric bilinear form $\lara_{\gog}$ on $\gog$ given in \eqref{eq:lara-gog},
the annihilator of $\b \cong \b_{\rm diag}$ in $\gsts$ is the Lie ideal $\n \oplus 0$ of $\gsts$, 
and the embedding 
\[
\b_- \longrightarrow \gsts, \;\; \xi_0 +\xi_- \longmapsto (-\xi_0, \; \xi_0 + \xi_-), \hs \xi_0 \in \t, \, \xi_- \in \n_-,
\]
gives a Lie algebra isomorphism $\b_- \cong \gsts/(\n\oplus 0)$. The pairing  between $\g_{\rm diag}$ and $\gsts$ via $\lara_{\gog}$ restricts to a non-degenerate pairing 
$\langle\, , \, \rangle_{(\b, \b_-)}$ between $\b\cong \b_{\rm diag} \subset \g_{\rm diag}$ and $\b_-\hookrightarrow \gsts$ 
given by
\begin{equation}\label{eq:bb-pair}
\langle x_+ + x_0, \; \xi_0 + \xi_-\rangle_{(\b, \b_-)} 
=-\frac{1}{2}\langle x_+, \; \xi_-\rangle_\g - \langle x_0, \;\xi_0\rangle_\g, \hs x_+ \in \n,\,
\, x_0, \,\xi_0 \in \t, \, \xi_- \in \n_.
\end{equation}
It follows that under the identification $\b^*\cong \b_-$ via $\langle\, , \, \rangle_{(\b, \b_-)}$, the Lie bialgebra of the Poisson Lie group $(B, \pist)$ is the pair $(\b, \b_-)$. Consequently, under the identification $\b \cong \b_-^*$ via 
$\langle\, , \, \rangle_{(\b, \b_-)}$, 
the linearization of $(B, \pist)$ at $e$ becomes
$(\b_-^*, \pi_{0, \b_-})$, where recall from Notation-Definition \ref{nota:pi0-l} that
$\pi_{0, \b_-}$ is the Kirillov-Kostant-Souriau  Poisson structure on
$\b_-^*$ defined by the Lie algebra $\b_-$. Consider now the Poisson map
\[
(B, \,\pist) \longrightarrow (BwB/B, \, \pi_w), \;\; b \longmapsto b\wB,
\]
whose differential at $e \in B$ gives a linear Poisson map $(\b, \pi_{0, \b_-}) \to (\n_w, \pi_{w, 0})$. 
Thus under the identification of $\n_w$ with $\n_{w, -}^*$ using the pairing
$-\frac{1}{2} \lara_\g$ between $\n_w$ and $\n_{w, -}$, the linear Poisson structure $\pi_{w, 0}$ on $\n_w$ 
becomes the Kirillov-Kostant-Souriau  Poisson structure $\pi_{0, \n_{w, -}}$ on
$\n_{w, -}^*$, i.e., for $\xi_-, \eta_- \in \n_{w, -}$, one has 
\[
\left\{-\frac{1}{2} \zeta_w^*(\xi_-), \; -\frac{1}{2} \zeta_w^*(\eta_-)\right\}_{\pi_{w, 0}} 
= -\frac{1}{2} \zeta_w^*([\xi_-, \eta_-]) \in 
\n_w^*,
\]
which gives $\{\zeta_w^*(\xi_-), \;\zeta_w^*(\eta_-)\}_{\pi_{w, 0}} = -2 \zeta_w^*([\xi_-, \eta_-])$
for $\xi_-, \eta_-\in \n_-$. 
Thus $\zeta_w(\pi_{w, 0}) = -2\pi_{0, \n_{w, -}}$.
\end{proof}

\begin{example}\label{ex:nn}
When $w = w_0$, the longest element in $W$, the linearization of $\pi_{w_0}$ at 
${w_0}_\cdot B \in G/B$ is then $-2$ times the
Kirillov-Kostant-Souriau  Poisson structure $\pi_{0, \n_-}$ on $\n \cong \n_-^*$.
\end{example}

\begin{cor}\label{cor:int-sys-GB}
For every extended cluster $\Phi$ of the standard cluster structure $\calC_{\rm st}^w$ on $BwB/B$, the set 
$\Phi^\low$ of lowest degree terms of elements in $\Phi$ at $\wB \in BwB/B$ contains a homogeneous polynomial integrable
system on $(T_{\wB}(BwB/B), \pi_{w, 0})\cong (\n_w, \pi_{w, 0}) \cong (\n_{w, -}^*, -2\pi_{0, \n_{w, -}})$.
\end{cor}

\begin{rem}\label{rem:different-0} 
 The integrable systems on $(\n_w, \pi_{w, 0})$  obtained from different extended clusters of $\calC_{\rm st}^w$
    are in general different.  In fact, the Poisson commutative sub-algebras of 
    $(\CC[\n_w], \{\, , \, \}_{\pi_{w,0}})$ generated by $\Phi^\low$ for different extended clusters 
    $\Phi$ of $\calC_{\rm st}^w$ are in general different.  To see this, recall 
    that  the Bott-Samelson coordinates $z_1, \ldots, z_l$ 
    are all cluster or frozen variables of $\calC_{\rm st}^w$, hence the sub-algebra of $\CC[\n_w]$ generated by all the $\Phi^\low$, where $\Phi$ ranges over all extended clusters, is $\CC[\n_w]$, which in general is not Poisson commutative.  
\end{rem}

\begin{nota-rem}\label{nota-rem:BS-linear}
For $w \in W$, recall that $\pi_w$ also denotes the induced Poisson structure on $N_w$ via the isomorphism $N_w \to BwB/B$ 
in \eqref{eq:Nw-BwB}. 
The linearization of $(N_w, \pi_w)$ at $e \in N_w$ is then $(\n_w, \pi_{w, 0})$.
Let $\bfw = (s_{i_1}, \ldots, s_{i_l})$ be any reduced word of $w$. By viewing the Bott-Samelson coordinates on $BwB/B$ associated to $\bfw$ as coordinates on $N_w$ via the parametrization 
\begin{equation}\label{eq:BS-Nw}
\CC^l \ni z = (z_1, \ldots, z_l) \longmapsto g(z) 
=(e_{i_1}(z_1)\overline{s_{i_1}}\,e_{i_2}(z_2)\overline{s_{i_2}}\,\cdots \,e_{i_l}(z_l)\overline{s_{i_l}}) \overline{w}^{\, -1},
\end{equation}
for every pair $1 \leq j < k \leq l$, we then have $\{z_j, z_k\}_{\pi_w} \in \CC[z_1, \ldots, z_l]$ with $0$ constant term.
As in Remark \ref{rem:v-z}, we will abuse notation and denote the linear function $d_ez_k =z_k^\low$
on $T_eN_w = \n_w$ also as $z_k$ for $k \in [1, l]$. 
For $1 \leq j < k \leq l$,
the linear function $\{z_j, z_k\}_{\pi_{w, 0}}$ on $\n_w$ is then the linear term of the polynomial 
$\{z_j, z_k\}_{\pi_w}$.
To describe $\{z_j, z_k\}_{\pi_{w, 0}}$ using  Lemma \ref{lem:pi0-nw}, 
for $k \in [1, l]$, let $w_k = s_{i_1} \cdots s_{i_k} \in W$, let
\[
 E_{\bfw, k} = \Ad_{\overline{w_{k-1}}} e_{i_k}\in \g_{\beta_{\bfw, k}}
\]
and  let
$E_{{\bfw, k, -}} \in \g_{-\beta_{\bfw, k}}$ be such that 
$\langle E_{\bfw, k}, \, E_{\bfw, k,-}\rangle_\g = 1$, where $\beta_{\bfw, k}$ is given in
\eqref{eq:beta-bfw-k}. 
Then $\{E_{{\bfw, k}}: k \in [1, l]\}$ is a basis of
$\n_w$, which is dual to the basis $\{E_{{\bfw, k,-}}: k \in [1, l]\}$ of $\n_{w, -}$ under the pairing $\lara_\g$
between $\n_w$ and $\n_{w, -}$. On the other hand, as 
\[
g(z) =(e_{i_1}(z_1)\overline{s_{i_1}}\,e_{i_2}(z_2)\overline{s_{i_2}}\,\cdots \,
e_{i_l}(z_l)\overline{s_{i_l}})\, \overline{w}^{\, -1} 
= \exp(z_1 E_{\bfw, 1}) \exp(z_2 E_{\bfw, 2}) \cdots \exp(z_l E_{\bfw, l}),
\]
the element $d_ez_k \in T_e^*N_w$ for $k \in [1, l]$ 
corresponds to $E_{\bfw, k, -}$ under the identification $T_e^*N_w = \n_w^* \cong \n_{w, -}$. 
It thus follows from Lemma \ref{lem:pi0-nw} that for $1 \leq j < k \leq l$,
\begin{equation}\label{eq:zjk}
\{z_j, z_k\}_{\pi_{w, 0}} =\begin{cases} -2 c_{j, k} z_{p(j, k)}, &\hs \mbox{if}\; \beta_{\bfw, j} + \beta_{\bfw, k} \;\mbox{is a root},\\
0, & \hs \mbox{otherwise},\end{cases}
\end{equation}
where  when
$\beta_{\bfw, j} + \beta_{\bfw, k}$ is a root, $p(j, k) \in [j+1, k-1]$ and $c_{j,k} \in \CC$ are such that
$\beta_{\bfw, j} + \beta_{\bfw, k}= \beta_{\bfw, p(j,k)}$ and 
$[E_{\bfw, j, -}, E_{\bfw, k, -}] 
=c_{j,k} E_{\bfw, p(j,k), -}$.
Our choice of $\lara_\g$ in fact ensures that every 
$c_{j, k}$ appearing in \eqref{eq:zjk} is an integer (see \cite[Theorem 5.21]{EL:BS}). 

\end{nota-rem}

\subsection{The index of the Lie algebra \texorpdfstring{$\n_{w, -}$}{nw-}}\label{ss:index-nw} For a Lie algebra $\l$,
recall from \eqref{eq:ind-l} and
\eqref{eq:pi0-l-all} the definitions of  
$\ind(\l)$, the index of $\l$, and of ${\rm mag}(\l)$, the magic number of $\l$.  

Fix $w \in W$. For $\alpha \in {\rm supp}(w)$, let $\vphi_{w,\alpha} \in \CC[N_w]$ be as defined in 
\eqref{eq:vphi-w-alpha}, and consider the homogeneous polynomial $\vphi_{w, \alpha}^\low \in \CC[\n_w]$, the 
 lowest degree term of $\vphi_{w, \alpha}$ at $e \in N_w$. 
The following 
Lemma \ref{lem:deg-cdeg} points to the 
representation theoretical nature of the integer $\deg(\vphi_{w, \alpha}^\low)$.
Recall the signed generalized minor 
$\delta_{\omega_\alpha, w\omega_\alpha}$ on $\g$, which, by Definition-Lemma \ref{de-lem:delta},
is also the co-degree term of the linear function
$\widehat{\Delta}_{\omega_\alpha, w\omega_\alpha}$ on $U\g$.

\begin{lem}\label{lem:deg-cdeg}
For any $w \in W$ and $\alpha \in {\rm supp}(w)$, as polynomial functions on $\n_w$, one has 
\[
\vphi_{w,\alpha}^\low  =\delta_{\omega_\alpha, \, w\omega_\alpha}|_{\n_w} =
\mbox{the co-degree term of}\; \widehat{\Delta}_{\omega_\alpha, w\omega_\alpha}|_{U\n_w},
\]
and $\deg(\vphi_{w, \alpha}^\low)$ is also equal to the degree of
$\delta_{\omega_\alpha, w\omega_\alpha}$ as a homogeneous polynomial on $\g$. 
\end{lem}

\begin{proof}
By  \eqref{eq:vphi-w-alpha}, 
$\vphi_{w, \alpha}= \Delta_{\omega_\alpha, w\omega_\alpha}|_{N_w}$, so 
$\vphi_{w,\alpha}^\low \in \CC[\n_w]$ coincides with the co-degree term of 
$\widehat{\Delta}_{\omega_\alpha, w\omega_\alpha}|_{U\n_w} \in (U\n_w)^*$. By Lemma \ref{lem:delta-1}, one has
$\vphi_{w,\alpha}^\low = \delta_{\omega_\alpha, \, w\omega_\alpha}|_{\n_w}
\in \CC[\n_w]$ and 
$\deg (\vphi_{w, \alpha}^\low) = \deg (\delta_{\omega_\alpha, w\omega_\alpha})$.
\end{proof}

\begin{rem}\label{rem:nw-Casimir}
For $w \in W$ and $\alpha \in {\rm supp}(w)$, it follows from 
Lemma \ref{lem:Nwm-inv} that  
\[
\vphi_{w, \alpha}^\low =\delta_{\omega_\alpha, \, w\omega_\alpha}|_{\n_w} \in \CC[\n_w]^{N_{w, -}}
\]
is a Casimir function for $(\n_w, \pi_{w, 0})$, a fact we already know from  Lemma \ref{lem:low-of-frozen}.
\end{rem}

For $w \in W$, set
\begin{equation}\label{eq:dw}
d_w = \sum_{\alpha \in {\rm supp}(w)} \deg (\vphi_{w,\alpha}^\low) = 
\sum_{\alpha \in {\rm supp}(w)} \deg (\delta_{\omega_\alpha, w\omega_\alpha}).
\end{equation}

\begin{prop}\label{prop:index-nw-0} For any  $w \in W$, the rank 
${\rm rk}(\pi_{0, \n_{w,-}})$ of the Kirillov-Kostant-Souriau  Poisson structure 
$\pi_{0, \n_{w, -}}$ on $\n_{w, -}^*$,  
the index  and the magic number  of the Lie algebra 
$\n_{w, -}$  are respectively given by
\[
{\rm rk}(\pi_{0, \n_{w,-}}) = 2(\ell(w)-d_w), \hs   \ind(\n_{w, -}) = 2d_w - \ell(w), \hs \mbox{and} \hs 
{\rm mag}(\n_{w, -}) = d_w.
\]
\end{prop}

\begin{proof} Consider the $T$-Pfaffian ${\rm Pf}_T(\pi_w)$ of $(BwB/B, \pi_w)$
and its lowest degree term ${\rm Pf}_T(\pi_w)^\low$ at $\wB$.  
By Corollary \ref{cor:TPf}, Lemma \ref{lem:pi0-nw},  and Theorem \ref{thm:TT-fixed}, one has
\begin{equation}\label{eq:Pf-dw}
\deg ({\rm Pf}_T(\pi_w)^\low) = d_w-\ell(w) = -\frac{1}{2}{\rm rk}(\pi_{0, \n_{w, -}}).
\end{equation}
The assertions for $\ind(\n_{w, -})$ and ${\rm mag}(\n_{w, -})$ now follow from the definitions and $\eqref{eq:Pf-dw}$.
\end{proof}

In the following result, for $w \in W$, $\ker(1+w)$ stands for the kernel of the linear map $1+w: \t^* \rightarrow \t^*$.

\begin{lem}\label{lem:dw}
For any $w \in W$, one has $\ind(\n_{w, -}) \geq \dim \ker (1+w)$,
\begin{equation}\label{eq:ind-nw-dw}
{\rm rk}(\pi_{0, \n_{w,-}}) \leq \ell(w) - \dim \ker (1+w) \hs \mbox{and} \hs 
2d_w \geq \ell(w) + \dim \ker (1+w).
\end{equation}
\end{lem}

\begin{proof}
By \cite[Theorem 1.1]{LM:T-leaves}, the rank of the Poisson structure $\pi_w$ on $N_w$ is $\ell(w) - \dim \ker (1+w)$.
Thus the rank of the linearization $(\n_w,\pi_{w, 0})$ of $(N_w, \pi_w)$ at $e \in N_w$ satisfies
${\rm rk}(\pi_{0, \n_{w,-}}) \leq \ell(w) - \dim \ker (1+w)$, which  gives the rest of the statements in Lemma 
\ref{lem:dw}. 
\end{proof}

\begin{example}\label{ex:s1s2}
Let $w = s_is_j$, where $i, j \in [1, r]$, $i \neq j$, and $\langle \alpha_i, \alpha_j \rangle \neq 0$. Then in the
Bott-Samelson coordinates $(z_1, z_2)$, we have $\{z_1, z_2\}_{\pi_w} = \langle \alpha_i, \alpha_j \rangle z_1z_2$ (see (\ref{eq:zjk-piw})), while
$\{z_1, z_2\}_{\pi_{w, 0}} = 0$, so $2 = \ind (\n_{w,-}) > \dim \ker (1 + w) = 0$. The inequalities in Lemma \ref{lem:dw} in this case are thus strict.
\end{example}

When $w = w_0$, the three inequalities in Lemma \ref{lem:dw} become equalities, and
we have a second explicit formula $d_{w_0}={\rm mag}(\n_-)$. For the precise statement,
let $\rho =\frac{1}{2}\sum_{\beta \in R_+}\beta = \sum_{\alpha \in \Gamma} \omega_\alpha$,
where  recall that $R_+$ is the set of all positive roots. Recall also 
$\kappa^\vee$ given in \eqref{eq:kappa-0}.

\begin{prop}\label{prop:w0} One has $\ind(\n_-) = \dim (\ker(1+w_0))$, 
${\rm rk}(\pi_{0, \n_-}) = \elw - \dim \ker (1+w_0)$, and  
\begin{equation}\label{eq:rho-kappa}
{\rm mag}(\n_-) = d_{w_0} =  \frac{1}{2}(\elw +\dim \ker (1+w_0))= (\kappa^\vee, \, \rho).
\end{equation}
\end{prop}

\begin{proof}
By Lemma \ref{lem:Kostant}, 
$\ind(\n_-) \leq \dim \r_-=|\calB| =\dim (\ker(1+w_0))$. Thus by Lemma \ref{lem:dw},
$\ind(\n_-) = \dim (\ker(1+w_0))$. Consequently
${\rm rk}(\pi_{0, \n_-}) = \elw - \dim \ker (1+w_0)$ and 
$2d_{w_0} = \elw + \dim \ker (1+w_0)$.
On the other hand, 
$d_{w_0} = \sum_{\alpha \in \Gamma} \deg (\vphi_{w_0, \alpha}^\low)$ by definition,
and it follows from Lemma 
\ref{lem:deg-cdeg}  and Lemma \ref{lem:di} that $d_{w_0} = (\kappa^\vee, \, \rho)$.
\end{proof}

As a consequence of Property $\calI$ of $\calC_{\rm st}^{w_0}$ at ${w_0}_\cdot B$ and Proposition
\ref{prop:w0}, for any Bott-Samelson coordinates
$(z_, \ldots, z_{\ell(w_0)})$ on $Bw_0B/B$ and taking lowest degree terms at $(0, \ldots, 0)$, we have
\begin{equation}\label{eq:low-mu-w0}
\deg \left(\left(\frac{dz_1 \wedge \cdots \wedge dz_{\ell(w_0)}}{\prod_{\alpha \in \Gamma} \vphi_{w_0, \alpha}(z)}\right)^\low \right)= \frac{1}{2}{\rm rk}(\pi_{0, \n_-})= \frac{1}{2}(\ell(w_0) - \dim \ker (1+w_0)).
\end{equation}

\begin{rem}\label{rem:another-proof}
While we have arrived at the last identity 
\eqref{eq:rho-kappa} using the $T$-Pfaffian of the standard Poisson structure $\pi_{w_0}$ on $Bw_0B/B$, 
it can also be proved by directly using the cascade $\calB$. Indeed, by construction each $\beta \in \calB$ is the highest root of some simple Lie sub-algebra $\g(\beta)$ of $\mathfrak g$, and we write $E(\beta)$ for the set of positive roots of $\g(\beta)$ which are not orthogonal to $\beta$. Then by \cite[Proposition 1.1]{K12}, one has 
$\n = \bigoplus \limits_{\beta \in \calB} {\mathfrak{H}}(\beta)$, where for each $\beta \in \calB$,
$ {\mathfrak{H}}(\beta) = \bigoplus \limits_{\nu \in E(\beta)} \mathfrak g_{\nu}$
is a Heisenberg Lie algebra of dimension $2 h_{\beta}^{\vee} - 3$, and
$h_{\beta}^{\vee}$ is the dual Coxeter number of $\g(\beta)$.  Thus
$\elw = \dim \n = \sum \limits_{\beta \in \calB} (2h_{\beta}^{\vee} - 3)$, and it follows that 
\[
\elw + \dim \ker (1+w_0) =\sum \limits_{\beta \in \calB} (2h_{\beta}^{\vee} - 3) + |\calB| = 2\sum_{\beta \in \calB}
(h_\beta^\vee -1).
\]
On the other hand,  by one of the equivalent definitions of the dual Coxeter number, for each $\beta \in \calB$ one has 
\[
h_\beta^\vee -1 = \text{sum of the coefficients of } \beta^{\vee} = \sum_{i=1}^r (\beta^{\vee},\omega_i) =(\beta^\vee, \rho),
\]
This gives another proof of the last identity of \eqref{eq:rho-kappa}.
\end{rem}

\subsection{Thimm's method and Hamiltonian flows of \texorpdfstring{$\vphi_{\bfw, k}^\low$}{phi-low}}\label{ss:Thimm-flow}
Let again $w \in W$, and let $\bfw = (s_{i_1}, \ldots, s_{i_l})$ be any reduced word of $w$.
Recall from \eqref{eq:vphi-bfw-k} that we have
\[
\Phi_\bfw = \{\vphi_{\bfw, k} = \Delta_{\omega_{i_k}, \, s_{i_1} \cdots s_{i_k} \omega_{i_k}}|_{N_w}: k \in [1, l]\} \subset \CC[N_w],
\]
where each $k \in [1, l]$, 
$\vphi_{\bfw, k}(g) = \Delta_{\omega_{i_k}, \, s_{i_1} \cdots s_{i_k}\omega_{i_k}}(g)$ for $g \in N_w$.
By Lemma \ref{Lem:mainlem}, the set 
\begin{equation}\label{eq:varphi-w-k-low}
\Phi_\bfw^\low = \{\vphi_{\bfw, k}^\low = \delta_{\omega_{i_k}, s_{i_1}\cdots s_{i_k}\omega_{i_k}}|_{\n_w}: k \in [1, l]\} \subset \CC[\n_w]
\end{equation}
of lowest degree terms at $e \in N_w$ is $\pi_{w, 0}$-involutive. 
In this section, we  give 
a second proof of the $\pi_0$-involutivity of $\Phi_\bfw^\low$ which is representation theoretical and is
in the sprit of the well-known Thimm's method.  
We also show that all the functions in $\Phi_{\bfw}^\low$ have complete Hamiltonian flows with respect 
to $\pi_{w, 0}$, or, more specifically, have quasi-polynomial Hamiltonian flows.

For $k \in [1,l]$, let $w_k = s_{i_1} \cdots s_{i_k}$.  We then have the sequence of Lie sub-algebras
\[
0 \subset \n_{w_1, -} \subset \n_{w_2, -}  \subset \cdots \subset \n_{w_{l-1}, -} \subset \n_{w_l, -} = \n_{w, -},
\]
with $\dim \n_{w_k, -} =k$ for every $k \in [1, l]$. Identifying $\n_{w_k} \cong \n_{w_k, -}^*$ via 
$\lara_\g$, we get the sequence of projections
\begin{equation}\label{eq:pr-n}
\n_w = \n_{w_l}\, {\longrightarrow} \,\n_{w_{l-1}} \,\longrightarrow \, \cdots \,
\,{\longrightarrow} \,\n_{w_2}\,
{\longrightarrow}\, \n_{w_1}\, \longrightarrow \,0,
\end{equation}
which gives rise to the sequence of Poisson sub-algebras 
\begin{equation}\label{eq:sequence-CC}
0\, \longhookrightarrow\, \CC[\n_{w_1}] \,{\longhookrightarrow} \,\CC[\n_{w_2}] \,
{\longhookrightarrow} \,\cdots \,\longhookrightarrow
\,\CC[\n_{w_{l-1}}] \, {\longhookrightarrow} \,(\CC[\n_w], \, \{\, , \, \}_{\pi_{w,0}}).
\end{equation}
Note that  for each $k \in [1, l-1]$, the projection 
$\n_{w_{k+1}} \to \n_{w_{k}}$ in \eqref{eq:pr-n} 
is the restriction of ${\rm pr}_{\n_{w_{k}}}: \g \to \n_{w_{k}}$  to
$\n_{w_{k+1}}$. Thus under the identification of $\CC[\n_{w_k}]$ as a sub-algebra of $\CC[\n_w]$ in 
\eqref{eq:sequence-CC}, we have  
\[
\vphi_{\bfw, k}^\low =\vphi_{w_k, \alpha_{i_k}}^\low =\delta_{\omega_{i_k}, \, w_k \omega_{i_k}}|_{\n_{w_k}} \in \CC[\n_{w_k}].
\]
By Lemma \ref{lem:low-of-frozen}, $\vphi_{w_k, \alpha_{i_k}}^\low $ is a Casimir element of $\CC[\n_{w_k}]$ 
for each $k \in [1, l]$, giving another explanation of the $\pi_0$-involutivity of the 
set $\{\vphi_{\bfw, k}^\low: k \in [1, l]\}$ as per Thimm's method \cite{Guillemin-Sternberg:collective}.

Turning to the Hamiltonian flows of the functions $\vphi_{{\bf w}, k}^\low$ for $k \in [1, l]$, 
let $z = (z_1, \ldots, z_l)$ be the Bott-Samelson coordinates on $N_w$ defined by $\bfw$  in 
\eqref{eq:BS-Nw}, and we regard each $\vphi_{\bfw, k}$ as a polynomial in $z$ as in 
\eqref{eq:vphi-k-z}. As noted in Notation-Remark \ref{nota-rem:BS-linear}, by denoting 
$d_ez_k = z_k^\low \in \n_w^*$ also by $z_k$ for
$k \in [1, l]$, we also regard $(z_1, \ldots, z_l)$ as 
a linear  coordinate system on $\n_w$, and the linear Poisson bracket $\{z_j, z_k\}_{\pi_{w, 0}}$ for $1 \leq j < k \leq l$ is then the linear term of
$\{z_j, z_k\}_{\pi_w} \in \CC[z]$. 

Recall that a quasi-polynomial in a variable $t$ is a  $\mathbb C$-linear combination of functions of 
$t$ of the form $q(t) e^{at}$, where $q(t)$ is a polynomial in $t$ and $a \in \mathbb C$ is a constant. 
We say that a curve $t \mapsto \gamma(t)$ in a vector space $V$ is
{\it quasi-polynomial} if it is defined for all $t \in \CC$ and if, when written in one, equivalently in any, linear coordinate system on $V$, all of its
components are quasi-polynomials in $t$. 

\begin{defn}\label{defn:quasi-poly}
If $V$ is a vector space with an algebraic Poisson structure $\pi$, a polynomial function $\vphi$ on $V$
is said to have {\it quasi-polynomial Hamiltonian flow with respect to $\pi$} if all the integral curves of the Hamiltonian vector field of $\vphi$ with respect to $\pi$ are quasi-polynomial.
\end{defn}

\begin{prop} \label{prop:quasi-poly-GB}
For every reduced word $\bfw$ of $w$ and for every $k \in [1, l]$, where $l = \ell(w)$, the polynomial function $\varphi_{\bfw, k}^{\low}$ on $\n_w$  has quasi-polynomial Hamiltonian flow with respect to $\pi_{w, 0}$.
\end{prop}

\begin{proof}
Fix $k\in [1, l]$, and let $X_{\varphi_{\bfw, k}^{\low}}$ be 
the Hamiltonian vector field of $\varphi_{\bfw, k}^{\low}$ with respect to $\pi_{w, 0}$.
 Let $x \in \n_w$ and write the integral curve of $X_{\varphi_{\bfw, k}^{\low}}$ through $x$ in the 
  linear coordinates $(z_1, \ldots, z_l)$ as $(z_1(t), \ldots, z_l(t))$.  By Lemma \ref{lem:low-of-frozen},
$\{z_j, \, \vphi_{\bfw, k}^\low\}_{\pi_{w,0}} = 0$ for all $j \in [1, k]$, so 
  $z_j(t)$ is constant in $t$ for every $j \in [1,k]$. Let $j \in [k+1, l]$ and assume that $z_1(t), \ldots, z_{j-1}(t)$
are all quasi-polynomials in $t$. Since $\vphi_{\bfw, k} \in \CC[z_1, \ldots, z_k]$, by \eqref{eq:zjk}, 
$\{z_j, \vphi_{\bfw, k}\}_{\pi_{w,0}} \in \CC[z_1, \ldots, z_{j-1}]$. By the induction assumption, 
$\frac{d z_j(t)}{dt}$ is a quasi-polynomial in $t$, so $z_j(t)$, being an anti-derivative of $\frac{d z_j(t)}{dt}$, is a quasi-polynomial in $t$. 
\end{proof}

\subsection{Examples of integrable systems on \texorpdfstring{$(\n_w, \pi_{w, 0})$}{(nw,piw)}}\label{ss:choice-GB}
For any $w \in W$ and  any reduced word ${\bf w}$ of $w$, we now describe an explicit integrable system on $(\n_w, \pi_{w, 0})$ contained in $\Phi_\bfw^\low$ in \eqref{eq:varphi-w-k-low}. 

\begin{lem} \label{lem:degjump}
Let $w \in W$ with $l = \ell(w)$ and let $\bfw$ be any reduced word of $w$. For every $k \in [1,l]$ one has
\begin{equation}\label{eq:deg-phi-phi}
\deg (\varphi_{\bfw, k^-}^{\low})\leq \deg (\varphi_{\bfw, k}^{\low})\leq 1+\deg (\varphi_{\bfw, k^-}^{\low}),
\end{equation}
and the following conditions on $k \in [1, l]$ are equivalent:

1) $\deg (\varphi_{\bfw, k}^{\low})=1+\deg (\varphi_{\bfw, k^-}^{\low})$;

2) $\{\varphi_{\bfw, k^-}^{\low},\; z_k\}_{\pi_{w, 0}}= 0$;

3) $\varphi_{\bfw, k^-}^{\low}$ is a Casimir function of 
$(\CC[\n_w], \{\, , \, \}_{\pi_{w, 0}}) = (\CC[z_1, \ldots, z_k], \{\, , \, \}_{\pi_{w, 0}})$;

4) $\varphi_{\bfw, k}^{\low} = z_k\varphi_{k^-}^{\low} + \varsigma_{k}$ for some 
$\varsigma_{k} \in \CC[z_1, \ldots, z_{k-1}]$;

5) $\varphi_{\bfw, k}^{\low} \notin \CC[z_1, \ldots, z_{k-1}]$.

\noindent
When none of these equivalent conditions
holds, one has $\vphi_{\bfw, k}^\low = 
\frac{1}{\langle \alpha_{i_k}, \, \alpha_{i_k}\rangle} \{\varphi_{\bfw, k^- }^{\low}, \, z_k\}_{\pi_{w, 0}}
\in \CC[z_1, \ldots, z_{k-1}]$.
\end{lem}

\begin{proof} Let $k \in [1,l]$. By \eqref{eq:phik-CGL}, $\vphi_{\bfw, k} = z_k \vphi_{\bfw, k^-}+\phi_{\bfw, k}$
for some 
$\phi_{\bfw,k} \in \CC[z_1, \ldots, z_{k-1}]$. It follows that $\deg (\varphi_{\bfw, k}^{\low}) \leq 1+
\deg (\varphi_{\bfw, k^-}^{\low})$. By \eqref{eq:vphik-recursive}, one has
\begin{equation}\label{eq:phi-ab}
\vphi_{\bfw, k} = \frac{1}{\lambda_k} \{\vphi_{\bfw, k^-}, \, z_k\}_{\pi_w} + \xi_k z_k \vphi_{\bfw, k^-},
\end{equation}
where $\lambda_k ={\langle \alpha_{i_k}, \, \alpha_{i_k}\rangle}$ and
$\xi_k = (\langle s_{i_1}\cdots s_{i_{k-1}}\alpha_{i_k},\omega_{i_k}-s_{i_1}\cdots s_{i_k}\omega_{i_k}\rangle)/\lambda_k$. 

If $\{\varphi_{\bfw, k^- }^{\low}, \,z_k\}_{\pi_{w,0}}=0$, then all the monomial terms on the 
right-hand side of \eqref{eq:phi-ab} have degree at least $1+\deg (\vphi_{\bfw, k^-}^\low)$, so 
$\deg (\varphi_{\bfw, k}^{\low}) = 1+
\deg (\varphi_{\bfw, k^-}^{\low})$. 
If $\{\varphi_{\bfw, k^- }^{\low}, \, z_k\}_{\pi_{w,0}}\neq 0$, then  
$\vphi_{\bfw, k}^\low = \frac{1}{\lambda_k} \{\varphi_{\bfw, k^- }^{\low}, \, z_k\}_{\pi_{w, 0}}$, 
so $\deg(\vphi_{\bfw, k}^\low) = \deg(\vphi_{\bfw, k^-}^\low)$.  We have thus proved 
\eqref{eq:deg-phi-phi} and the equivalence of 1) and 2). 

By Lemma \ref{lem:low-of-frozen}, $\{\vphi_{\bfw, k^-}^\low, z_j\}_{\pi_{w,0}} =0$ for every $j \in [1, k-1]$.
Thus 2) and 3) are equivalent. 

 It follows from $\vphi_{\bfw, k} = z_k \vphi_{\bfw, k^-}+\phi_{\bfw, k}$ that 1) implies 4), and it is evident that
 4) implies 5). If 5) holds, since $\phi_{\bfw,k} \in \CC[z_1, \ldots, z_{k-1}]$,
 $\vphi_{\bfw, k}^\low$ must contain $z_k\vphi_{\bfw, k^-}^\low$ as a summand, 
 so 1) holds. 
\end{proof}

For $w \in W$ and any reduced word $\bfw = (s_{i_1}, \cdots, s_{i_l})$ of $w$, let 
\begin{equation}\label{eq:K-bfw}
K_{\bf w} = \{k \in [1, l]:\, \deg \varphi_{\bfw, k}^\low = 1+\deg \varphi_{\bfw, k^-}^\low\}.
\end{equation}

\begin{thm} \label{thm:choice-GB}
For $w \in W$ and any reduced word $\bfw = (s_{i_1}, \cdots, s_{i_l})$ of $w$, the set 
$\{\varphi_{\bfw, k}^{\low}:  k \in K_{\bfw}\}$ 
is a polynomial integrable system on $(\n_w, \pi_{w, 0})$.
\end{thm}

\begin{proof}
 As $\vphi_{\bfw, k}^\low \in \CC[z_1, \ldots, z_k]$ for each $k \in [1,l]$, 
 Condition 5) in Lemma \ref{lem:degjump} implies that the functions in $\{\varphi_{\bfw, k}^{\low}:  k \in K_\bfw\}$ 
 are independent, and Condition 1) in Lemma \ref{lem:degjump} implies that 
\[
|K_\bfw|=\sum_{k \in [1, l]\backslash {\rm ex}} \deg (\varphi_{\bfw, k}^\low),
\]
which, by Proposition \ref{prop:index-nw-0}, is equal to the magic number of $(\n_w, \pi_{w, 0})$. 
\end{proof}

\begin{example}\label{ex:An-0}
Consider $G = \SL(n+1,\CC)$ and the reduced word
\begin{equation}\label{eq:spw0}
\bfwz = (s_1, s_2, \ldots, s_{n-1}, s_n, s_1, s_2, \ldots, s_{n-1}, \ldots, s_1, s_2, s_1)
\end{equation}
of the longest element $w_0$.  Up to a sign, the set $\{\varphi_{\bfwz, k}^{\low}: 1 \le k \le \frac{1}{2}n(n+1)\} = \{\varphi_{\bfwz, k}^{\low} \colon k \in K_{\bfwz}\}$ 
consists of the determinants of solid square submatrices that contain the first row and do not contain zero entries, depicted as follows (the matrices being a generic element of $\n$):
    \begin{equation*}
        \tikz[baseline=(M.west)]{%
            \node[matrix of math nodes,matrix anchor=west,left delimiter={[},right delimiter={]},ampersand replacement=\&] (M){%
                      0 \& * \& * \& * \& \cdots \& * \& * \& * \\
                      0 \& 0 \& * \& * \& \cdots \& * \& * \& * \\
                      0 \& 0 \& 0 \& * \& \cdots \& * \& * \& * \\
                      0 \& 0 \& 0 \& 0 \& \cdots \& * \& * \& * \\
                      \vdots \& \vdots \& \vdots \& \vdots \& \ddots \& \vdots \& \vdots \& \vdots \\
                      0 \& 0 \& 0 \& 0 \& \cdots \& 0 \& * \& * \\
                      0 \& 0 \& 0 \& 0 \& \cdots \& 0 \& 0 \& * \\
                      0 \& 0 \& 0 \& 0 \& \cdots \& 0 \& 0 \& 0 \\
                };
            \node[draw,fit=(M-1-2), inner sep=-1pt] {};
            \node[draw,fit=(M-1-3), inner sep=-1pt] {};
            \node[draw,fit=(M-1-4), inner sep=-1pt] {};
            \node[draw,fit=(M-1-6), inner sep=-1pt] {};
            \node[draw,fit=(M-1-7), inner sep=-1pt] {};
            \node[draw,fit=(M-1-8), inner sep=-1pt] {};
        },
        \tikz[baseline=(M.west)]{%
            \node[matrix of math nodes,matrix anchor=west,left delimiter={[},right delimiter={]},ampersand replacement=\&] (M){%
                      0 \& * \& * \& * \& \cdots \& * \& * \& * \\
                      0 \& 0 \& * \& * \& \cdots \& * \& * \& * \\
                      0 \& 0 \& 0 \& * \& \cdots \& * \& * \& * \\
                      0 \& 0 \& 0 \& 0 \& \cdots \& * \& * \& * \\
                      \vdots \& \vdots \& \vdots \& \vdots \& \ddots \& \vdots \& \vdots \& \vdots \\
                      0 \& 0 \& 0 \& 0 \& \cdots \& 0 \& * \& * \\
                      0 \& 0 \& 0 \& 0 \& \cdots \& 0 \& 0 \& * \\
                      0 \& 0 \& 0 \& 0 \& \cdots \& 0 \& 0 \& 0 \\
                };
            \node[draw,fit=(M-1-3)(M-2-4), inner sep=1pt] {};
            \node[draw,fit=(M-1-6)(M-2-7), inner sep=-1pt] {};
            \node[draw,fit=(M-1-7)(M-2-8), inner sep=1pt] {};
        },
        \tikz[baseline=(M.west)]{%
            \node[matrix of math nodes,matrix anchor=west,left delimiter={[},right delimiter={]},ampersand replacement=\&] (M){%
                      0 \& * \& * \& * \& \cdots \& * \& * \& * \\
                      0 \& 0 \& * \& * \& \cdots \& * \& * \& * \\
                      0 \& 0 \& 0 \& * \& \cdots \& * \& * \& * \\
                      0 \& 0 \& 0 \& 0 \& \cdots \& * \& * \& * \\
                      \vdots \& \vdots \& \vdots \& \vdots \& \ddots \& \vdots \& \vdots \& \vdots \\
                      0 \& 0 \& 0 \& 0 \& \cdots \& 0 \& * \& * \\
                      0 \& 0 \& 0 \& 0 \& \cdots \& 0 \& 0 \& * \\
                      0 \& 0 \& 0 \& 0 \& \cdots \& 0 \& 0 \& 0 \\
                };
            \node[draw,fit=(M-1-6)(M-3-8), inner sep=-1pt] {};
            \draw(M-1-4.north west) -- (M-3-4.south west);
            \draw(M-1-4.north west) -- (M-1-5.north west);
            \draw(M-3-4.south west) -- (M-3-5.south west);
        }, \cdots.
    \end{equation*}
    Here, when $n$ is odd, the last minor is of size $\frac{n+1}{2} \times \frac{n+1}{2}$; when $n$ is even, the last minor is of size $\frac{n}{2} \times \frac{n}{2}$.
\end{example}

\begin{example}\label{ex:G2-0}
 Consider $G_2$ with simple roots $\alpha_1$ and $\alpha_2$ such that 
$3\langle \alpha_1, \alpha_1\rangle = \langle \alpha_2, \alpha_2\rangle = 6$, and let $w_0$
be the longest element in $W$. Writing $\pi = \pi_{w_0}$, in the Bott-Samelson coordinates $z = (z_1, \ldots, z_6)$ on $N =N_{w_0}$ associated to 
the reduced word 
${\bf w}_0 = (s_1, s_2, s_1, s_2, s_1, s_2)$ of $w_0$, we have (\cite[Example 5.23]{EL:BS} and \cite[Proposition 4.10]{L})
\begin{align*}
\{z_{1},z_{2}\}_\pi &= -3z_{1}z_{2}, \hs 
\{z_{1},z_{3}\}_\pi = -z_{1}z_{3}-2z_{2}, \hs 
\{z_{1},z_{4}\}_\pi = -6z_{3}^2, \\
\{z_{1},z_{5}\}_\pi &= z_{1}z_{5}-4z_{3}, \hs  
\{z_{1},z_{6}\}_\pi = 3z_{1}z_{6}-6z_{5}, \hs
\{z_{2},z_{3}\}_\pi = -3z_{2}z_{3},\\ 
\{z_{2},z_{4}\}_\pi &= -6z_{3}^3-3z_{2}z_{4}, \hs 
\{z_{2},z_{5}\}_\pi = -6z_{3}^2, \hs 
\{z_{2},z_{6}\}_\pi = 3z_{2}z_{6}-18z_{3}z_{5}+6z_{4},\\ 
\{z_{3},z_{4}\}_\pi &= -3z_{3}z_{4}, \hs 
\{z_{3},z_{5}\}_\pi = -z_{3}z_{5}-2z_{4}, \hs 
\{z_{3},z_{6}\}_\pi = -6z_{5}^2,\\ 
\{z_{4},z_{5}\}_\pi &= -3z_{4}z_{5}, \hs 
\{z_{4},z_{6}\}_\pi = -6z_{5}^3-3z_{4}z_{6}, \hs 
\{z_{5},z_{6}\}_\pi = -3z_{5}z_{6}.
\end{align*}
By repeatedly using \eqref{eq:gqq} and \eqref{eq:master}, or by \eqref{eq:vphik-recursive}, we get
\begin{align*}
&\vphi_{\bfw_0, 1}=z_1, \hs \vphi_{\bfw_0, 2}=z_2, \hs \vphi_{\bfw_0, 3}=z_1z_3-z_2, 
\hs \vphi_{\bfw_0, 4}=z_2z_4-z^3_3,\\
&\vphi_{\bfw_0, 5}=z_1z_3z_5-z_2z_5-z_1z_4+z^2_3, \hs
\vphi_{\bfw_0, 6}=z_2z_4z_6-z^3_3z_6-z_2z^3_5-3z_3z_4z_5+3z^2_3z^2_5+z^2_4,
\end{align*}
and the frozen variables are $\vphi_{\bfw_0, 5}$ and $\vphi_{\bfw_0, 6}$. Correspondingly, we have
\begin{align*}
&\vphi_{\bfw_0, 1}^\low=z_1, \hs \vphi_{\bfw_0, 2}^\low=z_2, \hs \vphi_{\bfw_0, 3}^\low=-z_2, 
\hs \vphi_{\bfw_0, 4}^\low=z_2z_4,\\
&\vphi_{\bfw_0, 5}^\low=-z_2z_5-z_1z_4+z^2_3, \hs
\vphi_{\bfw_0, 6}^\low=z^2_4.
\end{align*}
The integrable system on the linearization $(\n, \pi_{w_0, 0})$  constructed in 
Theorem \ref{thm:choice-GB} is
\[
\{\vphi_{\bfw_0, k}^\low: \,k =1,2,4,5\} = \{z_1, \;z_2, \;z_2z_4, \;z_3^2-z_2z_5-z_1z_4\}.
\]
\end{example}

\subsection{Another integrable system and the Poisson center of \texorpdfstring{$(\n_w, \pi_{w, 0})$}{(nw,pi)}}\label{ss:Poi-center}
Let again $w \in W$ with $l = \ell(w)$ and let $\pi_{w, 0}$ be the linear Poisson structure on $\n_w$ as in Notation 
\ref{nota:nw-pi0}. For a reduced word $\bfw$ of $w$, we saw in Example \ref{ex:G2-0} that 
the homogeneous polynomials $\{\vphi_{{\bf w}, k}^\low: k \in [1, l]\}$ on $\n_w$ 
are not necessarily irreducible.
Consider also the Poisson center (see \eqref{eq:Nwm-nw} for the action of $N_{w, -}$ on $\n_w$)  
\[
C_w=\CC[\n_w]^{N_{w, -}}
\]
 of $(\n_w, \pi_{w, 0})$.
By \cite[Proposition 1]{Dixm:57}, 
the fraction field ${\rm Frac}(C_w)$ of $C_w$ is a transcendental field extension of
$\CC$ with pure transcendental degree equal to $\ind(\n_{w, -})$, the index of the Lie algebra $\n_{w, -}$.
Thus the maximal number of algebraically 
independent elements in $C_w$ is also $\ind(\n_{w, -})$. 
As also seen in Example \ref{ex:G2-0}, since $C_{w_0} = \CC[z_3^2-z_2z_5-z_1z_4, z_4]$ by Remark \ref{rem:2rho}, the integrable system associated to a reduced word of $w$
constructed in Theorem \ref{thm:choice-GB}  does not necessarily contain 
 $\ind(\n_{w, -})$ many Casimir functions for $\pi_{w, 0}$. 
In this section, we explain a method of altering $\{\vphi_{{\bf w}, k}^\low: k \in [1, l]\}$ so that the problem about irreducibility is rectified, and the resulting set contains more Casimir functions in general.

For every $w \in W$, the following statement on $C_w$ follows from the same arguments used in the proof of
\cite[Proposition 3.3]{K12}
and we give a (slightly different) proof for the convenience of the reader.

\begin{lem}\label{lem:pq}
For any $w \in W$,  prime factors of any non-zero $q\in 
C_w$ are also in $C_w$. 
\end{lem}

\begin{proof}
Let $q \in C_w$ be non-zero and let $p$ be a prime factor of
$q$. Write $q = p^mp_1$, where $m \geq 1$ and $p_1$ is co-prime with $p$. Let $k \in [1, l]$ be arbitrary.
It follows from 
$\{z_k, q\}_{\pi_{w, 0}} = 0$ that 
\[
0 = \{z_k, \; p^mp_1\}_{\pi_{w, 0}} = mp^{m-1}p_1\{z_k, \, p\}_{\pi_{w, 0}}+p^m\{z_k, \, p_1\}_{\pi_{w, 0}},
\]
so $p|\{z_k, p\}_{\pi_{w, 0}}$. If $\{z_k, p\}_{\pi_{w, 0}} \neq 0$, then by comparing the 
largest degree of the
monomials in $\{z_k, p\}_{\pi_{w, 0}}$ with that in $p$, we must have $\{z_k, p\}_{\pi_{w, 0}} = c_k p$
for some $c_k \in \CC^\times$, contradicting the fact that $\{z_k, -\}_{\pi_w, 0}$ is a locally nipotent 
operator on $\CC[\n_w]$. Thus $\{z_k, p\}_{\pi_{w, 0}}=0$.  As $k \in [1, l]$ is arbitrary, we have 
$p \in C_w$.
\end{proof}

Let now ${\bf w} = (s_{i_1}, \ldots, s_{i_l})$ be any reduced word of $w$. For $k \in [1, l]$, let again  
$w_k = s_{i_1}\cdots s_{i_k}$, and let
\begin{equation}\label{eq:ACFZ}
R_k = \CC[\n_{w_k}] \cong \CC[z_1, \ldots, z_k] \hs \mbox{and} \hs C_k =C_{w_k}.
\end{equation}
For each $k \in [1, l]$, we identify 
$(R_k, \,  \{\,, \, \}_{\pi_{w_k, 0}})$ 
as a Poisson sub-algebra of $(R_l=\CC[\n_w],\,  \{\,, \, \}_{\pi_{w, 0}})$, so 
$\{\vphi, \, \phi\}_{\pi_{w_k, 0}} = \{\vphi, \phi\}_{\pi_{w, 0}}$ for all $\vphi, \phi \in R_k$.
Let $\{\varphi_{{\bf w}, k}^\low: k \in [1, l]\}$ be as in \eqref{eq:varphi-w-k-low}. Recall from
\eqref{eq:K-bfw} that $K_{\bf w}$ is the set of all $k \in [1, l]$ satisfying any of the equivalent conditions in 
Lemma \ref{lem:degjump}. 

\begin{lem-de}\label{lem-de:psi-k}
For each $k \in K_{\bf w}$ the element 
$\vphi_{{\bf w}, k}^\low \in C_k$ has a prime factor in $R_k$, denoted by $\psi_{\bfw, k}$, which is
unique up to a non-zero scalar multiple and is
of the form
\begin{equation}\label{eq:psi-bfw-k}
\psi_{\bfw, k} = z_k c_{k}+b_{k} \in R_k,
\end{equation}
where $c_{k} \in R_{k-1}\cap C_k$ is a non-zero factor of $\vphi_{\bfw, k^-}^\low$ in 
$R_{k-1}$, and $b_{k}  \in R_{k-1}$. 
\end{lem-de}

\begin{proof}
Let $k \in K_{\bfw}$. By Lemma \ref{lem:degjump},  
$\varphi_{\bfw, k}^{\low} = z_k\varphi_{k^-}^{\low} + \varsigma_{k}$ for some 
$\varsigma_{k} \in R_{k-1}$. If $\varphi_{\bfw, k}^{\low} \in R_k$
is irreducible, we take $\psi_{\bfw, k} = \varphi_{\bfw, k}^{\low}$. Otherwise by Gauss Lemma 
$\varphi_{\bfw, k}^{\low}$ has a prime factor $\psi_{\bfw, k}$ in  $R_k$ of the 
form in \eqref{eq:psi-bfw-k}, unique up to a non-zero scalar multiple, where $c_{k}\in R_{k-1}$ is a factor of
$\vphi_{\bfw, k^-}^\low$, $b_{k}\in R_{k-1}$ is a factor of $\varsigma_{k}$, and $c_k$ and $b_k$ are coprime in
$R_{k-1}$.
By Lemma \ref{lem:pq}, $c_{k} \in C_{k}$. 
\end{proof}

\begin{thm}\label{thm:psi}
For any $w \in W$ of length $l = \ell(w)$ and any reduced word $\bfw$ of $w$, the set 
\[
\{\psi_{\bfw, k}: k \in K_\bfw\}
\]
is a polynomial integrable system on $(\n_w, \pi_{w, 0})$, and  for each $k \in K_w$, the polynomial
$\psi_{\bfw, k} \in R_k$ is homogeneous, irreducible, and its Hamiltonian flow in $\n_w$ with respect to
$\pi_{w, 0}$ is
quasi-polynomial.
\end{thm}

\begin{proof} 
That each $\psi_{\bfw, k}$ is homogeneous and  irreducible follows from the definition of $\psi_{\bfw, k}$. 
By Lemma \ref{lem:pq}, $\{\psi_{\bfw, j}, \psi_{\bfw, k}\}_{\pi_{w, 0}} = 0$ for all $j, k \in
K_\bfw$ and $j < k$. Moreover, as $\psi_{\bfw, k} \in R_k \backslash R_{k-1}$, the set
$\{\psi_{\bfw, k}: k \in K_\bfw\}$ is independent. 
That each $\psi_{\bfw, k}$ has quasi-polynomial Hamiltonian flow with respect to
$\pi_{w, 0}$ is proved exactly that same way as that for $\vphi_{\bfw, k}^\low$ in 
Proposition  \ref{prop:quasi-poly-GB}.
\end{proof}

\begin{example}\label{ex:G2-1}
Continuing with Example \ref{ex:G2-0}, we have $K_{\bfw_0} = \{1, 2, 4, 5\}$ and 
\[
\psi_{\bfw_0, 1} = z_1, \hs \psi_{\bfw_0, 2} = z_2, \hs 
\psi_{\bfw_0, 4} = z_4, \hs \psi_{\bfw_0, 5} = \vphi_{{\bf w}_0, 5}^\low = -z_2z_5-z_1z_4+z^2_3.
\]
Notice that $\psi_{\bfw_0, 4}$ and $\psi_{\bfw_0, 5}$ are free generators of $C_{w_0}$ and $\psi_{\bfw_0, 4} \notin \{\vphi_{\bfw_0, k}^\low \colon k = 1,2,4,5\}$.
\end{example}

\section{The standard complex semi-simple Poisson Lie group \texorpdfstring{$(G, \pist)$}{(G,pist)}} \label{s:G}
Consider now the standard complex semi-simple Poisson Lie group $(G,\pist)$ (see $\S$\ref{ss:pist}), and
recall that $\pist$ vanishes at the identity element $e$ of $G$. We have seen in $\S$\ref{ss:pist} that the linearization 
$\pi_0$ of $\pist$ at $e$ can be identified with the Kirillov-Kostant-Souriau  Poisson structure
$\pi_{0, \gsts}$ on $(\gsts)^*$, where $\gsts$ is the Lie sub-algebra of $\gog$ given in  
\eqref{eq:gsts}, and we identify $\g \cong \g_{\rm diag}$ with $(\gsts)^*$  using the bilinear form $\lara_{\gog}$ on
$\gog$ in \eqref{eq:lara-gog}.
In this section, we show that the 
Berenstein-Fomin-Zelevinsky cluster structure $\calC_{\rm BFZ}(G)$ in $\CC(G)$,
with a frozen variable modification when $w_0$ does not act on $\t$ as $-1$, has Property $\calI$ at $e$. We also study in 
more detail 
some polynomial integrable systems on  $(\g, \pi_0)$ obtained from some of the
extended clusters of $\calC_{\rm BFZ}(G)$.


\subsection{The BFZ cluster structure \texorpdfstring{$\mathcal{C}_{\rm BFZ}(G)$}{Cbfz}  in \texorpdfstring{$\CC(G)$}{C(G)}} \label{sect:logcansysG}
 Let again $w_0 \in W$ be the longest element and choose any two reduced words 
${\bf w}_0 = (i_1, \,i_2, \, \ldots, \, i_{\elw})$ and  
${\bf w}_0^\prime = (j_1, \,j_2, \, \ldots, \, j_{\elw})$
of $w_0$. In the notation of \cite{BFZ:III} and by \cite[$\S$2.3]{BFZ:III}, the sequence
\begin{equation}\label{eq:bfi}
{\bf i} = (i_{\elw}, \, \ldots, \,i_2, \, i_1,\, -j_1, \,-j_2, \, \ldots, \, -j_{\ell(w_0)})
\end{equation}
 gives rise to a set $\Phi({\bf i})$ of  $r + 2\elw$ independent regular functions on $G$ given by
\begin{align} \nonumber
    \Phi({\bf i}) &= \{\Delta_{\omega_i,\, w_0\omega_i}: i \in [1, r]\} \sqcup 
    \{\Delta_{\omega_{i_k}, \; s_{i_1} \cdots s_{i_{k-1}} \omega_{i_k}}, \;
    \Delta_{s_{j_1} \cdots s_{j_{k}} \omega_{j_k}, \; \omega_{j_k}}: k \in [1,{\ell(w_0)}]\}\\
\label{eq:Phi-G}
&=\{\Delta_{\omega_i,\,\omega_i}: i \in [1, r]\} \sqcup 
    \{\Delta_{\omega_{i_k}, \; s_{i_1} \cdots s_{i_{k}} \omega_{i_k}}, \;
    \Delta_{s_{j_1} \cdots s_{j_{k}} \omega_{j_k}, \; \omega_{j_k}}: k \in [1,{\ell(w_0)}]\}.
\end{align}
For $i \in [1, r]$, let $f_i = \Delta_{\omega_i,\, w_0\omega_i}$ and  $g_i = \Delta_{w_0 \omega_i, \omega_i}$.
Then $f_i \in \Phi({\bf i})$ for each $i \in [1, r]$, and if 
$k \in [1, \elw]$ is the 
largest such that $j_k = i$, then 
$g_i = \Delta_{s_{j_1} \cdots s_{j_{k}} \omega_{j_k}, \; \omega_{j_k}}\in \Phi({\bf i})$.
Note also that all the functions in $\Phi({\bf i})$ are $T$-weight vectors for the $T$-action on $G$ by left translation.

In \cite[$\S$2.3]{BFZ:III}, A. Berenstein, S. Fomin and A. Zelevinsky constructed a seed $\Sigma({\bf i})$ 
in $\CC(G) = \CC(G^{w_0, w_0})$
which has $\Phi({\bf i})$ in \eqref{eq:Phi-G} as an extended cluster and $\{f_1, \ldots, f_r, g_1, \ldots, g_r\}$ as the frozen variables. We denote by $\mathcal{C}_{\rm BFZ}(G)$ the mutation
equivalence class of seeds in $\CC(G)$ defined by $\Sigma({\bf i})$ and call
$\calCBFZ(G)$  the 
{\it BFZ cluster structure in $\CC(G)$}.
By \cite[Paragraph after Theorem 1.1]{ShenWeng:DBS},
$\mathcal{C}_{\rm BFZ}(G)$ is independent of the choices of the two reduced words for $w_0$. 

\begin{lem}\label{lem:comp-T-pist}
The cluster structure $\calCBFZ(G)$ in $\CC(G)$ is regular on $G$
and compatible with both the Poisson structure $\pist$ and the $T$-action.
\end{lem}

\begin{proof}
The first statement is proved in \cite{Qin-Yakimov:G} and  the second in \cite[Theorem 4.18]{GSV:book}.
\end{proof}

In what follows, we introduce a regular $\pist$-compatible 
frozen variable modification of $\mathcal{C}_{\rm BFZ}(G)$ (see Definition \ref{defn:regular}).
Recall the involution $i \mapsto i^{\ast}$ defined by $\omega_{i^{\ast}} = -w_0 \omega_i$.  Fix a decomposition 
\begin{equation}\label{eq:III}
[1, r] = I_0 \sqcup I_1 \sqcup I_2,
\end{equation}
where $I_0 = \{i \in [1, r]: i = i^*\}$, $I_1 \sqcup I_2 = \{i \in [1, r]: i \neq i^*\}$, and $I_1$ contains exactly one element of each pair $(i, i^*)$ such that $i \neq i^*$.  For $i \in I_2$, introduce
\begin{equation*} 
g_i^\prime = f_ig_i-f_{i^*}g_{i^*} = \Delta_{\omega_i, w_0\omega_i} \Delta_{w_0\omega_i, \omega_i} - \Delta_{\omega_{i^*}, w_0\omega_{i^*}}
\Delta_{w_0\omega_{i^*}, \omega_{i^*}}.
\end{equation*}
Setting $c_i=\frac{f_i}{g_{i^*}} \in \CC(G)$ for $i \in [1, r]$, we then have 
$g_i' = (c_i - c_{i^{\ast}}) g_i g_{i^{\ast}}$ for $i \in I_2$. 
Let
\[
{\rm Froz} = \{f_i, g_i: i \in [1, r]\} \hs \mbox{and} \hs
\overline{{\rm Froz}} = \{f_i: i \in [1, r]\} \sqcup \{g_i: i \in I_0 \sqcup I_1\} \sqcup \{g_i': i \in I_2\}.
\]

\begin{lem}
1) Each $c_i$, for $i \in [1, r]$, is a rational Casimir function on 
    $G$ with respect to $\pist$;
    
2) The set $\overline{{\rm Froz}}$ is a regular
    $\pist$-compatible frozen variable modification
    of $\calCBFZ(G)$.
\end{lem}
\begin{proof}
Note that $\overline{{\rm Froz}} \subset \CC[G]$ and freely generates
the subfield $\CC({\rm Froz})$ of $\CC(G)$. 
    By Definition \ref{defn:regular}, we only need prove 1).  
By \cite[Proposition 4.19]{GSV:book}, for all $i, k \in [1, r]$ and $u, v \in W$, one has
\begin{align*}
\{f_i, \; \Delta_{u\omega_k,v\omega_k}\}_{\pist} 
&= \{\Delta_{\omega_i, \, w_0\omega_i}, \; \Delta_{u\omega_k,v\omega_k}\}_{\pist} =\left(\langle \omega_i,\, u\omega_k\rangle-\langle w_0\omega_i, \, v\omega_k\rangle\right) f_i\Delta_{u\omega_k,v\omega_k},\\
\{g_i, \; \Delta_{u\omega_k,v\omega_k}\}_{\pist} 
&= \{\Delta_{w_0\omega_i, \, \omega_i}, \; \Delta_{u\omega_k,v\omega_k}\}_{\pist} = \left(-\langle w_0\omega_i,\, u\omega_k\rangle+\langle \omega_i, \, v\omega_k\rangle\right) g_i\Delta_{u\omega_k,v\omega_k}.
\end{align*}
It follows that for all $i, k \in [1, r]$ and $u, v \in W$, one has
\begin{align*}
\{c_i, \, \Delta_{u\omega_k,v\omega_k}\}_{\pist} & = 
\frac{1}{g_{i^*}} \{f_i, \; \Delta_{u\omega_k,v\omega_k}\}_{\pist} -\frac{f_i}{g_{i^*}^2} 
\{g_{i^*}, \; \Delta_{u\omega_k,v\omega_k}\}_{\pist} \\
& = \left(\langle \omega_i,\, u\omega_k\rangle-\langle w_0\omega_i, \, v\omega_k\rangle +
\langle w_0\omega_{i^*},\, u\omega_k\rangle-\langle \omega_{i^*}, \, v\omega_k\rangle\right) c_i \Delta_{u\omega_k,v\omega_k}\\
& =0.
\end{align*}
In particular, for every $i \in [1, r]$, one has $\{c_i, \vphi\}_{\pist} = 0$ for every $\vphi \in \Phi({\bf i})$, and 
since $\Phi({\bf i})$ is a set of transcendental generators of $\CC(G)$, $c_i$ is a rational Casimir function on $(G, \pist)$.
\end{proof}

We can now state the main result of this section. For any extended cluster
$\Phi$ of $\calC_{{\rm BFZ}}(G)$, let 
\[
\overline{\Phi} = (\Phi\backslash {\rm Froz}) \sqcup \overline{{\rm Froz}}.
\]

\begin{thm} \label{thm:mainG}
The cluster structure $\calC_{\rm BFZ}(G)$ has Property $\calI$ at the identity element $e \in G$
after the frozen variable modification $\overline{{\rm Froz}}$. Consequently, for every extended cluster 
$\Phi$ of $\calC_{{\rm BFZ}}(G)$, the set of lowest degree terms
of the functions in $\overline{\Phi}$ at $e$ contains a polynomial integrable system on $(\g, \pi_0)$.
\end{thm}

\begin{proof} By Lemma-Definition \ref{lem:same-volume} and 
Lemma-Definition \ref{lem:mu-modify}, the log-volume form defined by all the extended clusters, respectively by their
modifications by $\overline{{\rm Froz}}$, agree up to multiplication by $-1$. For ${\bf i}$  in \eqref{eq:bfi}
and the extended cluster $\Phi({\bf i})$ in \eqref{eq:Phi-G}, 
we now compute the log-volume forms $\mu_{\Phi({\bf i})}$ and $\mu_{\overline{\Phi({\bf i})}}$ in some holomorphic coordinates on $G$ near $e$.

Consider the open subset $B_-B$ of $G$ with the unique decomposition $B_-B = N_- N T$. 
Using the double reduced word ${\bf i}$ in \eqref{eq:bfi}, we parametrize $B_-B$ by
$\CC^{2\elw} \times (\CC^\times)^r$ via
\begin{equation}\label{eq:xi-y-z}
\CC^{2\elw} \times (\CC^\times)^r  \ni (y, z, \xi)= 
(y_1, \,\ldots, \,y_{\elw},\, z_1, \,\ldots, \,z_{\elw}, \, \xi_1, \,\ldots, \,\xi_r)\longmapsto 
g_-(y) g_+(z) t,
\end{equation}
where  $t = \alpha_1^\vee(\xi_1) \cdots \alpha_r^\vee(\xi_r)\in T$, and
\[
\overline{w_0}\,  (g_-(y))^{-1}  = e_{j_1^*}(y_1) \overline{s_{j_1^*}} \, \cdots \, 
e_{j_{\elw}^*}(y_{\elw}) \overline{s_{j_{\elw}^*}}, \hs 
g_+ (z)\, \overline{w_0} = e_{i_1}(z_1) \overline{s_{i_1}} \, \cdots \, 
e_{i_{\elw}}(z_{\elw}) \overline{s_{i_{\elw}}}.
\]
In the $(y, z, \xi)$-coordinates on $B_-B$ thus obtained, the restrictions of functions in $\Phi({\bf i})$ to $B_-B$
are given by $\Delta_{\omega_i, \omega_i} = \xi_i$ for every $i \in [1, r]$, and, using
\eqref{eq:useful} and Lemma \ref{lem:w0-g}, for  $k \in [1, \elw]$,
\begin{align}\nonumber
\Delta_{\omega_{i_k}, \; s_{i_1} \cdots s_{i_{k}} \omega_{i_k}}& = t^{s_{i_1} \cdots s_{i_{k}} \omega_{i_k}}
\Delta_{\omega_{i_k}, \; s_{i_1} \cdots s_{i_{k}} \omega_{i_k}}(g_+(z))\\
\label{eq:kk-1}
&=\left(\prod_{i=1}^r \xi_i^{(s_{i_1} \cdots s_{i_{k}} \omega_{i_k}, \, \alpha_i^\vee)}\right) 
\Delta_{\omega_{i_k}, \;  \omega_{i_k}}(e_{i_1}(z_1) \overline{s_{i_1}} \, \cdots \, 
e_{i_k}(z_k) \overline{s_{i_k}}),\\
\nonumber \Delta_{s_{j_1} \cdots s_{j_{k}} \omega_{j_k}, \; \omega_{j_k}} &= t^{\omega_{j_k}}
\Delta_{s_{j_1} \cdots s_{j_{k}} \omega_{j_k}, \; \omega_{j_k}}(g_-(y)) =\xi_{j_k} 
\Delta_{\omega_{j_k^*}, \; s_{j_1^*} \cdots s_{j_{k}^*} \omega_{j_k^*}}(\overline{w_0} \,(g_-(y))^{-1}\,\overline{w_0}^{\, -1})
\\
\label{eq:kk-2}&= \xi_{j_k} 
\Delta_{\omega_{j_k^*}, \;  \omega_{j_k^*}}(e_{j_1^*}(y_1) 
\overline{s_{j_1^*}} \, \cdots \, 
e_{j_k^*}(y_k) \overline{s_{j_k^*}}).
\end{align}
By Corollary \ref{cor:TPf} applied to $w = w_0$, 
\[
\mu_{\Phi({\bf i})} = \frac{d\xi_1 \wedge \cdots \wedge d\xi_r}{\xi_1 \cdots \xi_r} \wedge 
\frac{dy_1 \wedge \cdots \wedge dy_{\elw}}{\prod_{\alpha \in \Gamma} \vphi_{w_0, \alpha}(y)}\wedge
\frac{dz_1 \wedge \cdots \wedge dz_{\elw}}{\prod_{\alpha \in \Gamma} \vphi_{w_0, \alpha}(z)}.
\]
Taking the lowest degree term at $e \in G$, one knows by \eqref{eq:low-mu-w0} that 
\[
\deg \left(\mu_{\Phi({\bf i})}^\low\right) = r+ {\rm rk}(\pi_{0, \n_-}) = r + \elw -\dim \ker (1+w_0).
\]
By Proposition \ref{prop:ind-gsts}, ${\rm rk}(\g, \pi_0) = 2\elw$. If $w_0 = -1$ on $\t$, we have then proved that $\Phi({\bf i})$, and thus 
the  cluster structure $\calC_{\rm BFZ}(G)$ in $\CC(G)$, has Property $\calI$ at $e$.
When $w_0$ does not act as $-1$ on $\t$, we consider the modification $\overline{\Phi({\bf i})}$
of $\Phi({\bf i})$ by 
$\overline{{\rm Froz}}$. It follows from the definitions that, up to a sign,
    \begin{align*}
        {\mu}_{\overline{\Phi({\bf i})}} = \left( \prod \limits_{i \in I_2} \frac{f_ig_i}{g'_i} \right) \mu_{\Phi({\bf i})}
        =\left( \prod \limits_{i \in I_2} \frac{f_ig_i}{f_ig_i-f_{i^*}g_{i^*}} \right) \mu_{\Phi({\bf i})}.
    \end{align*}
For every $i \in [1, r]$, in the
$(\xi, y, z)$ coordinates we have 
$f_i = (\xi_{i^*})^{-1} \vphi_{w_0, \alpha_i}(z)$ and $g_i = \xi_i \vphi_{w_0, \alpha_{i^*}}(y)$, so 
\[
f_i g_i = \frac{\xi_i}{\xi_{i^*}} \vphi_{w_0, \alpha_i}(z) \vphi_{w_0, \alpha_{i^*}}(y)
\hs \mbox{and} \hs 
f_{i^*}g_{i^*}= \frac{\xi_{i^*}}{\xi_{i}} \vphi_{w_0, \alpha_{i^*}}(z) \vphi_{w_0, \alpha_{i}}(y).
\]
For $i \in [1, r]$, write  $\xi_i = 1 + \eta_i$ so that $\eta_i=0$ at $e$. Then for $i \in I_2$, 
\begin{align*}
f_ig_i -f_{i^*} g_{i^*}&= \frac{1}{\xi_i\xi_{i^*}}(\xi_i^2 \vphi_{w_0, \alpha_i}(z) \vphi_{w_0, \alpha_{i^*}}(y)-
\xi_{i^*}^2 \vphi_{w_0, \alpha_{i^*}}(z) \vphi_{w_0, \alpha_{i}}(y))\\
& = \frac{1}{(1+\eta_i)(1+\eta_{i^*})}((\eta_i^2 +2\eta_i+1)\vphi_{w_0, \alpha_i}(z) \vphi_{w_0, \alpha_{i^*}}(y)-
(\eta_{i^*}^2 +2\eta_{i^*}+1) \vphi_{w_0, \alpha_{i^*}}(z) \vphi_{w_0, \alpha_{i}}(y)).
\end{align*}
Same as in Remark \ref{rem:v-z}, 
we abuse notation and write
\[
\eta_i = d_e\eta_i, \hs y_k = d_ey_k, \hs z_k = d_ez_k, \hs i \in [1, r], \, k \in [1, \elw],
\]
and we regard $(\eta, y, z) = (\eta_1, \ldots, \eta_r, y_1, \ldots, y_{\elw}, 
z_1, \ldots, z_{\elw})$ as a linear coordinate system on $\g = T_eG$. 
By Lemma \ref{lem:deg-cdeg}  and Lemma \ref{lem:di}, for $i \in I_2$ one has
$\vphi_{w_0, \alpha_i}^\low = a_i \vphi_{w_0, \alpha_{^*}}^\low$ with  $a_i = (-1)^{(2\rho^\vee - \kappa^\vee, \, \omega_i)}$, so
\[
(f_i g_i)^\low = \vphi_{w_0, \alpha_i}^\low(z) \vphi_{w_0, \alpha_{i^*}}^\low(y)
 =\vphi_{w_0, \alpha_{i^*}}^\low(z) \vphi_{w_0, \alpha_{i}}^\low(y) =(f_{i^*}g_{i^*})^\low.
\]
and since $\eta_i \neq \eta_{i^*}$, $(f_ig_i -f_{i^*} g_{i^*})^\low$ contains the term
\[
2\eta_i \vphi_{w_0, \alpha_i}^\low(z) \vphi_{w_0, \alpha_{i^*}}^\low (y) -2\eta_{i^*} 
\vphi_{w_0, \alpha_{i^*}}^\low(z) \vphi_{w_0, \alpha_{i}}^\low(y)=2(\eta_i-\eta_{i^*})
\vphi_{w_0, \alpha_i}^\low(z) \vphi_{w_0, \alpha_{i^*}}^\low (y),
\]
so there exists a homogeneous polynomial $A_i(y, z) \in \CC[y, z]$ of degree 
$1 + \deg ((f_ig_i)^\low)$ such that 
\begin{equation}\label{eq:gi-low}
(f_ig_i -f_{i^*} g_{i^*})^\low  = 2(\eta_i-\eta_{i^*})
\vphi_{w_0, \alpha_i}^\low(z) \vphi_{w_0, \alpha_{i^*}}^\low (y) + A_i(y, z).
\end{equation}
In particular, $\deg ((f_ig_i -f_{i^*} g_{i^*})^\low) = 1 + \deg ((f_ig_i)^\low)$. It now follows that 
\[
\deg \left({\mu}_{\overline{\Phi({\bf i})}}^\low\right) = \deg \left({\mu}_{{\Phi({\bf i})}}^\low\right)-|I_2|=
\elw +r-\dim \ker (1+w_0)-|I_2| = \elw=\frac{1}{2}{\rm rk}(\g, \pi_0).  \tag*{\qedhere}
\]
\end{proof}


\begin{rem} \label{rem:degofPf}
We see in the proof of Theorem \ref{thm:mainG} that the lowest degree term of the $T$-Pfaffian of $(G, \pist)$ at $e \in G$ is
equal to $-\elw-r + \dim \ker(1+w_0)$.
\end{rem}

\begin{rem} \label{rem:int-delta-G}
{\rm Suppose that  $w_0$ acts $\t$ acts as $-1$, so that  $\overline{\rm Froz} = {\rm Froz}$. 
Recall that the initial extended cluster in $\calC_{\rm BFZ}(G)$ constructed in \cite{BFZ:III} using any double reduced word ${\bf i}$ of $(w_0, w_0)$ consists of generalized minors. Thus the polynomial integrable systems
on $(\g, \pi_0)$ we obtain via Theorem \ref{thm:mainG} from such extended clusters all consist of signed 
generalized minors on $\g$. When $w_0$ does not act $\t$ as $-1$, the polynomials $(g_i^\prime)^\low$ on 
$\g$ for $i \in I_2$ are not necessarily signed minors, as we will see for the case of $\SL(n,\CC)$ in 
Example \ref{ex:An-1}.
}
\end{rem}

\subsection{Rational Casimirs of \texorpdfstring{$(\g, \pi_0)$}{(g,pi0)}}\label{ss:casi-G}
Let $\CC\left(\overline{\rm Froz}^\low\right)$ be the subfield of $\CC(\g)$ generated by 
\[
\overline{\rm Froz}^\low = \{f_i^\low =\delta_{\omega_i, w_0\omega_i}: i\in[1,r]\}\cup\{ 
g_i^\low =\delta_{w_0\omega_i, \omega_i}: i \in I_0 \sqcup I_1\}
\cup\{(g_i^\prime)^\low: i \in I_2\} \subset \CC[\g].
\]
Recall that the index of the Lie algebra $\gsts$ is 
$r = \dim \t$. We now construct $r$ explicit independent rational Casimir functions on 
$(\g, \pi_0) \cong ((\gsts)^*, \, \pi_{0, \gsts})$ that are contained in 
$\CC\left(\overline{\rm Froz}^\low\right)$.

Recall that for each $i \in [1, r]$, we have the 
rational Casimir function $c_i = f_ig_{i^*}^{-1}$ of $(G, \pist)$, and note that the lowest degree term at $e$ of any
non-zero rational functions of $(c_1, \ldots, c_r)$ is a rational Casimir on $(\g, \pi_0)$.
Let the decomposition $[1, r] = I_0 \sqcup I_1 \sqcup I_2$ be as in \eqref{eq:III} and keep the notation as in the 
proof of Theorem \ref{thm:mainG}.
If $i \in I_0 \sqcup I_1$, then 
\[
c_{i^*}^\low = \frac{f_{i^*}^\low}{g_{i}^\low} =\frac{\vphi_{w_0, \alpha_{i^*}}^\low(z)}{\vphi_{w_0, \alpha_{i}}^\low(y)}
\in  \CC\left(\overline{\rm Froz}^\low\right).
\]
If $i \in I_2$, then $c_ic_{i^*}^{-1}-1 = g_i^\prime (f_{i^*}g_{i^*})^{-1}$, so 
\[
(c_ic_{i^*}^{-1}-1)^\low  = \frac{(g_i^\prime)^\low}{f_{i^*}^\low g_{i^*}^\low}=
\frac{2(\eta_i-\eta_{i^*})
\vphi_{w_0, \alpha_i}^\low(z) \vphi_{w_0, \alpha_{i^*}}^\low (y) + A_i(y, z)}{
\vphi_{w_0,\alpha_{i^*}}^\low(z) \vphi_{w_0,\alpha_{i^*}}^\low(y)} \in 
\CC\left(\overline{\rm Froz}^\low\right).
\]
It now follows from the algebraic independence of $\{\delta_{\omega_i, w_0\omega_i}: i \in I_0 \sqcup I_1\}$ and the 
dependence of $(g_i^\prime)^\low$ on $\eta_i-\eta_{i^*}$ for $i \in I_2$ that 
\[
\left\{c_{i^*}^\low: i \in I_0 \sqcup I_1\right\} \sqcup 
\left\{(c_ic_{i^*}^{-1}-1)^\low: i \in I_2\right\}
\]
is a set of $r$ algebraically independent rational Casimir functions on 
$(\g, \pi_0)$ contained in 
$\CC\left(\overline{\rm Froz}^\low\right)$.

\subsection{Examples of polynomial integrable systems on \texorpdfstring{$(\g, \pi_0)$}{(g,pi0)}} Consider any reduced word
${\bf w}_0 = (i_1, \,i_2, \, \ldots, \, i_{\elw})$ 
of $w_0$, and recall that $d_{w_0} = \frac{1}{2}(\elw + \dim \ker(1+w_0))$. Set
\[
\delta({\bf w}_0) = \{\delta_{\omega_{i_k}, \, s_{i_1}\cdots s_{i_k} \omega_{i_k}} \in \CC[\g]: k \in [1, \elw]\},
\hs
\delta({\bf w}_0)|_\n = \{\delta_{\omega_{i_k}, \, s_{i_1}\cdots s_{i_k} \omega_{i_k}}|_\n \in \CC[\n]: k \in [1, \elw]\}.
\]
By Theorem  \ref{thm:choice-GB} applied to $w = w_0$, all maximal independent subsets of $\delta({\bf w}_0)|_\n$ 
have cardinality $d_{w_0}$, which, by Lemma \ref{lem:delta-1}, implies that
all maximal independent subsets of $\delta({\bf w}_0)$ 
have cardinality $d_{w_0}$. Similarly, let ${\bf w}_0^\prime = (j_1, \,j_2, \, \ldots, \, j_{\elw})$ be any other 
reduced word of $w_0$ and set
\[
\delta^\prime ({\bf w}_0^\prime)=\{\delta_{s_{j_1}\cdots s_{j_k} \omega_{j_k}, \, \omega_{j_k}} \in \CC[\g]: k \in [1, \elw]\}.
\]
Then  all maximal independent subsets of $\delta^\prime ({\bf w}_0^\prime)$ have cardinality $d_{w_0}$
by Lemma \ref{lem:delta-w0}.
Let  
\[
S_0 = \{(g_i^\prime)^\low = (\Delta_{\omega_i, w_0\omega_i} \Delta_{w_0\omega_i, \omega_i} -
\Delta_{\omega_{i^*}, w_0\omega_{i^*}} \Delta_{w_0\omega_{i^*}, \omega_{i^*}})^\low:  i \in I_2\},
\]
where the lowest degree term is taken at $e \in G$. Consider the set 
\begin{equation}\label{eq:set}
\delta(\bfw_0) \sqcup \delta^\prime(\bfw_0^\prime) \sqcup S_0\subset \CC[\g]
\end{equation}

\begin{thm}\label{thm:choice-G}
For any two reduced words $\bfw_0$ and $\bfw_0^\prime$ of $w_0$, the set in \eqref{eq:set}
of homogeneous polynomial functions on $\g$ is $\pi_0$-involutive, and all of its elements 
have quasi-polynomial Hamiltonian flows with respect to $\pi_0$ (see Definition \ref{defn:quasi-poly}). Moreover,  for any maximal independent subset 
$S$ of $\delta(\bfw_0)$ and any maximal independent subset 
$S'$ of $\delta^\prime(\bfw_0^\prime)$, the set
$S \sqcup S' \sqcup S_0$
is a polynomial integrable system on $(\g, \pi_0)$. 
\end{thm}

\begin{proof}
The set in \eqref{eq:set}, being a subset of $\overline{\Phi({\bf i})}^\low$, is $\pi_0$-involutive by Theorem \ref{thm:mainG}.
Consider the coordinates $(\xi, y, z) \in (\CC^\times)^r \times \CC^{2\elw}$ on $B_-B$ defined using the
two reduced words $\bfw_0$ and $\bfw_0^\prime$ of $w_0$ given in \eqref{eq:xi-y-z}. 
 Consider again the linear coordinate system $(\eta, y, z)$ on $\g = T_e(B_-B)$ as in
the proof of Theorem \ref{thm:mainG}. 
Under the identification $\gsts \cong \g^*$ and for $i \in [1, r]$ and $k \in [1, \elw]$,  we then have
\begin{equation}\label{eq:eta-y-z}
\eta_i \in \{(x_0, -x_0): x_0 \in \t\} \subset \gsts, \hs y_k \in \n\oplus 0 \subset \gsts,\hs  z_k \in 0\oplus \n_-  \subset \gsts.
\end{equation}
In the linear coordinates $(\eta, y, z)$ on $\g$, functions in $\delta(\bfw_0)$ (resp. $\delta'(\bfw_0^\prime)$) depend only on $z$ (resp. $y$), 
and recall from  \eqref{eq:gi-low} that 
$(g_i^\prime)^\low = 2(\eta_i-\eta_{i^*})
\vphi_{w_0, \alpha_i}^\low(z) \vphi_{w_0, \alpha_{i^*}}^\low (y)+A(y, z)$
for some $A(y, z)\in \CC[y, z]$. It follows that for any  $S \subset \delta(\bfw_0)$
and $S' \subset \delta'(\bfw_0^\prime)$ as described, elements in $S \sqcup S' \sqcup S_0$ are independent, and by 
Proposition \ref{prop:w0}, the cardinality of $S \sqcup S' \sqcup S_0$
is 
\[
2d_{w_0} + |I_2|= \elw + \dim \ker (1+w_0) + |I_2| = \elw + r,
\]
which, by Proposition \ref{prop:ind-gsts}, is the magic number of the Lie algebra 
$\gsts$.

It remains to prove that every function $\varphi$ in the set in \eqref{eq:set} has quasi-polynomial  Hamiltonian flow.
Note first that the Poisson bracket $\{\,, \, \}_{\pi_0}$ between the linear coordinates
$(\eta, y, z)$ on $\g$ coincides with  their Lie brackets as elements in $\gsts$ via \eqref{eq:eta-y-z}. 
Suppose that $\vphi \in \delta(\bfw_0) \subset \CC[z]$, and let 
$\gamma(t) = (\eta(t), y(t), z(t))$ be an integral curve of the Hamiltonian flow of $\vphi$. Then $y(t)$ is a
constant, and by Proposition \ref{prop:quasi-poly-GB}, $z(t)$ is quasi-polynomial. Moreover, for each $i \in [1, r]$,
$\frac{d\eta_i(t)}{dt}$ is a polynomial in $(y(t), z(t))$, so $\eta_i(t)$ is quasi-polynomial
in $t$. Thus $\gamma(t)$ is
quasi-polynomial. Similarly, every function in $\delta^\prime(\bfw_0^\prime)$ has quasi-polynomial  Hamiltonian flow.

Let now $i \in I_2$, and recall that $g_i' = (c_i - c_{i^{\ast}}) g_i g_{i^{\ast}}$,
where $c_i, c_{i^*} \in \CC(G)$ are Casimirs. It thus follows from $(g_i^\prime)^\low = (c_i - c_{i^{\ast}})^\low 
\vphi_{w_0, \alpha_i}^\low(y)\vphi_{w_0, \alpha_{i^*}}^\low(y)$ that  
$\{z_k, (g_i^\prime)^\low\}_{\pi_0}=0$ for every $k \in [1, \elw]$. By Lemma \ref{lem:low-of-frozen}, one also has
$\{y_k, (g_i^\prime)^\low\}_{\pi_0} =0$ for every $k \in [1, \elw]$. 
Let 
$\gamma(t) = (\eta(t), y(t), z(t))$ be an integral curve of the Hamiltonian vector field of $(g_i^\prime)^\low$.
Then $y(t)$ and $z(t)$ are both constants. For each $i' \in [1, r]$, it follows from \eqref{eq:gi-low} that 
$\{\eta_{i'}, (g_i^\prime)^\low\}_{\pi_0} = 2(\eta_i-\eta_{i^*})A_1(y, z) + A_2(y, z)$ for 
some $A_1(y, z), A_2(y, z) \in \CC[y, z]$. As $y(t)$ and $z(t)$ are constants, $\eta_{i'}(t)$ is quasi-polynomial.
\end{proof}

\begin{rem}
    The maximal independent subsets $S, S'$ in Theorem \ref{thm:choice-G} can be taken, for example, as in Theorem \ref{thm:choice-GB} or in Theorem \ref{thm:psi}.
\end{rem}

\begin{example}\label{ex:An-1}
Let $G = \SL(n+1, \CC)$ and let $\bfw_0 = \bfw_0^\prime = (s_1, \ldots, s_n, s_1, \ldots, s_{n-1}, \ldots, s_1, s_2, s_1)$.
Using the exponential map $\exp: \g \rightarrow G$ we regard entries of $u =(u_{i,j})\in \g = \mathfrak{sl}(n+1, \CC)$ 
as a local holomorphic coordinate system on $G$ near the identity element $e \in G$.   For subsets $I, J$ of $[1,n+1]$ with $|I|=|J|$, we write $\Delta_{I,J}(u)$ for the determinant of the submatrix of $u$ consisting of those rows labeled by $I$ and columns labeled by $J$. Take $I_2 = [\frac{n}{2}+1, n]$ if $n$ is even and $[\frac{n+1}{2}+1, n]$ if $n$ is odd. For $i \in I_2$, a direct calculation gives 
    \begin{align*}
        (g_i')^\low =\quad &\sum \limits_{n-i+1 < k < i+1} \Delta_{[i+1,n+1],[1,n-i+1]}(u) \Delta_{[1,n-i+1] \cup \{k\},[i+1,n+1] \cup \{k\}}(u) \\
         + &\sum \limits_{n-i+1 < k < i+1} \Delta_{[1,n-i+1],[i+1,n+1]}(u) \Delta_{[i+1,n+1] \cup \{k\},[1,n-i+1] \cup \{k\}}(u).
    \end{align*}
\end{example}

\section{The Poisson Lie group \texorpdfstring{$(B, \pist)$}{(B,pi)} }\label{s:B}
\subsection{The index of the Lie algebra \texorpdfstring{$\b_-$}{b-}}\label{ss:ind-b} Recall that through the identification 
$\b^* \cong \b_-$ via the non-degenerate pairing
$\lara_{(\b, \b_-)}$ given in \eqref{eq:bb-pair}, the Lie bialgebra of the Poisson Lie group
$(B, \pist)$ becomes $(\b, \b_-)$, and the linarization $(\b, \pi_0)$  of $(B, \pist)$ at
the identity element
$e \in B$ becomes the 
Kirillov-Kostant-Souriau  Poisson structure $\pi_{0, \b_-}$ on $\b \cong \b_-^*$. Writing
\[
(\b, \, \pi_0) \cong (\b_-^*, \, \pi_{0, \b_-}),
\]
we thus have ${\rm rk}(\b, \pi_0) = \dim \b - \ind(\b_-)$, where $\ind(\b_-)$ is the index of the Lie algebra $\b_-$.
The following statement on $\ind(\b_-)$ is given in \cite[Remark 1.5.1]{Panyushev:index}.  We provide a proof using the Poisson structure
$\pist$.

\begin{lem}\label{lem:ind-bm}
One has $\ind(\b_-) = \dim {\rm im}(1+w_0)$ and  ${\rm rk}(\b, \pi_0)  = \ell(w_0)+\dim {\rm ker}(1+w_0)$.
\end{lem}

\begin{proof}
Note first that since $w_0^2 = 1$, one has ${\rm im}(1-w_0) \subset\ker (1+w_0)$. By writing an arbitrary
$x \in \t$ as $x = \frac{1}{2}(x-w_0 x) + \frac{1}{2}(x + w_0x)$, one sees that $\ker (1+w_0) \subset {\rm im}(1-w_0)$. Thus 
${\rm im}(1-w_0) =\ker (1+w_0)$.
By \cite[Proposition 2.4]{KZ:leaves}, the rank of $\pi$ in $G^{e, w_0}$ is equal to 
\[
\elw + \dim {\rm im}(1-w_0) = \elw + \dim \ker(1+w_0).
\]
As $G^{e, w_0}$ is an open dense  $T$-leaf of $(B, \pist)$, by Lemma \ref{lem:rk-L-P} 
one has ${\rm rk}(B, \pist) = \elw + \dim \ker(1+w_0)$.
Since $(\b, \pi_0)$ is the linearization of $(B, \pist)$ at
$e$, one has
${\rm rk}(\b, \pi_0) \leq {\rm rk}(B, \pist)$, so
\[
\ind(\b_-) = \dim \b - {\rm rk}(\b, \pi_0) \geq  \dim \b - {\rm rk}(B, \pist) =  \dim \b - \elw
-\dim \ker(1+w_0)= \dim {\rm im} (1+w_0).
\]
Under the identification $\b \cong \b_-^*$  via the pairing $\lara_{(\b, \b_-)}$, 
the stabilizer sub-algebra ${\rm Stab}_{\b_-}(x) \subset \b_-$
of the co-adjoint representation of $\b_-$ on $\b$ through any $x = x_+ + x_0\in \b$, where 
$x_+ \in \n$ and $x_0 \in \t$,
is given by
\[
{\rm Stab}_{\b_-}(x_+ + x_0) = \{\xi\in \b_-: [\xi, \, x_+] \in \n_-\}.
\]
Consider again Kostant's cascade $\calB$ of roots, and let 
$e_+ \in \n$ be as in \eqref{eq:e-pm}. The same argument as in the proof of Proposition 
\ref{prop:ind-gsts} shows that ${\rm Stab}_{\b_-}(e_+) = \{\xi_0 \in \t: \beta(\xi_0) = 0\,\forall
\beta \in \calB\}$, 
 which has dimension equal to $r-|\calB| =r-\dim \ker (1+w_0) = \dim {\rm im} (1+w_0)$. We thus conclude that
$\ind(\b_-) = \dim {\rm im} (1+w_0)$.
\end{proof}

Recall that $d_{w_0} = \frac{1}{2}(\elw + \dim \ker (1+w_0))$ is the magic number of $(\n, \pi_{w_0, 0})$ which is also the magic number of  the Lie algebra $\n_-$.

\begin{cor}\label{cor:nagic-bm}
The rank and the magic number of $(\b, \pi_0)$ are respectively given by
\begin{align*}
{\rm rk}(\b, \pi_0) &= 2d_{w_0} = \elw + \dim \ker (1+w_0)\\
{\rm mag}(\b, \pi_0) &= d_{w_0} + \dim {\rm im}(1+w_0) = r + \frac{1}{2}(\elw - \dim \ker (1+w_0)).
\end{align*}
\end{cor}

\subsection{The BFZ cluster structure \texorpdfstring{$\calC_{\rm BFZ}(B)$}{Cbfz(B)}  in \texorpdfstring{$\CC(B)$}{C(B)}}\label{ss:CBFZ-B}
Let $\bfw_0 = (s_{i_1}, \ldots, s_{i_{\elw}})$ be any reduced word of $w_0$. In the notation of \cite{BFZ:III}
and by \cite[$\S$2.3]{BFZ:III}, the sequence ${\bf i} = (i_{\elw}, \ldots, i_2, i_1)$ gives rise to 
the set $\Phi_B({\bf i})$ of $r + \elw = \dim B$
regular functions on $B$, where
\begin{align}\nonumber
\Phi_B({\bf i}) &=\{\Delta_{\omega_i, \, w_0\omega_i}|_B: i \in [1, r]\} \sqcup 
\{\Delta_{\omega_{i_k}, \, s_{i_1} \cdots s_{i_{k-1}} \omega_{i_k}}|_B: k \in [1, \elw]\}\\
\label{eq:Phi-B}&=
\{\Delta_{\omega_i, \, \omega_i}|_B: i \in [1, r]\} \sqcup 
\{\Delta_{\omega_{i_k}, \, s_{i_1} \cdots s_{i_{k}} \omega_{i_k}}|_B: k \in [1, \elw]\}.
\end{align}
In \cite[$\S$2.3]{BFZ:III}, Berenstein, Fomin and Zelevinsky constructed a seed $\Sigma_B({\bf i})$ 
in $\CC(B) = \CC(G^{e, w_0})$
which has $\Phi_B({\bf i})$ in \eqref{eq:Phi-B} as an extended cluster and 
\[
{\rm Froz}_B = \{\Delta_{\omega_i, \, \omega_i}|_B, \; \Delta_{\omega_i, \, w_0\omega_i}|_B: \;i \in [1, r]\}
\]
as the frozen variables. We denote by $\mathcal{C}_{\rm BFZ}(B)$ the mutation
equivalence class of seeds in $\CC(B)$ defined by $\Sigma_B({\bf i})$ and call
$\calCBFZ(B)$  the 
{\it BFZ cluster structure in $\CC(B)$}.
By \cite[Paragraph after Theorem 1.1]{ShenWeng:DBS},
$\mathcal{C}_{\rm BFZ}(B)$ is independent of the choice of the reduced word for $w_0$. 

\begin{lem}\label{lem:comp-T-B}
The cluster structure $\calCBFZ(B)$ in $\CC(B)$ is regular on $B$
and compatible with both the Poisson structure $\pist$ and the $T$-action on $B$.
\end{lem}

\begin{proof}
The second statement is proved in \cite[Theorem 4.18]{GSV:book}. 
We now use the Starfish Lemma \cite[Proposition 6.4.1]{FWZ:online} to prove the first statement. 
As in the proof of Theorem \ref{thm:mainG}, 
we use the reduced word ${\bf w}_0 = (s_{i_1}, \ldots, s_{i_{\elw}})$ of $w_0$ to parametrize $B$ by
$\CC^{\elw} \times (\CC^\times)^r$ via
\begin{equation}\label{eq:xi-z}
\CC^{\elw} \times (\CC^\times)^r  \ni (z, \xi)= 
(z_1, \,\ldots, \,z_{\elw}, \, \xi_1, \,\ldots, \,\xi_r)\longmapsto 
g_+(z) t,
\end{equation}
where  again $t = \alpha_1^\vee(\xi_1) \cdots \alpha_r^\vee(\xi_r)\in T$ and
$g_+ (z)\, \overline{w_0} = e_{i_1}(z_1) \overline{s_{i_1}} \, \cdots \, 
e_{i_{\elw}}(z_{\elw}) \overline{s_{i_{\elw}}}$.
In the $(z, \xi)$-coordinates on $B$ thus obtained, elements in $\Phi_B({\bf i})$ are given by 
$\Delta_{\omega_i, \omega_i}|_B = \xi_i$ for $i \in [1, r]$ and 
\[
\Delta_{\omega_{i_k}, \; s_{i_1} \cdots s_{i_{k}} \omega_{i_k}}|_B=\left(\prod_{i=1}^r \xi_i^{(s_{i_1} \cdots s_{i_{k}} \omega_{i_k}, \, \alpha_i^\vee)}\right) 
\Delta_{\omega_{i_k}, \;  \omega_{i_k}}(e_{i_1}(z_1) \overline{s_{i_1}} \, \cdots \, 
e_{i_k}(z_k) \overline{s_{i_k}}), \hs k \in [1, \elw].
\]
Thus the elements in $\Phi_B({\bf i})$ are all prime elements of $\CC[B]$ and are
pairwise coprime (see, for example, \eqref{eq:phik-CGL}). Let $\psi_k$ be a 
cluster variable in $\Phi_B({\bf i})$ and let $\psi_k^\prime \in \CC(B)$ be the new cluster variable after mutating 
$\Phi_B(\bf i)$ in direction $k$. By \cite[Theorem 4.3]{Z:component} and \cite[Proof of Theorem 2.10]{BFZ:III}, 
$\psi_k^\prime$ is regular on $B$. 
Moreover, by \cite[Theorem 1.3]{GLS:factorial}, $\psi_k^\prime|_{G^{e, w_0}}$ is irreducible in $\CC[G^{e, w_0}]$.
Since  $\CC[G^{e, w_0}]$ is the localization of
$\CC[B]$ at the set  $F=\{\Delta_{\omega_i, w_0\omega_i}|_B: i \in [1, r]\}$, up to associates the element
$\psi_k^\prime \in \CC[B]$ has exactly one prime factor $p$ in $\CC[B]$ that does not lie in $F$. The two prime elements $\psi_k$ and $p$ of $\CC[B]$
can not be associates in $\CC[B]$, for otherwise the mutation relation between $\psi_k$ and $\psi_k^\prime$ would
produce an algebraic relation between the variables in the extended cluster obtained from $\Phi_B({\bf i})$ by mutation in direction $k$, which we know
is algebraically independent. We thus conclude that $\psi_k$ and $\psi_k^\prime$ are coprime in $\CC[B]$.
By Starfish Lemma, all cluster variables of $\calCBFZ(B)$ are regular functions on $B$.
\end{proof}

Similar to the cluster structure $\calCBFZ(G)$ in $\CC(G)$, when $w_0\neq -1$ on $\t$ we need a regular $\pist$-compatible 
frozen variable modification of $\calCBFZ(B)$. To this end, for $i \in [1, r]$, let 
$p_i = \Delta_{\omega_i, \omega_i}|_B$, $q_i = \Delta_{\omega_i, w_0\omega_i}|_B$, and
\[
\theta_i = \frac{p_iq_i}{p_{i^*}q_{i^*}} \in \CC(B).
\]
Let the decomposition $[1, r] = I_0 \sqcup I_1 \sqcup I_2$ be as in \eqref{eq:III}. Then $\theta_i = 1$ for $i \in I_0$
and $\theta_i\theta_{i^*}=1$ for $i \in I_1 \sqcup I_2$. For $i \in I_2$, let 
$a_i = (-1)^{(2\rho^\vee - \kappa^\vee, \, \omega_i)}$ (see notation from Lemma \ref{lem:di}) and define
\[
q_i^\prime = p_iq_i -a_i p_{i^*}q_{i^*} 
= (\theta_i -a_i) p_{i^*}q_{i^*}\in \CC(B).
\]
Recall that ${\rm Froz}_B = \{p_i, q_i: i \in [1, r]\}$. Introduce
\[
\overline{{\rm Froz}_B} = \{p_i: i \in [1, r]\} \sqcup \{q_i: i \in I_0\sqcup I_1\} \sqcup\{q_i^\prime: i \in I_2\}.
\]

\begin{lem}
1) For each $i \in I_2$, $\theta_i \in \CC(B)$ is a rational Casimir function on $(B, \pist)$;

2) The set $\overline{{\rm Froz}_B}$ is a regular $\pist$-compatible frozen variable modification of $\calCBFZ(B)$.
\end{lem}

\begin{proof} Since $\overline{{\rm Froz}_B}$ consists of regular functions on $B$ and freely generates
the subfield $\CC({\rm Froz}_B)$ of $\CC(B)$,  
by Definition \ref{defn:regular}, we only need prove 1).  
By \cite[Proposition 4.19]{GSV:book}, for $i, k \in [1, r]$ and $v \in W$  one has, and as functions on $G$,
\begin{align*}
&\{\Delta_{\omega_i, \,\omega_i}, \; \Delta_{\omega_k,\, v\omega_k}\}_{\pist} = \la \omega_i,\, v \omega_k -\omega_k\ra
\Delta_{\omega_i,\, \omega_i} \Delta_{\omega_k,\, v\omega_k},\\
& \{\Delta_{\omega_i, \,w_0\omega_i}, \; \Delta_{\omega_k,\, v\omega_k}\}_{\pist} 
=(\la \omega_i, \, \omega_k\ra +  \la \omega_{i^*}, \, v \omega_k \ra)
\Delta_{\omega_i,\, w_0\omega_i}\Delta_{\omega_k, v\omega_k},
\end{align*}
so  $\{\Delta_{\omega_i, \,\omega_i}\Delta_{\omega_i, \,w_0\omega_i},\; \Delta_{\omega_k,\, v\omega_k}\}_{\pist} = 
\la \omega_i+\omega_{i^*},\, v \omega_k\ra
\Delta_{\omega_i, \,\omega_i}\Delta_{\omega_i, \,w_0\omega_i} \Delta_{\omega_k,\, v\omega_k}$, and thus 
\[
\left\{\frac{\Delta_{\omega_i, \,\omega_i}\Delta_{\omega_i, \,w_0\omega_i}}{\Delta_{\omega_{i^*}, \,\omega_{i^*}}\Delta_{\omega_{i^*}, \,w_0\omega_{i^*}}}, \; \; \Delta_{\omega_k,\, v\omega_k}\right\}_{\pist} = 0 \in \CC(G).
\]
Restricting to $B$ and noting that the elements in $\Phi_B({\bf i})$ generate $\CC(B)$, we see that each $\theta_i$, 
$i \in [1, r]$, is a rational Casimir function on $B$.
\end{proof}

\begin{thm}\label{thm:mainB}
The cluster structure $\calC_{\rm BFZ}(B)$ has Property $\calI$ at the identity element $e \in B$
after the frozen variable modification $\overline{{\rm Froz}_B}$. Consequently, for every extended cluster 
$\Phi$ of $\calC_{{\rm BFZ}}(B)$, the set of lowest degree terms
of the functions in $\overline{\Phi}$ at $e \in B$ contains a polynomial integrable system on $(\b, \pi_0)$.
\end{thm}

\begin{proof} Take any reduced word $\bfw_0 = (s_{i_1}, \ldots, s_{i_{\elw}})$  
and parametrize 
$B$ again by $\CC^{\elw} \times (\CC^\times)^r$ as in \eqref{eq:xi-z}. Let 
$\Phi_B({\bf i})$ be given in \eqref{eq:Phi-B} and let $\overline{\Phi_B({\bf i})}$ be the modification of
$\Phi_B({\bf i})$ by $\overline{{\rm Froz}_B}$. Similar to the proof of Theorem \ref{thm:mainG}, we
have, up to multiplication by $-1$,
    \begin{align*}
{\mu}_{\overline{\Phi_B({\bf i})}} = \left( \prod \limits_{i \in I_2} \frac{p_iq_i}{q'_i} \right) \mu_{\Phi_B({\bf i})}
 &=\left( \prod \limits_{i \in I_2} \frac{p_iq_i}{p_iq_i-a_ip_{i^*}q_{i^*}} \right) \frac{d\xi_1 \wedge \cdots \wedge d\xi_r}{\xi_1 \cdots \xi_r} \wedge 
\frac{dz_1 \wedge \cdots \wedge dz_{\elw}}{\prod_{\alpha \in \Gamma} \vphi_{w_0, \alpha}(z)}.
\end{align*}
Taking lowest degree terms at $e \in B$ and again 
using arguments similar to that used in the proof of Theorem \ref{thm:mainG}, one checks that 
$\deg ((p_iq_i-a_ip_{i^*}q_{i^*})^\low)= 1 + \deg ((p_iq_i)^\low)$ for every $i \in I_2$
and that 
\[
\deg \left(\mu_{\Phi_B({\bf i})}^\low\right) = r+\ell(w_0)-d_{w_0} = r+
\frac{1}{2} (\elw-\dim \ker (1+w_0)).
\]
Recalling that $|I_2| = r-\dim \ker (1+w_0)$, one thus has
\[
\deg \left({\mu}_{\overline{\Phi_B({\bf i})}}^\low\right) = \deg \left({\mu}_{{\Phi_B({\bf i})}}^\low\right)-|I_2|=
\frac{1}{2}( \elw +\dim \ker (1+w_0))=\frac{1}{2}{\rm rk}(\b, \pi_0),
\]
where in the last step we used Lemma \ref{lem:ind-bm}. 
\end{proof}

\begin{rem}\label{rem:casi-B}
{\rm Note that for $i \in I_2$, one has 
\[
(\theta_i-a_i)^\low  = \frac{(q_i^\prime)^\low}{p_{i^*}^\low q_{i^*}^\low} \in \CC\left(\overline{{\rm Froz}_B}^\low\right),
\]
where $(q_i^\prime)^\low(x_0 + x_+) = 2(\omega_i-\omega_{i^*})(x_0)\delta_{\omega_i, w_0\omega_i}(x_+) + B_i(x_+)$ and $x_0 \in \t, x_+ \in \n$.
Similar to the result for $\calCBFZ(G)$ in $\S$\ref{ss:casi-G}, 
$\{(\theta_i-a_i)^\low: i \in I_2\}$ is a set of $|I_2|=\ind(\b_-)$ algebraically independent rational Casimir functions on 
$(\b, \pi_0)$ contained in 
$\CC\left(\overline{{\rm Froz}_B}^\low\right)$. 
}
\end{rem}

Consider again any reduced word
${\bf w}_0 = (i_1, \,i_2, \, \ldots, \, i_{\elw})$ 
of $w_0$, and recall that $d_{w_0} = \frac{1}{2}(\elw + \dim \ker(1+w_0))$. Set again
\[
\delta({\bf w}_0) = \{\delta_{\omega_{i_k}, \, s_{i_1}\cdots s_{i_k} \omega_{i_k}} \in \CC[\g]: k \in [1, \elw]\}.
\]
By Theorem  \ref{thm:choice-GB} applied to $w = w_0$, all maximal independent subsets of $\delta({\bf w}_0)|_\n$ 
have cardinality $d_{w_0}$.
Let  
\[ 
C_0 = \{(q_i^\prime)^\low = (\Delta_{\omega_i, \omega_i} \Delta_{\omega_i, w_0\omega_i}|_B -a_i
\Delta_{\omega_{i^*}, \omega_{i^*}} \Delta_{\omega_{i^*}, w_0\omega_{i^*}}|_B)^\low:  i \in I_2\},
\]
where the lowest degree term is taken at $e \in B$. The following result is proved similarly as for the case of $G$.

\begin{thm}\label{thm:choice-B}
For any reduced word $\bfw_0$ of $w_0$, the set 
$\delta(\bfw_0)|_{\b} \cup C_0$ 
of homogeneous polynomial functions on $\b$ is $\pi_0$-involutive, and all of its elements 
have quasi-polynomial Hamiltonian flows with respect to $\pi_0$ (see Definition \ref{defn:quasi-poly}). Moreover,  for any maximal independent subset 
$C$ of $\delta(\bfw_0)|_{\b}$, the set
$C  \sqcup C_0$
is a polynomial integrable system on $(\b, \pi_0)$. 
\end{thm}

\section{The dual Poisson Lie group of \texorpdfstring{$\GL(n, \CC)$}{GL(n,C)}} \label{sect:GLndual}
\subsection{The Poisson Lie group \texorpdfstring{$(\GL(n, \CC)^*, \pi_{\rm st}^*)$}{GLn*}}\label{ss:GLn-star}
Throughout $\S$\ref{sect:GLndual}, unless otherwise specified, we let 
\[
G = \GL(n, \CC) \hs \mbox{and} \hs \g = \mathfrak{gl}(n, \CC).
\]
Denote the $n \times n$ identity matrix by $\bbone_n$. Let $B$ (resp. $B_-$) be the subgroup of $G$ consisting of all 
upper (resp. lower) triangular matrices in $G$. 
Write $T \subset G$ for the subgroup of invertible diagonal $n \times n$ matrices.  The Lie algebra of $B$ (resp. $B_-, T$) will be denoted as $\b$ (resp. $\b_-, \t$).
Equip $G$ with the standard multiplicative 
Poisson structure $\pist$ as in Example \ref{ex:SLn-GLn}. The Drinfeld double Poisson Lie group of $(G, \pist)$ is then 
$(G \times G, \Pi)$, where denoting an arbitrary element in $G \times G$ by 
$(X, Y) =((x_{ij}), (y_{pq}))$, one has
\cite[$\S$5.2]{GSV:related}
\begin{align} \label{eqn:Poisbradouble}
&\{x_{ij}, \, x_{pq}\}_{\Pi} = \frac{1}{2} ({\rm sign}(p-i) + {\rm sign}(q-j)) x_{iq}x_{pj}, \nonumber \\
&\{y_{ij}, \, y_{pq}\}_{\Pi} = \frac{1}{2} ({\rm sign}(p-i) + {\rm sign}(q-j)) y_{iq}y_{pj}, \\
&\{y_{ij}, \, x_{pq}\}_{\Pi} = \frac{1}{2} \left( (1+{\rm sign}(q-j))y_{iq}x_{pj}-(1+ {\rm sign}(i-p)) x_{iq}y_{pj} \right). \nonumber
\end{align}
For $(X, Y) \in B \times B_-$, let $[X]_0$ and $[Y]_0$ be the diagonals of $X$ and $Y$ respectively, and set
\[
    G^{\ast} =  \{(X,Y) \in B \times B_- \colon [X]_0[Y]_0 = \bbone_n\}
    \subset  B \times B_-.
\]
Note that $G^*$ has Lie algebra $\g_{\rm st}^*$
defined in (\ref{eq:gsts}). 
It is well-known \cite{GSV:2018, LM:mixed} that both $G^*$ and $B \times B_-$ are Poisson submanifolds of $(G \times G,\Pi)$.
Moreover, setting
\[
\pi_{\rm st}^*=-\Pi|_{G^*},
\]
then $(G^*, \pi_{\rm st}^*=-\Pi|_{G^*})$ is a dual Poisson Lie group of $(G, \pist)$ under the vector space isomorphism
\[
\g_{\rm st}^* \cong \g^*, ~ (x_1,x_2) \longmapsto \bigl( y \mapsto \tr((x_1 - x_2)y) \bigr).
\]
The linearization of $(G^*, \pi_{\rm st}^*)$ at $(\bbone_n, \bbone_n) \in G^*$ is thus $(\g^*, \pi_{0, \g})$, where
$\pi_{0, \g}$ is the Kirillov-Kostant-Souriau Poisson structure defined by $\g$.
Furthermore, $(B \times B_-, \Pi|_{B \times B_-})$ is a Poisson Lie subgroup of $(G \times G, \Pi)$.  The linearization of $(B \times B_-, \Pi|_{B \times B_-})$ at $(\bbone_n, \bbone_n) \in B \times B_-$ is $((\g \oplus \t)^*, \pi_{0, \g \oplus \t})$, where $\pi_{0, \g \oplus \t}$ is the Kirillov-Kostant-Souriau Poisson structure defined by the direct sum Lie algebra $\g \oplus \t$.

\begin{lem} \label{lem:rkBB-}
    The rank of $\pi_{0, \g \oplus \t}$ is $n^2-n$.
\end{lem}

\begin{proof}
    Since $\t$ is an abelian Lie algebra and the Lie bracket of the two direct summands $\g$ and $\t$ is zero, the rank of $\pi_{0, \g \oplus \t}$ is equal to that of $\pi_{0, \g}$, which, as is well-known, is equal to $n^2-n$.
\end{proof}

\subsection{The generalized upper cluster structure \texorpdfstring{$\calC_{\rm GSV}$}{Cgsv} on \texorpdfstring{$B \times B_-$}{B times B-}} \label{sect:log-cansys}

In \cite{GSV:related}, M. Gekhtman, M. Shapiro and A. Vainshtein constructed a generalized cluster structure in $\CC(B \times B_-)$ which is regular on $B \times B_-$ and compatible with $\Pi|_{B \times B_-}$.  We now recall their construction.


Throughout this section, for a matrix $A$ of size $p \times q$ and integers $i,j,k,l$ 
with $1 \le i \le j \le p$ and $1 \le k \le l \le q$, we denote by $A_{[i,j]}^{[k,l]}$ the 
submatrix of $A$ with row index set $[i,j]$ and column index set $[k,l]$.  We also write $A_{[i,j]} = A_{[i,j]}^{[1, q]}$ and $A^{[k,l]} = A_{[1,p]}^{[k,l]}$.


Let $(X,Y) = ((x_{ij}),(y_{pq})) \in B\times B_-$.
Define the staircase matrix
\begin{equation}\label{eq:Psi}
        \Psi = 
    \begin{bmatrix}
        Y_{[2,n]} & & & & &\\
        X_{[2,n]} & Y_{[2,n]} & & & &\\
        & X_{[2,n]} & Y_{[2,n]} & & &\\
        & & \ddots & \ddots & &\\
        & & & X_{[2,n]} & Y_{[2,n]} &\\
        & & & & X_{[2,n]} & Y^{[1,1]}_{[2,n]}
    \end{bmatrix}
\end{equation}
consisting of $n-1$ block rows and $n-1$ block columns, so $\Psi$ is of size $(n-1)^2 \times (n-1)^2$. 
Set
 \begin{equation}\label{eq:varphi-i}
 \varphi_i = \det \Psi_{[i,\,{(n-1)^2}]}^{[i,\,{(n-1)^2}]}, \hs i \in [1,\, {(n-1)^2}].
 \end{equation}
Furthermore, define the polynomial function $c_i(X, Y)$, for $i \in [0, n]$, by
\[
\det(\lambda Y+ X)=\sum \limits_{i=0}^n (-1)^{i(n-1)}c_i(X,Y)\lambda^{n-i}
\]
and introduce the subset $\Phi$ of $\mathbb{C}[B\times B_-]$ given by
\begin{equation}\label{eq:initialG*}
    \Phi=\{ \varphi_i: i\in [1,\, {(n-1)^2}] \} \sqcup \{x_{ii}, y_{ii}: i\in [1,n]\} \sqcup \{c_i: i\in [1,n-1]\}.
\end{equation}  
The following statement is a special case of \cite[Theorem 5.1]{GSV:related}.

\begin{thm} \label{thm:GSV}
    The set $\Phi$ is the extended cluster of a seed in $\mathbb{C}(B\times B_-)$ of a generalized upper cluster structure on $B \times B_-$, denoted as $\mathcal{C}_{\rm GSV}$, which is regular on $B \times B_-$ and compatible with the Poisson structure 
    $\Pi|_{B \times B_-}$ on $B \times B_-$. The frozen variables of $\mathcal{C}_{\rm GSV}$ are ${\rm Froz}(\mathcal C_{\rm GSV}) = \{x_{ii}, y_{ii}: i\in [1,n]\}\sqcup \{c_i: i\in [1,n-1]\}$. 
\end{thm}

We now define a frozen variable modification of $\mathcal C_{\rm GSV}$.  For $i\in [0,n]$, define
$\overline{c}_i(X,Y)$ by
\begin{align} \label{eq:modifiedCasimirs}
    \det((\lambda-1) {\mathbbm 1}_n + XY^{-1}) = \sum \limits_{i=0}^n \overline{c}_i(X,Y) \lambda^{i}.
\end{align}
Denote by $\overline{y}_{11}=y_{11}\cdots y_{nn}-1$, and for $i \in [2,n]$, let $\overline{x}_{ii} = x_{ii}$ and $\overline{y}_{ii} = x_{ii}y_{ii} - 1$.

\begin{lem}
    The set
    \[
    \overline{{\rm Froz}(\mathcal C_{\rm GSV})} = \{\overline{x}_{ii} \colon i \in [2,n]\} \sqcup \{\overline{y}_{ii} \colon i \in [1,n]\} \sqcup \{\overline{c}_i \colon i \in [0,n-1]\}
    \]
    is a regular $\Pi|_{B \times B_-}$-compatible frozen variable modification of $\mathcal C_{\rm GSV}$.
\end{lem}

\begin{proof}
    By definition, for $i \in [0,n]$,
    \begin{align} \label{eqn:ccbar}
        \overline{c}_i(X,Y) = \frac{1}{y_{11} \cdots y_{nn}} \sum \limits_{j=0}^{n-i} (-1)^{j(n-1)+n-i-j} \binom{n-j}{i} c_j(X,Y).
    \end{align}
    Note that, by definition, $c_0(X,Y) = y_{11} \cdots y_{nn}$ and $c_n(X,Y) = x_{11} \cdots x_{nn}$.  Hence, for $i \in [1,n-1]$,
    \begin{align*}
        \overline{c}_i(X,Y) = (-1)^{n-i} \binom{n}{i} + \frac{1}{y_{11} \cdots y_{nn}} \sum \limits_{j=1}^{n-i} (-1)^{j(n-1)+n-i-j} \binom{n-j}{i} c_j(X,Y).
    \end{align*}
    Moreover,
    \[
    \overline{c}_0 (X,Y) = (-1)^n \binom{n}{0} + \frac{1}{y_{11} \cdots y_{nn}} \sum \limits_{j=1}^{n-1} (-1)^{j(n-1)+n-j} \binom{n-j}{0} c_j(X,Y) + \frac{1}{y_{11} \cdots y_{nn}} x_{11} \cdots x_{nn}.
    \]
    This proves that all elements of $\overline{{\rm Froz}(\mathcal C_{\rm GSV})}$ are in the subfield $\CC({\rm Froz}(\mathcal C_{\rm GSV}))$. Moreover, $\overline{{\rm Froz}(\mathcal C_{\rm GSV})}$ is a free generating set of $\CC({\rm Froz}(\mathcal C_{\rm GSV}))$ and, hence, is a frozen variable modification of $\mathcal C_{\rm GSV}$.

    That $\overline{{\rm Froz}(\mathcal C_{\rm GSV})}$ is regular on $B \times B_-$ is clear from the definitions.  It is well-known that $c_i$ is a Casimir function on $(B \times B_-, \Pi_{B \times B_-})$ for $i \in [0,n]$. Thus $\overline{y}_{11}=c_0-1$ is a Casimir. By (\ref{eqn:ccbar}), for $i \in [0,n-1]$, $\overline{c}_i$ is a $\CC$-linear combination of $c_0/c_0, \ldots, c_n/c_0$, hence is itself a Casimir.  By (\ref{eqn:Poisbradouble}), $x_{ii}y_{ii} - 1$ is also a Casimir function for $i \in [1,n]$.  It now follows 
    that $\overline{{\rm Froz}(\mathcal C_{\rm GSV})}$ is compatible with $\Pi|_{B \times B_-}$.
\end{proof}

\subsection{Property \texorpdfstring{$\calI$}{I}  of \texorpdfstring{$\calC_{\rm GSV}$}{Cgsv} at \texorpdfstring{$(\bbone_n, \bbone_n) \in B \times B_-$}{(1,1) in B*B-}}\label{ss:GSV-I}
We first compute the modified log-volume form $\overline{\mu}_{\mathcal C_{\rm GSV}}$ of $\mathcal C_{\rm GSV}$ by $\overline{{\rm Froz}(\mathcal C_{\rm GSV})}$.

\begin{thm} \label{thm:dcbar}
    As meromorphic differential forms on $B \times B_-$, we have
    \[
    \frac{d \overline{c}_0}{\overline{c}_0} \wedge \cdots \wedge \frac{d \overline{c}_{n-1}}{\overline{c}_{n-1}} = \frac{1}{(y_{11} \cdots y_{nn})^n \overline{c}_0 \cdots \overline{c}_{n-1}} \varphi_1 dx_{11} \wedge dx_{12} \wedge \cdots \wedge dx_{1n} + \cdots,
    \]
    where $\cdots$ stands for forms that are not a $\CC(B \times B_-)$-multiple of $d x_{11} \wedge d x_{12} \wedge \cdots \wedge d x_{1n}$.
\end{thm}

\begin{proof}
    By (\ref{eqn:ccbar}), we have
    \[
    d \overline{c}_0 \wedge \cdots \wedge d \overline{c}_{n-1} = \left( \frac{1}{y_{11} \cdots y_{nn}} \right)^n (-1)^{(n-1)n(n+1)/2} d c_1 \wedge \cdots \wedge d c_n + \cdots.
    \]
    Set $U = XY^{-1}$, so $c_i(X,Y) = y_{11} \cdots y_{nn} c_i(U, \bbone_n)$ for $i \in [0,n]$.  Define $T_i(U) = \tr(U^i)$ for $i \in [1,n]$.  By Newton's identity, for $i \in [1,n]$, we have
    \[
    c_i(U, \bbone_n) = (-1)^{in-1} \frac{1}{i} T_i(U) + \text{ polynomial of } \{T_{i-1}(U), \ldots, T_1(U)\}.
    \]
    Hence,
    \begin{align*}
        d c_1(X,Y) \wedge \cdots \wedge d c_n(X,Y) = & (y_{11} \cdots y_{nn})^n d c_1(U, \bbone_n) \wedge \cdots \wedge d c_n(U, \bbone_n) + \cdots \\
        = & (-1)^{(n-1)n(n+2)/2} \frac{1}{n!} (y_{11} \cdots y_{nn})^n d T_1(U) \wedge \cdots \wedge d T_n(U) + \cdots.
    \end{align*}
    Observe that, for $i,j \in [1,n]$,
    \begin{align*}
        \frac{\partial T_i(U)}{\partial x_{1j}} = \frac{\partial \tr (U^i)}{\partial x_{1j}} = i \tr (\frac{\partial U}{\partial x_{1j}} U^{i-1}) = i \tr (e_{1j} Y^{-1} U^{i-1}) = i (Y^{-1} U^{i-1})_{[j,j]}^{[1,1]},
    \end{align*}
    where $e_{1j}$ stands for the $n \times n$ matrix whose only non-zero entry is in position $(1,j)$, and that non-zero entry is $1$.  It follows that, writing $e_1$ for $[1 ~ 0 ~ \cdots ~ 0]^T$,
    \begin{align*}
        d T_1(U) \wedge \cdots \wedge d T_n(U) & =  \det \left[ (i (Y^{-1} U^{i-1})_{[j,j]}^{[1,1]})_{j \in [1,n]}^{i \in [1,n]} \right] d x_{11} \wedge d x_{12} \wedge \cdots \wedge d x_{1n} + \cdots \\
        & = n! \det \left[ Y^{-1} U^0 e_1 ~ Y^{-1} U^1 e_1 ~ \cdots ~ Y^{-1} U^{n-1} e_1 \right] d x_{11} \wedge d x_{12} \wedge \cdots \wedge d x_{1n} + \cdots \\
        & = n! (\det Y)^{-1} \det \left[ U^0 e_1 ~ U^1 e_1 ~ \cdots ~ U^{n-1} e_1 \right] d x_{11} \wedge d x_{12} \wedge \cdots \wedge d x_{1n} + \cdots \\
        & = (-1)^{n(n-1)/2} n! (\det Y)^{-n} \varphi_1 d x_{11} \wedge d x_{12} \wedge \cdots \wedge d x_{1n} + \cdots,
    \end{align*}
    where the last step follows from \cite[Lemma 3.3]{GSV20}.  Combining everything, we are done.
\end{proof}

The proof of the following Theorem \ref{thm:dphiBB-}, being quite involved, is postponed to $\S$\ref{sect:logvolformexp}.

\begin{thm} \label{thm:dphiBB-}
    As meromorphic differential forms on $B \times B_-$, we have
    \begin{align*}
        \frac{d \varphi_1}{\varphi_1} \wedge \cdots \wedge \frac{d \varphi_{(n-1)^2}}{\varphi_{(n-1)^2}} \wedge & \frac{d x_{22}}{ x_{22} }\wedge \cdots \wedge \frac{d x_{nn}}{ x_{nn} } \wedge \frac{d y_{11}}{ y_{11} } \wedge \frac{d y_{22}}{ y_{22} } \wedge \cdots \wedge \frac{d y_{nn}}{ y_{nn} } \\
    = & \pm \frac{1}{x_{22} \cdots x_{nn} y_{11} \cdots y_{nn} \varphi_1}  \bigwedge \limits_{2 \le i \le j \le n} dx_{ij} \wedge \bigwedge \limits_{1 \le j \leq i \le n} dy_{ij}.
    \end{align*}
\end{thm}

We now have the following easy consequence of Theorem \ref{thm:dcbar} and Theorem \ref{thm:dphiBB-}.

\begin{thm} \label{thm:modifiedlogvol}
    We have
    \[
    \overline{\mu}_{\mathcal C_{\rm GSV}} = \pm \frac{x_{11}}{ (y_{11} \cdots y_{nn})^{n} \overline{c}_0 \cdots \overline{c}_{n-1} (y_{11}\cdots y_{nn} - 1) (x_{22}y_{22} - 1) \cdots (x_{nn}y_{nn} - 1)} \bigwedge \limits_{1 \leq i \leq j \leq n} d x_{ij} \wedge \bigwedge \limits_{1 \leq j \leq i \leq n} d y_{ij}.
    \]
\end{thm}

\begin{proof}
    For $i \in [2,n]$, we have $d \overline{y}_{ii} = x_{ii} d y_{ii} + y_{ii} d x_{ii}$, so $d \overline{x}_{ii} \wedge d \overline{y}_{ii} = x_{ii} d x_{ii} \wedge d y_{ii}$, which implies that
    \[
    \frac{d \overline{x}_{ii}}{\overline{x}_{ii}} \wedge \frac{d \overline{y}_{ii}}{\overline{y}_{ii}} = \frac{1}{x_{ii}y_{ii} - 1} d x_{ii} \wedge d y_{ii}.
    \]
    Now the conclusion follows from Theorem \ref{thm:dcbar} and Theorem \ref{thm:dphiBB-}.
\end{proof}

Let $(x,y) \in \b \oplus \b_-$. The entries of the matrices $x$ and $y$ become a local holomorphic coordinate system on $B \times B_-$ near $(\bbone_n, \bbone_n)$ via the map 
\[
\b \oplus \b_- \longrightarrow B \times B_-, \quad (x,y) \longmapsto (\exp(x), \exp(y)).
\]

\begin{prop} \label{lem:degcas}
    Let $(x,y) \in \b \oplus \b_-$.  For $i \in [0,n-1]$, taking lowest degree terms at $(\bbone_n, \bbone_n)$, we have
    \begin{align*}
        \overline{c}_i (\exp(x), \exp(y))^{\low} = ~ \text{sum of the principal} ~ (n-i) \times (n-i) ~ \text{minors of} ~ x-y.
    \end{align*}
    In particular, we have $\deg (\overline{c}_i^{\low}) = n-i$.
\end{prop}

\begin{proof} 
    By (\ref{eq:modifiedCasimirs}), we have 
    \begin{align*}
        \sum \limits_{i=0}^n \overline{c}_i(\exp (x), \exp (y)) \lambda^i = \det \biggl( \lambda {\mathbbm 1}_n + \bigl( \exp(x)\exp(-y) - {\mathbbm 1}_n \bigr) \biggr) = \det \biggl( \lambda {\mathbbm 1}_n + \bigl( x-y + \cdots \bigr) \biggr),
    \end{align*}
    where $\cdots$ stands for expressions of $x, y$ of degree at least $2$.  The formula for $\overline{c}_i^\low$ follows.
\end{proof}

\begin{thm} \label{thm:mainthmBB-}
    The generalized upper cluster structure $\mathcal C_{\rm GSV}$ on $B \times B_-$
    has Property $\calI$ at $(\bbone_n, \bbone_n) \in B \times B_-$ after the frozen variable modification $\overline{{\rm Froz}(\mathcal C_{\rm GSV})}$.
\end{thm}

\begin{proof}
    Let $(x,y) \in \b \oplus \b_-$ and $(X,Y) = (\exp(x), \exp(y))$.  For $i \in [2,n]$, taking lowest degree terms at $(\bbone_n, \bbone_n)$, we evidently have (note that $x_{ii}$ (resp. $y_{ii}$) is the $(i,i)$-entry of $X$ (resp. $Y$), not $x$ (resp. $y$))
    \[
    x_{ii}y_{ii} - 1 = e^{(i,i) \text{-entry of } x} e^{(i,i) \text{-entry of } y} - 1 = \bigl( (i,i) \text{-entry of } x \bigr) + \bigl( (i,i) \text{-entry of } y \bigr) + \text{ higher degree terms}.
    \]
    Hence, $\deg \bigl( (x_{ii}y_{ii} - 1)^\low \bigr) = 1$. Similarly, $\deg(\overline{y}_{11}^\low)=1$. By Theorem \ref{thm:modifiedlogvol} and Proposition \ref{lem:degcas}, we have
    \[
    \deg (\overline{\mu}_{\mathcal C_{\rm GSV}}^\low) = - (n + (n-1) + \cdots + 1 + n) + \frac{1}{2} n(n+1) + \frac{1}{2} n(n+1) = \frac{1}{2} n(n-1),
    \]
    which is equal to $\frac{1}{2} {\rm rk}(\pi_{0, \g \oplus \t})$ by Lemma \ref{lem:rkBB-}.  We are done by Theorem \ref{thm:main-cluster}.
\end{proof}

\subsection{Restriction to \texorpdfstring{$G^*$}{G*}}\label{ss:cluster-G-star}
We learned from M. Gekhtman (private communication) that the generalized upper cluster structure $\mathcal C_{\rm GSV}$ on $B \times B_-$ restricts to a generalized upper cluster structure $\calC$ on $G^*$ compatible with the Poisson structure $\pi_{\rm st}^*$. Since a reference for $\calC$ is not available in the literature at this point,  we will formulate our results on $(G^*, \pi_{\rm st}^*)$ in terms of the restrictions of the extended clusters of $\mathcal C_{\rm GSV}$ to $G^*$.

For any extended cluster $\Phi$ of $\mathcal{C}_{\rm GSV}$ in $\CC(B \times B_-)$, let
\begin{align} \label{eq:phi'}
    \overline{\Phi}|_{G^*} = \{\varphi|_{G^{\ast}} \colon \varphi \in \Phi \setminus {\rm Froz}(\mathcal C_{\rm GSV})\} \sqcup \{\overline{x}_{ii}|_{G^{\ast}} = x_{ii}|_{G^{\ast}} \colon i\in[2,n] \} \sqcup \{\overline{c}_i|_{G^{\ast}} \colon i\in[0, n-1]\}.
\end{align}
The main theorem of this section is

\begin{thm} \label{thm:mainthmGLn}
    For any extended cluster $\Phi$ of the generalized upper cluster structure $\mathcal{C}_{\rm GSV}$ on $B \times B_-$, the set $\overline{\Phi}|_{G^*}$ of functions on $G^{\ast}$ defined in \eqref{eq:phi'} is a log-canonical system on $(G^*, \pi_{\rm st}^*)$ with  Property $\calI$ at $(\bbone_n, \bbone_n)$.  In particular, the set $(\overline{\Phi}|_{G^*})^{\low}$ contains a polynomial integrable system on $(\g^{\ast}, \pi_{0, \g})$.
\end{thm}

\begin{proof}
    This is an easy consequence of Theorem \ref{thm:mainthmBB-}.  In fact, on $B \times B_-$, since $\deg \bigl( (d \overline{y}_{ii}/\overline{y}_{ii})^\low \bigr) = 0$ for all $i \in [1,n]$, where the lowest degree term is taken at $(\bbone_n, \bbone_n) \in B \times B_-$, inspecting the proof of Theorem \ref{thm:mainthmBB-}, we see that $\deg (\mu_{\overline{\Phi}|_{G^*}}^\low) = n(n-1)/2 - 0 = n(n-1)/2$, where the lowest degree term is taken at $(\bbone_n, \bbone_n) \in G^*$.  Recall the well-known fact that ${\rm rk}(\pi_{0, \g}) = n^2 - n$.  We are done by Theorem \ref{thm:main-0}.
\end{proof}

\subsection{A polynomial integrable system on \texorpdfstring{$(\g^*, \pi_{0, \g})$}{g*}}

For the log-canonical system
\[
\overline{\Phi}|_{G^*} = \{(\varphi_i)|_{G^{\ast}} \colon i \in [1, (n-1)^2]\} \sqcup \{\overline{x}_{ii}|_{G^{\ast}} = x_{ii}|_{G^{\ast}} \colon i\in[2,n] \} \sqcup \{\overline{c}_i|_{G^{\ast}} \colon i\in[0, n-1]\}
\]
on $(G^*, \pist^*)$ which has Property $\calI$ at $(\bbone_n, \bbone_n) \in G^*$ (see Theorem \ref{thm:mainthmGLn}), we present an explicit choice of a subset of $(\overline{\Phi}|_{G^*})^{\low}$ which forms an integrable system on $(\g_{\rm st}^* \cong \g^*,\pi_{0, \g})$.  For $(x,y)\in \g_{\rm st}^*$, set $u=x-y$, so $(\exp(x),\exp(y)) \in G^{\ast}$ and hence the entries of $u$ form a local holomorphic coordinate system on $G^{\ast}$ near $(\bbone_n,\bbone_n)$. 
We shall write down such an integrable system on $(\g_{\rm st}^*,\pi_{0, \g})$ using this local holomorphic coordinate system. For an $n \times n$ matrix $u$ and $I, J \subseteq [1,n]$ with $|I| = |J|$, recall that $\Delta_{I,J}(u)$ is the determinant of the submatrix of $u$ consisting of rows indexed by $I$ and columns by $J$.  Since every $k \in [1, (n-1)^2]$ can be uniquely written as $k = p(n-1)+i$, where $i \in [1,n-1]$ and $p \in [0,n-2]$, we will often write $\varphi_{p(n-1)+i}$ in place of $\varphi_k$.

In what follows, for ease of notation, we simply write $f$ in place of $f|_{G^*}$ for 
$f\in \CC[B \times B_-]$.

\begin{thm}\label{thm:choice-GLn}
Take $(x,y)\in \g_{\rm st}^*$, and $u=x-y$.

1) The following subset  is an integrable system on $(\g_{\rm st}^{\ast}, \pi_{0, \g})$:
\[
    \{ \varphi_{p(n-1)+i}^\low : i\leq p+1\} \cup\{\overline{c}_i^\low: 0 \le i \le n-1 \}.
\]
            
2) Put $F(u)=\begin{bmatrix}
        ue_1 & \cdots & u^n e_1
    \end{bmatrix}$ with $e_1=[1 ~ 0 ~ \cdots ~ 0]^T\in \mathbb{C}^n$. Up to a sign, we have
\begin{align*}
    \varphi_{p(n-1)+i}^\low(u) &=\Delta_{[i+1,n-p+i-1],[1,n-p-1]}(F(u)),&  &\forall\ i \le p+1;\\
    \overline{c}_i^\low (u) &= \text{sum of principal~} (n-i)\times (n-i) \text{~minors of~} u,&  &\forall\ 0 \le i \le n-1.
\end{align*}
\end{thm}

Concretely, the functions $\{ \varphi_{p(n-1)+i}^\low : i\leq p+1\}$ that appear in part 1) of Theorem \ref{thm:choice-GLn} are the following minors of $F(u)$ (the matrices being $F(u)$):
\begin{equation*}
        \tikz[baseline=(M.west)]{%
        \node[matrix of math nodes,matrix anchor=west,left delimiter={[},right delimiter={]},ampersand replacement=\&] (M) {%
          * \& * \& * \& \cdots \& * \& * \\
          * \& * \& * \& \cdots \& * \& * \\
          \vdots \& \vdots \& \vdots \& \iddots \& \vdots \& \vdots \\
          * \& * \& * \& \cdots \& * \& * \\
          * \& * \& * \& \cdots \& * \& * \\
          * \& * \& * \& \cdots \& * \& * \\
        };
        \node[draw,fit=(M-6-1),inner sep=0pt] {};
        \node[draw,fit=(M-5-1)(M-6-2),inner sep=0pt] {};
        \node[draw,fit=(M-4-1)(M-6-3),inner sep=0pt] {};
        \node[draw,fit=(M-2-1)(M-6-5),inner sep=0pt] {};
        } \qquad
        \tikz[baseline=(M.west)]{%
        \node[matrix of math nodes,matrix anchor=west,left delimiter={[},right delimiter={]},ampersand replacement=\&] (M) {%
          * \& * \& * \& \cdots \& * \& * \\
          * \& * \& * \& \cdots \& * \& * \\
          \vdots \& \vdots \& \vdots \& \iddots \& \vdots \& \vdots \\
          * \& * \& * \& \cdots \& * \& * \\
          * \& * \& * \& \cdots \& * \& * \\
          * \& * \& * \& \cdots \& * \& * \\
        };
        \node[draw,fit=(M-5-1),inner sep=0pt] {};
        \node[draw,fit=(M-4-1)(M-5-2),inner sep=0pt] {};
        \node[draw,fit=(M-2-1)(M-5-4),inner sep=0pt] {};
        } \qquad \cdots \qquad
        \tikz[baseline=(M.west)]{%
        \node[matrix of math nodes,matrix anchor=west,left delimiter={[},right delimiter={]},ampersand replacement=\&] (M) {%
          * \& * \& * \& \cdots \& * \& * \\
          * \& * \& * \& \cdots \& * \& * \\
          \vdots \& \vdots \& \vdots \& \iddots \& \vdots \& \vdots \\
          * \& * \& * \& \cdots \& * \& * \\
          * \& * \& * \& \cdots \& * \& * \\
          * \& * \& * \& \cdots \& * \& * \\
        };
        \node[draw,fit=(M-2-1),inner sep=0pt] {};
        }.
\end{equation*}

The proof of Theorem \ref{thm:choice-GLn} is postponed to $\S$\ref{sect:lowdegtermexp}.  The proof of the first statement of Theorem \ref{thm:choice-GLn} involves calculating the rank of the Jacobi matrix of the displayed functions with respect to the entries of $u$.  This makes use of the formulas presented in the second statement.  The second statement follows directly from the next Theorem \ref{prop:lambdalow}, whose proof is also postponed to $\S$\ref{sect:lowdegtermexp}.
\begin{thm} \label{prop:lambdalow}
    Let $\{e_1, \ldots, e_n\}$ be the standard basis of $\CC^n$.  For $i \in [1,n-1]$, $p \in [0,n-2]$, up to a sign,
    \[
         \varphi_{p(n-1)+i}^{\low}(u) =  \begin{cases}
         \vspace{5pt} \det
    \begin{bmatrix}
        e_{n-p+i} & \cdots & e_n & e_i & \cdots & e_2 & u^{n-p-1} e_1 & \cdots & ue_1 & e_1
    \end{bmatrix}, &\text{if~} i \leq p+1;\\ \vspace{5pt}
    \det \begin{bmatrix}
            e_{i-p} & \cdots & e_i & u^{n-p-2} e_1 & u^{n-p-3} e_1 & \cdots & u e_1 & e_1
        \end{bmatrix}, &\text{if~} i > p+1.
    \end{cases}
    \]
\end{thm}

\begin{example} \label{ex:runningexGLn}
   Take $n=3$.  Let $(x,y)\in \mathfrak{g}_{\rm st}^*$ with entries
   \[
    x=\begin{bmatrix}
        \frac{1}{2} u_{11} & u_{12} & u_{13}\\
        0 & \frac{1}{2} u_{22} & u_{23}\\
        0 & 0 & \frac{1}{2} u_{33}
    \end{bmatrix};\quad
    y=\begin{bmatrix}
        - \frac{1}{2} u_{11} & 0 & 0\\
        - u_{21} & - \frac{1}{2} u_{22} & 0\\
        - u_{31} & - u_{32} & - \frac{1}{2} u_{33}
    \end{bmatrix},
   \]
   such that $(X = \exp(x), Y = \exp(y)) \in G^*$, and we can use the entries $u_{ij}$ of $u = x-y$ as a local holomorphic coordinate system on $G^*$ near $(\bbone_3, \bbone_3)$. In terms of the entries of $(X = (x_{ij}), Y = (y_{ij}))$, we have:
   \begin{align*}
       \varphi_1&=-x_{33} y_{32} y_{21}^2+x_{33} y_{22} y_{31} y_{21}+x_{23} y_{31} y_{32} y_{21}-x_{22} y_{31} y_{33} y_{21}-x_{23} y_{22} y_{31}^2,\\
       \varphi_2&=-x_{33} y_{21} y_{32}+x_{23} y_{31} y_{32}-x_{22} y_{31} y_{33},\quad \varphi_3=x_{23} y_{31}-x_{33} y_{21},\quad \varphi_4=y_{31}.
   \end{align*}
   Explicit calculation shows that, for the cluster variables, we have
   \begin{align*}
       \varphi_1^\low&=-u_{32} u_{21}^2- u_{22} u_{31} u_{21}+ u_{31} u_{33} u_{21}-u_{23} u_{31}^2\\
       &=\Delta_{[2,3],[1,2]} (u) \Delta_{\{2\},\{1\}} (u) + \Delta_{[2,3],\{1,3\}} (u) \Delta_{\{3\},\{1\}} (u),\\
       \varphi_2^\low&=u_{31}=\Delta_{\{3\},\{1\}} (u), \quad \varphi_3^\low=u_{21}=\Delta_{\{2\},\{1\}} (u),\quad \varphi_4^\low=-u_{31}=-\Delta_{\{3\},\{1\}} (u),
   \end{align*}
   and for the frozen variables, we have $\overline{x}_{ii}^\low=1$ and 
   \begin{align*}
   \overline{c}_0^\low&=u_{11} u_{22} u_{33}+u_{12} u_{23} u_{31}+ u_{21} u_{13} u_{32}-u_{13} u_{22} u_{31}-u_{11} u_{23} u_{32}-u_{12} u_{21} u_{33}=\det u,\\
    \overline{c}_1^\low&=-u_{12} u_{21}+u_{11} u_{22}-u_{13} u_{31}-u_{23} u_{32}+u_{11} u_{33}+u_{22} u_{33}\\
    &=\Delta_{\{1,2\},\{1,2\}}(u)+\Delta_{\{1,3\},\{1,3\}}(u)+\Delta_{\{2,3\},\{2,3\}}(u),\\
    \overline{c}_2^\low&=u_{11}+u_{22}+u_{33}=\tr(u).
    \end{align*}
   The explicit choice of integrable system given by part 1) of Theorem \ref{thm:choice-GLn} is $\{\varphi_1^\low,\varphi_3^\low,\varphi_4^\low,\overline{c}_0^\low,\overline{c}_1^\low,\overline{c}_2^\low\}$.

   Note that we presented $(\overline{\Phi}|_{G^*})^\low$ in terms of both the entries of $u$ and its minors. General formulas of $\varphi_i^\low$ and $\overline{c}_i^\low$ in terms of the minors of $u$ are given in Corollary \ref{cor:lowinminor1}, Corollary \ref{cor:lowinminor2} and Proposition \ref{lem:degcas}.

   Take $n = 4$.  Let $x, y$ and $u$ be analogously defined.  Computer verification shows that the polynomial $\varphi_1^{\low}$, in terms of the $u_{ij}$'s, has as many as $86$ terms.   See Example \ref{ex:gl4} for a much cleaner formula for $\vphi_1^\low$ in terms of the minors of $u$.
\end{example}

\begin{rem} \label{rem:different}
Example \ref{ex:runningexGLn} shows that our integrable system is quite different from the one of Gelfand-Zeitlin. In fact, in the case where $n=3$, the Gelfand-Zeitlin integrable system, depending on convention, contains either $u_{11}+u_{22}$ or $u_{22}+u_{33}$, but the Poisson bracket of $\varphi_4^{\low} = -u_{31}$ with both functions are non-zero.  Hence our integrable system and the Gelfand-Zeitlin integrable system are genuinely different. 
\end{rem}

\begin{rem}
    The map $F: (\g_{\rm st}^{\ast}, \pi_{0, \g}) \rightarrow (\g_{\rm st}^{\ast}, \pi_{0, \g})$ that appeared in Theorem \ref{thm:choice-GLn} is not a Poisson map.  However, one can prove that it is a birational isomorphism with inverse map given by $f \mapsto f [e_1 ~ f^{[1,n-1]}]^{-1}$.  Algebro-geometric and Poisson geometric properties of $F$ will be investigated in a subsequent publication.
\end{rem}

\subsection{Proof of Theorem \ref{thm:dphiBB-}} \label{sect:logvolformexp}
The proof consists of several steps, beginning with an outline of the overall strategy. Put
\begin{align*}
    d(X,Y)
    = & (dx_{22} \wedge dx_{23} \wedge \cdots \wedge dx_{2n}) \wedge (dx_{33} \wedge dx_{34} \wedge \cdots \wedge dx_{3n}) \wedge \cdots  \wedge (dx_{n-1,n-1} \wedge dx_{n-1,n}) \wedge dx_{nn} \\
    &\wedge dy_{11}\wedge (dy_{21}\wedge y_{22}) \wedge \cdots \wedge (dy_{n1} \wedge dy_{n2} \wedge \cdots \wedge dy_{nn}).
\end{align*}
Our proof is based on
\begin{thm}\cite[Proposition 3.5]{GSV:related}\label{thm:gsv3.5}
    There exists an $(n-1)\times (n-1)$ unipotent upper triangular matrix ${\bf G}={\bf G}(X,Y)$ such that entries of ${\bf G}$ are rational functions in $X,Y$ and the $(2n-2)\times n$ matrix
    \begin{equation}
    \label{eq:bigM}
            M=\begin{bmatrix}
                Y_{[2,n]}\\
                {\bf G}X_{[2,n]}
            \end{bmatrix}
        \end{equation}
    satisfies
        \begin{equation*}
        \det M_{[n-i+k,n+k-1]}^{[n-i+1,n]}=\frac{\varphi_{kn-i+1}}{\varphi_{kn+1}}, \hs \text{~for~} k\in [1,n-2] \text{~and~} i\in [1,n].
        \end{equation*}
\end{thm}
In \cite[Proof of Proposition 3.5]{GSV:related}, the matrix ${\bf G}$ is constructed as a product ${\bf G}={\bf G}_{n-2}\cdots {\bf G}_1$, where ${\bf G}_i={\bf G}_i(X,Y)$ is an $(n-1)\times (n-1)$ unipotent upper triangular matrix  such that the entries of ${\bf G}_i$ are rational functions in $X,Y$ and the $(k,l)$-entries are all zero for $k<l$ and $k\neq i$. Denote by ${\bf G}_i'$ the block diagonal matrix $\diag(1, {\bf G}_i)$. We introduce a sequence of coordinates
\[
(X_1={\bf G}_1'X,Y),\ (X_2={\bf G}_2'{\bf G}_1'X,Y),\  \ldots,\ (X_{n-2}={\bf G}_{n-2}'\cdots {\bf G}_{1}'X,Y)
\]
on $B \times B_-$. The proof of Theorem \ref{thm:dphiBB-} is a combination of Theorem \ref{thm:volform}, where we show that $d(X,Y)=d(X_1,Y)=\cdots =d(X_{n-2},Y)$, and Theorem \ref{thm:dphi}, where we show that
\[
\frac{1}{\varphi_1} d(X_{n-2},Y)=\pm\frac{d \varphi_1}{\varphi_1} \wedge \cdots \wedge \frac{d \varphi_{(n-1)^2}}{\varphi_{(n-1)^2}} \wedge  d x_{22} \wedge \cdots \wedge d x_{nn}\wedge  d y_{11} \wedge \cdots \wedge d y_{nn}.
\]
A technical result, Theorem \ref{thm:theta}, which is required for the proof of Theorem \ref{thm:volform} is postponed to the end of this section.  

Now let us explain the construction of ${\bf G}$. As in the proof of \cite[Proposition 3.5]{GSV:related}, solve for the row vector $[g_1 ~ g_2 ~ g_3 ~ \cdots ~ g_{{(n-1)^2}-n}]$ such that
\begin{align} \label{eq:rowoper}
    [y_{21} ~ y_{22} ~ \underbrace{0 ~ \cdots ~ 0}_{{(n-1)^2}-n-2}] + [g_1 ~ g_2 ~ g_3 ~ \cdots ~ g_{{(n-1)^2}-n}] \Psi_{[n+1,{(n-1)^2}]}^{[n+1,{(n-1)^2}]} = [\underbrace{0 ~ 0 ~ 0 ~ \cdots ~ 0}_{{(n-1)^2}-n}],
\end{align}
where we recall that $\Psi$ is the staircase matrix (\ref{eq:Psi}). Then define
\[
    {\bf G}_1=\begin{bmatrix}
        1 & {\bf g}_1\\
        0 & \bbone_{n-2}
    \end{bmatrix}, \text{~where~} {\bf g}_1=[g_1 ~ g_2 ~ \cdots ~ g_{n-2}].
\]
To define ${\bf G}_2$, we solve for the row vector $\mathbf{h} = [h_1 ~ h_2 ~ \cdots ~ h_{{(n-1)^2}-2n}]$ such that
\[
[y_{31} ~ y_{32} ~ y_{33} ~ \underbrace{0 ~ \cdots ~ 0}_{{(n-1)^2}-2n-3}] + [h_1 ~ h_2 ~ \cdots ~ h_{{(n-1)^2}-2n}] \Psi_{[2n+1,{(n-1)^2}]}^{[2n+1,{(n-1)^2}]} = [\underbrace{0 ~ \cdots ~ 0}_{{(n-1)^2}-2n}].
\]
Define
\[
    {\bf G}_2=\begin{bmatrix}
    \bbone_2  &{\bf g}_2\\
     0 & \bbone_{n-3}
    \end{bmatrix}, \text{~where~} {\bf g}_2=\begin{bmatrix}
        0 & 0& \cdots &0\\
        h_1 &h_2 & \cdots & h_{n-3}
    \end{bmatrix}.
\]

Iterating this process, we get all ${\bf G}_i$'s. We then show
\begin{thm} \label{thm:volform}
    We have $d(X,Y) = d(X_1,Y)=d(X_2, Y)=\cdots=d(X_{n-2},Y)$.
\end{thm}

\begin{proof}
     We  first show $d(X,Y) = d(X_1,Y)$. Write $x_{ij}'$ for the $(i,j)$-entry of $X_1$. Thus $X_1={\bf G}_1'X$ gives 
   \begin{equation}\label{eq:forG1}
       [x'_{23} ~ x'_{24} ~ \cdots ~ x'_{2n}] = [x_{23} ~ x_{24} ~ \cdots ~ x_{2n}] + [g_1 ~ g_2 ~ \cdots ~ g_{n-2}] X_{[3,n]}^{[3,n]}. 
   \end{equation}
    Differentiating both sides of \eqref{eq:forG1} and writing $\mathbb{J}_1$ for the Jacobi matrix $[\frac{\partial g_j}{\partial x_{2i}}]_{i \in [3,n]}^{j \in [1,n-2]}$, we have
    \begin{align*}
         [dx'_{23} ~ dx'_{24} ~ \cdots ~ dx'_{2n}] &=  [dx_{23} ~ dx_{24} ~ \cdots ~ dx_{2n}] + [dg_1 ~ dg_2 ~ \cdots ~ dg_{n-2}] X_{[3,n]}^{[3,n]} + \heartsuit,\\
       & =  [dx_{23} ~ dx_{24} ~ \cdots ~ dx_{2n}] + [dx_{23} ~ dx_{24} ~ \cdots ~ dx_{2n}] \mathbb{J}_1 X_{[3,n]}^{[3,n]} + \heartsuit \\
        &=  [dx_{23} ~ dx_{24} ~ \cdots ~ dx_{2n}] \left({\mathbbm 1}_{n-2} + \mathbb{J}_1 X_{[3,n]}^{[3,n]}\right) + \heartsuit,
    \end{align*}
    where $\heartsuit$ stands for terms that are irrelevant for the computation of $d(X_1,Y)$. Hence
    \[
    dx'_{23} \wedge dx'_{24} \wedge \cdots \wedge dx'_{2n} = \det \left( {\mathbbm 1}_{n-2} + \mathbb{J}_1 X_{[3,n]}^{[3,n]} \right) dx_{23} \wedge dx_{24} \wedge \cdots \wedge dx_{2n} + \heartsuit,
    \]
    from which it follows that
    \[
    d(X_1,Y) = \det \left( {\mathbbm 1}_{n-2} + \mathbb{J}_1 X_{[3,n]}^{[3,n]} \right) d(X,Y).
    \]

    We now compute $\mathbb{J}_1$. Let ${\mathbb {\Delta}} = \Psi_{[n+1,{(n-1)^2}]}^{[n+1,{(n-1)^2}]}$ and $\Rho = {\mathbb {\Delta}}^{-1}$. Differentiating both sides of  (\ref{eq:rowoper}), we get
    \begin{align*}
        (d \mathbf{g}) {\mathbb {\Delta}} + \mathbf{g} d{\mathbb {\Delta}} = \heartsuit,
    \end{align*}
    where ${\bf g}=[g_1 ~ g_2 ~ \cdots ~ g_{(n-1)^2-n}]$. Hence we have
    \[
    d \mathbf{g} = - \mathbf{g} (d {\mathbb {\Delta}}) {\mathbb {\Delta}}^{-1} + \heartsuit = - \mathbf{g} (d {\mathbb {\Delta}}) {\rm P} + \heartsuit.
    \]
    Write $e_{ij}$ (resp. $e_i$) for the $({(n-1)^2}-n) \times ({(n-1)^2}-n)$ (resp. $({(n-1)^2}-n) \times 1$) matrix whose only non-zero entry is in position $(i,j)$ (resp. $(i,1)$), and this non-zero entry is $1$.  Denote by
    \[
        E_k=e_{(n-1),k} + e_{2(n-1),n+k} + \cdots + e_{(n-3)(n-1),(n-4)n+k},
    \]
    where the index appeared are nothing but the positions where $x_{2k}$ sits in $\mathbb {\Delta}$ for $3\leq k\leq n$. Thus
   \begin{align}
       d\mathbf{g}=-\mathbf{g}(d{\mathbb {\Delta}}){\rm P}+\heartsuit =-\mathbf{g}\left(\sum_{k=3}^n E_k dx_{2k}+\heartsuit\right){\rm P}+\heartsuit=-\sum_{k=3}^n \mathbf{g}E_k {\rm P}dx_{2k}+\heartsuit.
   \end{align}
   Since $e_{ij}=e_i\cdot e_j^T$, we have,   
    \[
        \mathbf{g}E_k {\rm P}=g_{n-1}\rho_k+g_{2(n-1)}\rho_{n+k}+\cdots+g_{(n-3)(n-1)}\rho_{(n-4)n+k}, 
    \]
    where $k\in[3,n]$ and we write $\rho_1, \ldots, \rho_{{(n-1)^2}-n}$ for the rows of $\Rho$. Hence
    \begin{align*}
        \mathbb{J}_1 = -g_{n-1} {\rm P}_{[3,n]}^{[1,n-2]} -g_{2(n-1)} {\rm P}_{[n+3,n+n]}^{[1,n-2]} - \cdots -g_{(n-3)(n-1)} {\rm P}_{[(n-4)n+3,(n-4)n+n]}^{[1,n-2]} + \heartsuit.
    \end{align*}

    For $i \in [1,n-3]$, define $\A_i = \Rho_{[(i-1)n+3,(i-1)n+n]}^{[1,n-2]} X_{[3,n]}^{[3,n]}$.  We have proved
    \[
    d(X_1,Y) = \det \bigl( {\mathbbm 1}_{n-2} - g_{n-1}\A_1 - g_{2(n-1)}\A_2 - \cdots - g_{(n-3)(n-1)}\A_{n-3} \bigr) d(X,Y).
    \]
    Let $\mathbb K$ be the field of rational functions on $B\times B_-$, so $g_{n-1}, g_{2(n-1)}, \ldots, g_{(n-3)(n-1)} \in \mathbb K$.  By Theorem \ref{thm:theta} below, the matrix
    \[
    - g_{n-1}\A_1 - g_{2(n-1)}\A_2 - \cdots - g_{(n-3)(n-1)}\A_{n-3}
    \]
    is nilpotent.  Hence
    \[
    \det \bigl( {\mathbbm 1}_{n-2} - g_{n-1}\A_1 - g_{2(n-1)}\A_2 - \cdots - g_{(n-3)(n-1)}\A_{n-3} \bigr) = 1.
    \]
    A similar argument shows $d(X_2,Y)=d(X_1,Y)$. Iterating this process, we are done.
\end{proof}

To compute $d(X_{n-2},Y)$, let us introduce 
\begin{align*}
    \xi_{i+(j-1)(n-2)} = \det M_{[i+j-1,n+i-1]}^{[j,n]}, \hs \text{for} ~ i \in [1,n-2] ~ \text{and} ~ j \in [1,n],
\end{align*}
where $M$ is the matrix in \eqref{eq:bigM}. By Theorem \ref{thm:gsv3.5} and $\varphi_{(n-2)n+1} = y_{n1} = \det M_{[n-1,n-1]}^{[1,1]}$, we then have
\begin{equation}\label{eq:dxi}
d \varphi_1 \wedge \cdots \wedge d \varphi_{(n-1)^2} = \pm (\varphi_{n+1} \varphi_{2n+1} \cdots \varphi_{(n-2)n+1})^n d\xi_1 \wedge \cdots \wedge d\xi_{n^2-2n} \wedge dy_{n1}. 
\end{equation}

\begin{prop}\label{prop:newxi}
    We have
    \begin{align*}
        d \xi_1 \wedge d \xi_2 \wedge \cdots & \wedge d \xi_{n^2-2n} \wedge d x_{22} \wedge \cdots \wedge d x_{nn} \wedge d y_{n1} \wedge  d y_{11} \wedge \cdots \wedge d y_{nn}
         =\pm (\xi_{n-1} \xi_n \cdots \xi_{n^2-2n}) d(X_{n-2},Y).
    \end{align*}
\end{prop}
\begin{proof}
    Rewrite $M$ as
\begin{align*}\NiceMatrixOptions{margin=4pt,renew-matrix}
    M =
    \begin{bmatrix}
        a_1 & b_1' &&&&& \\
        a_2 & a_{n-1} & b_2' &&&& \\
        a_3 & a_n & a_{2n-3} & b_3' &&& \\
        \vdots & \vdots & \vdots & \vdots & \ddots && \\
        a_{n-2} & a_{2n-5} & a_{3n-8} & a_{4n-11} & \cdots & b_{n-2}' & \\
        b_n & a_{2n-4} & a_{3n-7} & a_{4n-10} & \cdots & a_{n^2-4n+5} & b_{n-1}' \\
        \hline
        & b_1 & a_{3n-6} & a_{4n-9} & \cdots & a_{n^2-4n+6} & a_{n^2-3n+3} \\
        && b_2 & a_{4n-8} & \cdots & a_{n^2-4n+7} & a_{n^2-3n+4} \\
        &&& \ddots & \ddots & \vdots & \vdots \\
        &&&& b_{n-3} & a_{n^2-3n+2} & a_{n^2-2n-1} \\
        &&&&& b_{n-2} & a_{n^2-2n} \\
        &&&&&& b_{n-1}
    \end{bmatrix}.
\end{align*}
We then have
\begin{align*}
    & \xi_{n^2-3n+3} = \xi_{1+(n-1)(n-2)} = \det M_{[n,n]}^{[n,n]} = a_{n^2-3n+3} \\
    & \xi_{n^2-3n+4} = \xi_{2+(n-1)(n-2)} = \det M_{[n+1,n+1]}^{[n,n]} = a_{n^2-3n+4} \\
    & \quad \quad \vdots \\
    & \xi_{n^2-2n} = \xi_{(n-2)+(n-1)(n-2)} = \det M_{[2n-3,2n-3]}^{[n,n]} = a_{n^2-2n}.
\end{align*}
 Hence,
    \[
    d\xi_{n^2-3n+3} \wedge d\xi_{n^2-3n+4} \wedge \cdots \wedge d\xi_{n^2-2n} = da_{n^2-3n+3} \wedge da_{n^2-3n+4} \wedge \cdots \wedge da_{n^2-2n}.
    \]
Expanding along the first column,  for each $k \in [1,n^2-3n+2]$, we have 
\begin{equation}\label{eq:xik}
\xi_k = a_k \xi_{k+n-2} + \eta_k   
\end{equation}
for some $\eta_k \in \CC[a_{k+1}, a_{k+2}, \ldots, a_{n^2-2n}, b_1, \ldots, b_n, b_1', \ldots, b_{n-1}']$. Thus
    \begin{align*}
        & ~ d\xi_{n^2-3n+2} \wedge d\xi_{n^2-3n+3} \wedge \cdots \wedge d\xi_{n^2-2n} \wedge d b_1 \wedge \cdots \wedge d b_n \wedge d b_1' \wedge \cdots \wedge d b_{n-1}' \\
        = & ~ d(a_{n^2-3n+2} \xi_{n^2-2n} + \eta_{n^2-3n+2}) \wedge da_{n^2-3n+3} \wedge da_{n^2-3n+4} \wedge \cdots \wedge da_{n^2-2n} \wedge \\
        & \,\,\,\, \wedge d b_1 \wedge \cdots \wedge d b_n \wedge d b_1' \wedge \cdots \wedge d b_{n-1}' \\
        = & ~ d(a_{n^2-3n+2} \xi_{n^2-2n}) \wedge da_{n^2-3n+3} \wedge da_{n^2-3n+4} \wedge \cdots \wedge da_{n^2-2n} \wedge d b_1 \wedge \cdots \wedge d b_n \wedge d b_1' \wedge \cdots \wedge d b_{n-1}' \\
        = & ~ \xi_{n^2-2n} da_{n^2-3n+2} \wedge da_{n^2-3n+3} \wedge da_{n^2-3n+4} \wedge \cdots \wedge da_{n^2-2n}\wedge d b_1 \wedge \cdots \wedge d b_n \wedge d b_1' \wedge \cdots \wedge d b_{n-1}'.
    \end{align*}
    Iterating this computation, we are done.
\end{proof}

\begin{rem}
    It is worth pointing out that \eqref{eq:xik} is reminiscent of the formula for a homogeneous Poisson prime element in a Poisson CGL extension.  See (\ref{eq:phik-CGL}). More specifically, let $R = \CC[b_1, \ldots, b_n, b_1', \ldots, b_{n-1}']$.  Suppose that $R[a_{n^2-2n}, a_{n^2-2n-1}, \ldots, a_1]$ is a Poisson CGL extension such that $k^-$ is $k+n-2$, then the recursive formula for the homogeneous Poisson primes is exactly of the form $\xi_k = a_k \xi_{k+n-2} + \eta_k$.   
\end{rem}

\begin{thm} \label{thm:dphi}
    We have
    \begin{align*}
        \frac{d \varphi_1}{\varphi_1} \wedge \cdots \wedge \frac{d \varphi_{(n-1)^2}}{\varphi_{(n-1)^2}} \wedge & d x_{22} \wedge \cdots \wedge d x_{nn} \wedge d y_{11} \wedge \cdots \wedge d y_{nn} = \pm \frac{1}{\varphi_1} d(X_{n-2},Y).
    \end{align*}
\end{thm}

\begin{proof}
    Combining Theorem \ref{thm:gsv3.5}, \eqref{eq:dxi} and Proposition \ref{prop:newxi}, we have
    \begin{align*}
        d \varphi_1 \wedge \cdots \wedge d \varphi_{(n-1)^2} \wedge  d x_{22} \wedge \cdots \wedge d x_{nn} \wedge d y_{11} \wedge \cdots \wedge d y_{nn} = \pm\varphi_2 \cdots \varphi_{(n-1)^2}d(X_{n-2},Y).\tag*{\qedhere}
    \end{align*}
\end{proof}

 We now prove the technical result that we needed to prove Theorem \ref{thm:volform}.  Retain the notation in the statement and proof of Theorem \ref{thm:volform}.

\begin{thm} \label{thm:theta}
    We have
    \[
    \A_1 [\A_1 ~ \A_2 ~ \cdots ~ \A_{n-4} ~ \A_{n-3}] = - [\A_2 ~ \A_3 ~ \cdots ~ \A_{n-3} ~ 0].
    \]
    In particular, we have
    \[
    \A_i = (-1)^{i-1} \A_1^i ~ \forall i \in [1,n-3] \quad \text{and} \quad \A_1^{n-2} = 0,
    \]
    so a $\mathbb K$-linear combination of $\A_1, \A_2, \ldots, \A_{n-3}$ is a polynomial of the nilpotent matrix $\A_1$ with coefficients in $\mathbb K$, and, hence, is nilpotent.
\end{thm}

\begin{proof}
    For $i \in [1, n-3]$, let $\B_i = X_{[3,n]}^{[3,n]} \Rho_{[(i-1)n+3,(i-1)n+n]}^{[1,n-2]}$.  Since $\A_i = (X_{[3,n]}^{[3,n]})^{-1} \B_i X_{[3,n]}^{[3,n]}$ is a conjugation of $\B_i$, we prove the corresponding statements for $\B_i$, which turns out to be easier.
    
    Divide the matrix ${\mathbb {\Delta}}$ into blocks such that the rows are of sizes
    \[
    \underbrace{n-2, 1, n-2, 1, \ldots, n-2, 1}_{2(n-3) \text{ parts}}, n-2
    \]
    from top to bottom and the columns are of sizes
    \[
    \underbrace{2, n-2, 2, n-2, \ldots, 2, n-2}_{2(n-3) \text{ parts}}, 1
    \]
    from left to right. Recall that $\Rho={\mathbb {\Delta}}^{-1}$. Divide $\Rho$ into blocks of conformable partitions of ${\mathbb {\Delta}}$. 
    For example, for $n=5$, we divide the ${\mathbb {\Delta}}$ as in the following picture:
    \[
    {\mathbb {\Delta}} = \left[\begin{array}{cc|ccc|cc|ccc|c}
     y_{31} & y_{32} & y_{33} & 0 & 0                  & 0 & 0 & 0 & 0 & 0 & 0  \\
     y_{41} & y_{42} & y_{43} & y_{44} & 0                     & 0 & 0 & 0 & 0 & 0 & 0\\
    y_{51} & y_{52} & y_{53} & y_{54} & y_{55}                    & 0 & 0 & 0 & 0 & 0 & 0\\ \hline
     0 & x_{22} & x_{23} & x_{24} & x_{25} & {y_{21}}    &{y_{22}}  & {0} & {0} & {0} & {0}\\ \hline
     0 & 0 & x_{33} & x_{34} & x_{35} & y_{31}   &y_{32} & y_{33}  & 0 & 0  & 0\\
     0 & 0 & 0 & x_{44} & x_{45} & y_{41}  & y_{42} & y_{43} & y_{44} & 0 & 0\\
     0 & 0 & 0 & 0 & x_{55}  & y_{51}  & y_{52} & y_{53} & y_{54} & y_{55}& 0\\ \hline
     0 & 0 & 0 & 0 & 0  &0 & x_{22} & x_{23} & x_{24} & x_{25} & {y_{21}}\\ \hline
    0 & 0 & 0 & 0 & 0  &0 & 0 & x_{33} & x_{34} & x_{35}  & y_{31}\\
    0 & 0 & 0 & 0 & 0  &0 & 0 & 0 & x_{44} & x_{45}  & y_{41}\\
     0 & 0 & 0 & 0 & 0  &0 & 0 & 0 & 0 & x_{55}  & y_{51}\end{array}\right].
    \]

    Let ${\mathbb {\Delta}}_{ij}$ (resp. $\Rho_{ij}$) be the $(i,j)$-th block of ${\mathbb {\Delta}}$ (resp. $\Rho$). One readily observes that ${\mathbb {\Delta}}_{ij} = 0$ if $|i-j|>1$. Moreover, we also have ${\mathbb {\Delta}}_{ij} = {\mathbb {\Delta}}_{i+2,j+2}$ for $i \in [1,2(n-3)-1]$ and $j \in [1,2(n-4)]$; ${\mathbb {\Delta}}_{2i+1,2i} = X_{[3,n]}^{[3,n]}$ for $i \in [1,n-3]$; the matrix $\Rho_{[(i-1)n+3,(i-1)n+n]}^{[1,n-2]}$ which appeared in the definition of $\B_i$ is $\Rho_{2i,1}$ for $i \in [1,n-3]$.  In particular, we have $\B_1 = {\mathbb {\Delta}}_{32} \Rho_{21}$.

    We write ${\mathbb {\Delta}}_{{\rm bl}(i)}$ (resp. ${\mathbb {\Delta}}^{{\rm bl}(j)}$) for the $i$-th block row (resp. $j$-th block column) of ${\mathbb {\Delta}}$.  Define $\Rho_{{\rm bl} (i)}$ and $\Rho^{{\rm bl} (j)}$ similarly.  

    Since ${\mathbb {\Delta}}_{{\rm bl} (2i+1)} \Rho^{{\rm bl} (1)} = 0$ for $i \in [1,n-4]$, and $\B_i =  {\mathbb {\Delta}}_{2i+1,2i} \Rho_{2i,1}={\mathbb {\Delta}}_{3,2} \Rho_{2i,1}$, we have
    \begin{align*}
        \B_i = & ~ {\mathbb {\Delta}}_{2i+1,2i} \Rho_{2i,1} = - ({\mathbb {\Delta}}_{2i+1,2i+1} \Rho_{2i+1,1} + {\mathbb {\Delta}}_{2i+1,2i+2} \Rho_{2i+2,1}) \\
        = & - [{\mathbb {\Delta}}_{2i+1,2i+1} ~ {\mathbb {\Delta}}_{2i+1,2i+2}] 
        \begin{bmatrix}
            \Rho_{2i+1,1} \\
            \Rho_{2i+2,1}
        \end{bmatrix}
        = - Y_{[3,n]}
        \begin{bmatrix}
            \Rho_{2i+1,1} \\
            \Rho_{2i+2,1}
        \end{bmatrix}
        = - [{\mathbb {\Delta}}_{11} ~ {\mathbb {\Delta}}_{12}] 
        \begin{bmatrix}
            \Rho_{2i+1,1} \\
            \Rho_{2i+2,1}
        \end{bmatrix}.
    \end{align*}
    Since $\Rho_{{\rm bl} (2)} {\mathbb {\Delta}}^{{\rm bl} (1)} = 0$ and $\Rho_{{\rm bl} (2)} {\mathbb {\Delta}}^{{\rm bl} (2)} = {\mathbbm 1}_{n-2}$, we have
    \begin{align*}
        \Rho_{21} {\mathbb {\Delta}}_{11} = - \Rho_{22} {\mathbb {\Delta}}_{21} \quad \text{and} \quad \Rho_{21} {\mathbb {\Delta}}_{12} = - \Rho_{22} {\mathbb {\Delta}}_{22} - \Rho_{23} {\mathbb {\Delta}}_{32} + {\mathbbm 1}_{n-2}.
    \end{align*}
    It follows that
    \begin{align*}
        - \Rho_{21} [{\mathbb {\Delta}}_{11} ~ {\mathbb {\Delta}}_{12}] = [\Rho_{22} {\mathbb {\Delta}}_{21} ~ \Rho_{22} {\mathbb {\Delta}}_{22} + \Rho_{23} {\mathbb {\Delta}}_{32} - {\mathbbm 1}_{n-2}] = [\Rho_{22} ~ \Rho_{23}] 
        \begin{bmatrix}
            {\mathbb {\Delta}}_{21} & {\mathbb {\Delta}}_{22} \\
            0 & {\mathbb {\Delta}}_{32}
        \end{bmatrix}
        - [0 ~ {\mathbbm 1}_{n-2}].
    \end{align*}
    Combining these computations, we get, for $i \in [1,n-4]$,
    \begin{align*}
        \B_1 \B_i = & ~{\mathbb {\Delta}}_{32} \Rho_{21} (- [{\mathbb {\Delta}}_{11} ~ {\mathbb {\Delta}}_{12}] 
        \begin{bmatrix}
            \Rho_{2i+1,1} \\
            \Rho_{2i+2,1}
        \end{bmatrix}
        ) = {\mathbb {\Delta}}_{32} (- \Rho_{21} [{\mathbb {\Delta}}_{11} ~ {\mathbb {\Delta}}_{12}]) 
        \begin{bmatrix}
            \Rho_{2i+1,1} \\
            \Rho_{2i+2,1}
        \end{bmatrix} \\
        = & ~{\mathbb {\Delta}}_{32} \bigl(
        [\Rho_{22} ~ \Rho_{23}] 
        \begin{bmatrix}
            {\mathbb {\Delta}}_{21} & {\mathbb {\Delta}}_{22} \\
            0 & {\mathbb {\Delta}}_{32}
        \end{bmatrix}
        - [0 ~ {\mathbbm 1}_{n-2}]
        \bigr)
        \begin{bmatrix}
            \Rho_{2i+1,1} \\
            \Rho_{2i+2,1}
        \end{bmatrix} \\
        = & ~{\mathbb {\Delta}}_{32} [\Rho_{22} ~ \Rho_{23}] 
        \begin{bmatrix}
            {\mathbb {\Delta}}_{21} & {\mathbb {\Delta}}_{22} \\
            0 & {\mathbb {\Delta}}_{32}
        \end{bmatrix}
        \begin{bmatrix}
            \Rho_{2i+1,1} \\
            \Rho_{2i+2,1}
        \end{bmatrix}
        - {\mathbb {\Delta}}_{32} \Rho_{2i+2,1} \\
        = & ~{\mathbb {\Delta}}_{32} [\Rho_{22} ~ \Rho_{23}] X_{[2,n]}
        \begin{bmatrix}
            \Rho_{2i+1,1} \\
            \Rho_{2i+2,1}
        \end{bmatrix}
        - \B_{i+1}.
    \end{align*}

    For $j \in [2,n-3]$, since $\Rho_{{\rm bl} (2)} {\mathbb {\Delta}}^{{\rm bl} (2j-1)} = 0$ and $\Rho_{{\rm bl} (2)} {\mathbb {\Delta}}^{{\rm bl} (2j)} = 0$, we have
    \begin{align*}
        [\Rho_{2,2j} ~ \Rho_{2,2j+1}] X_{[2,n]} & = [\Rho_{2,2j} ~ \Rho_{2,2j+1}]
        \begin{bmatrix}
            {\mathbb {\Delta}}_{2j,2j-1} & {\mathbb {\Delta}}_{2j,2j} \\
            0 & {\mathbb {\Delta}}_{2j+1,2j}
        \end{bmatrix} \\
        &= -  [\Rho_{2,2j-2} ~ \Rho_{2,2j-1}]
        \begin{bmatrix}
            {\mathbb {\Delta}}_{2j-2,2j-1} & 0 \\
            {\mathbb {\Delta}}_{2j-1,2j-1} & {\mathbb {\Delta}}_{2j-1,2j}
        \end{bmatrix}
        = - [\Rho_{2,2j-2} ~ \Rho_{2,2j-1}] Y_{[2,n]}.
    \end{align*}
    Similarly, for $i \in [1,n-4]$, since ${\mathbb {\Delta}}_{{\rm bl} (2i)} \Rho^{{\rm bl} (1)} = 0$ and ${\mathbb {\Delta}}_{{\rm bl} (2i+1)} \Rho^{{\rm bl} (1)} = 0$, we have
    \begin{align*}
        X_{[2,n]} 
        \begin{bmatrix}
            \Rho_{2i-1,1} \\
            \Rho_{2i,1}
        \end{bmatrix}
        = &
        \begin{bmatrix}
            {\mathbb {\Delta}}_{2i,2i-1} & {\mathbb {\Delta}}_{2i,2i} \\
            0 & {\mathbb {\Delta}}_{2i+1,2i}
        \end{bmatrix}
        \begin{bmatrix}
            \Rho_{2i-1,1} \\
            \Rho_{2i,1}
        \end{bmatrix} \\
        = & - 
        \begin{bmatrix}
            {\mathbb {\Delta}}_{2i,2i+1} & 0 \\
            {\mathbb {\Delta}}_{2i+1,2i+1} & {\mathbb {\Delta}}_{2i+1,2i+2}
        \end{bmatrix}
        \begin{bmatrix}
            \Rho_{2i+1,1} \\
            \Rho_{2i+2,1}
        \end{bmatrix} 
        = - Y_{[2,n]} 
        \begin{bmatrix}
            \Rho_{2i+1,1} \\
            \Rho_{2i+2,1}
        \end{bmatrix}.
    \end{align*}
    Combining these, we see that, for $i \in [1,n-4]$,
    \begin{align*}
        [\Rho_{22} ~ \Rho_{23}] X_{[2,n]}
        \begin{bmatrix}
            \Rho_{2i+1,1} \\
            \Rho_{2i+2,1}
        \end{bmatrix} 
        &=  - [\Rho_{22} ~ \Rho_{23}] Y_{[2,n]}
        \begin{bmatrix}
            \Rho_{2i+3,1} \\
            \Rho_{2i+4,1}
        \end{bmatrix} 
        =  [\Rho_{24} ~ \Rho_{25}] X_{[2,n]}
        \begin{bmatrix}
            \Rho_{2i+3,1} \\
            \Rho_{2i+4,1}
        \end{bmatrix} 
        =  \cdots \\
        &=  [\Rho_{2,2(n-i-3)} ~ \Rho_{2,2(n-i-3)+1}] X_{[2,n]}
        \begin{bmatrix}
            \Rho_{2(n-3)-1,1} \\
            \Rho_{2(n-3),1}
        \end{bmatrix}\\
        &=  -[\Rho_{2,2(n-i-3)} ~ \Rho_{2,2(n-i-3)+1}]
        \begin{bmatrix}
            {\mathbb {\Delta}}_{2(n-3),2(n-3)+1} \Rho_{2(n-3)+1,1} \\
            {\mathbb {\Delta}}_{2(n-3)+1,2(n-3)+1} \Rho_{2(n-3)+1,1}
        \end{bmatrix}.
    \end{align*}
    Noticing that
    \begin{align*}
        \begin{bmatrix}
            {\mathbb {\Delta}}_{2(n-3),2(n-3)+1} \\
            {\mathbb {\Delta}}_{2(n-3)+1,2(n-3)+1}
        \end{bmatrix}
        = Y_{[2,n]}^{[1,1]} = 
        \begin{bmatrix}
            {\mathbb {\Delta}}_{2(n-i-3),2(n-i-3)+1} \\
            {\mathbb {\Delta}}_{2(n-i-3)+1,2(n-i-3)+1}
        \end{bmatrix}^{[1,1]},
    \end{align*}
    and $Y_{[2,n]}^{[1,1]}$ are the only non-zero entries in the first column of ${\mathbb {\Delta}}^{{\rm bl} (2(n-i-3)+1)}$, we see that
    \begin{align*}
        [\Rho_{22} ~ \Rho_{23}] X_{[2,n]}
        \begin{bmatrix}
            \Rho_{2i+1,1} \\
            \Rho_{2i+2,1}
        \end{bmatrix}
        = -\Rho_{{\rm bl} (2)} {\mathbb {\Delta}}^{{\rm bl} (2(n-i-3)+1)} \Rho_{2(n-3)+1} = 0.
    \end{align*}
    So we have proved
    \[
    \B_1 \B_i = - \B_{i+1} \quad \forall i \in [1,n-4].
    \]
    The proof of $\B_1 \B_{n-3} = 0$ is analogous.
\end{proof}

Combining Theorem \ref{thm:volform} and Theorem \ref{thm:dphi}, we have proved Theorem \ref{thm:dphiBB-}.

\subsection{Proof of Theorem \ref{thm:choice-GLn}} \label{sect:lowdegtermexp}
Retain the notation and assumption in $\S$\ref{sect:log-cansys}. In this section, we work out expressions of $\Phi^\low$, where $\Phi$ is given in (\ref{eq:initialG*}).  We start with expressions of the $\varphi_i$'s in terms of $U:=XY^{-1}$ and the diagonal entries of $Y$ for $(X,Y)\in B\times B_-$. Our expressions are analogous to \cite[Lemma 3.3, 3.4]{GSV20}.  Using these expressions, we find formulas for $\varphi_i^{\low}$.

Let $(X,Y) \in B \times B_-$.  Denote by $\Lambda$ the following $n\times (n-1)$ block matrix
\[
   \Lambda= 
    \begin{bmatrix}
        Y_{[2,n]} & & & & \\
        X_{[2,n]} & Y_{[2,n]} & & & \\
        & X_{[2,n]} & Y_{[2,n]} & & \\
        & & \ddots & \ddots & \\
        & & & X_{[2,n]} & Y_{[2,n]} \\
        & & & & X_{[2,n]}
    \end{bmatrix}.
\]
Note that $\Lambda$ is a square matrix of size $n(n-1)$. Put
\begin{equation}\label{eq:lambda}
    \lambda_i=\det\Lambda_{[i,n(n-1)]}^{[i,n(n-1)]}.
\end{equation}  
Because of the shape of $\Lambda$, we have
\[
   \lambda_{n(n-1)-i+1}=x_{n-i+1,n-i+1}\cdots x_{n,n}, \text{~for~} i=1,\ldots, n-1,
\]
and
\begin{equation}\label{eq:lambdai}
    \lambda_i=\varphi_i\cdot x_{22}\cdots x_{nn},\hs \text{~for~} i=1,\ldots, (n-1)^2.
\end{equation}
We reparametrize indices and write $\lambda_{p(n-1)+i}$, with $i \in [1, n-1]$ and $p \in [0, n-2]$.

\begin{thm}\label{Thm:expforphi}
    Let $(X,Y) \in B \times B_-$.  Recall that $U=XY^{-1}$ and $e_1=[1 ~ 0 ~ \cdots ~0]^T\in \mathbb{C}^n$. If $i \le p+1$, up to a sign, we have
    \begin{align*}
       \lambda_{p(n-1)+i} = & \det (Y_{[n-p+i,n]}^{[n-p+i,n]}) \cdot (\det Y)^{n-p-1} \cdot \\
        & \det \begin{bmatrix}
        (U^{n-p})^{[n-p+i,n]} &  (U^{n-p-1})^{[1,i]} & U^{n-p-2} e_1 & \cdots & Ue_1 & e_1
    \end{bmatrix}.
    \end{align*}
    If $i\geq p+1$, up to a sign, we have
    \begin{align*}
       \lambda_{p(n-1)+i}  = & \det (Y_{[i-p,n]}^{[i-p,n]}) \cdot (\det Y)^{n-p-2} \cdot \\
        & \det \begin{bmatrix}
        (U^{n-p-1})^{[i-p,i]} & U^{n-p-2} e_1 & \cdots & U^2 e_1 & U e_1 & e_1
    \end{bmatrix}.
    \end{align*}

\end{thm}

\begin{rem} \label{rem:lambda1}
    Notice that, when $i = p+1$, the two formulas in Theorem \ref{Thm:expforphi} coincide.  Hence, although the case $i = p+1$ appears twice in the statement of Theorem \ref{Thm:expforphi}, we only have one formula for $\lambda_{p(n-1)+(p+1)}$.  However, in the proof below, we will group the case $i = p+1$ together with $i > p+1$.
    
    In the case of $i = 1, p = 0$, we have
\begin{align*}
    \lambda_1 & = \det (Y_{[1,n]}^{[1,n]}) (\det Y)^{n-2} \det
    \begin{bmatrix}
        U^{n-1} e_1 & U^{n-2} e_1 & \cdots & U^2 e_1 & U e_1 & e_1
    \end{bmatrix} \\
    & = (\det Y)^{n-1} \det
    \begin{bmatrix}
        U^{n-1} e_1 & U^{n-2} e_1 & \cdots & U^2 e_1 & U e_1 & e_1
    \end{bmatrix}.
\end{align*}
This agrees with the formula  in \cite[Lemma 3.3]{GSV20}.

In the case of $i = 2, p = 0$, we have
\begin{align*}
    \lambda_2 = \det (Y_{[2,n]}^{[2,n]}) (\det Y)^{n-2} \det
    \begin{bmatrix}
        U^{n-1} e_2 & U^{n-2} e_1 & \cdots & U^2 e_1 & U e_1 & e_1
    \end{bmatrix},
\end{align*}
which agrees with the formula in \cite[Lemma 3.4]{GSV20}.   
\end{rem}
\begin{proof}[Proof of Theorem \ref{Thm:expforphi}]
    Let $\Theta$ be the $n \times n$ block matrix
\[
    \Theta =
    \begin{bmatrix}
        Y & 0 & 0 & \cdots & 0 & 0 \\
        X & Y & 0 & \cdots & 0 & 0 \\
        0 & X & Y & \cdots & 0 & 0 \\
        \vdots & \vdots & \vdots & \ddots & \vdots & \vdots \\
        0 & 0 & 0 & \cdots & Y & 0 \\
        0 & 0 & 0 & \cdots & X & Y
    \end{bmatrix}.
\]
Write $m=n(n-1)$. Denote by $\bm{U}$ the following $1\times (n-1)$ block matrix
\[
    \bm{U}=-\begin{bmatrix}
        (-U)^{n-1} &  (-U)^{n-2} &  (-U)^{n-3} & \cdots & -U
    \end{bmatrix}.
\]
Direct computation shows that
\[
    \bm{U}\Theta_{[1,m]}^{[1,m]}=[
        \underbrace{0 ~ 0 ~ \cdots ~ 0}_{n-2} ~ X].
\]
Thus we have
\[
    \Theta=\begin{bmatrix}
        {\mathbbm 1}_{m} & 0\\
        \bm{U} & \mathbbm{1}_{n}
    \end{bmatrix}\cdot \begin{bmatrix}
        \Theta_{[1,m]}^{[1,m]} & 0\\
        0 & Y
    \end{bmatrix}=:L_1\cdot L_2.
\]

By definition, for $k \in [1,m]$, we have
\[\lambda_k = \det \Theta_{J_k}^{[k,m]},\] where $J_k$ consists of the last $m-k+1$ elements of $[1,n^2] \setminus \{1, n+1, 2n+1, \ldots, (n-1)n+1\}$. Thus by the Binet-Cauchy formula, we have
\begin{align*}
    \lambda_k = \det \Theta_{J_k}^{[k,m]} = \det (L_1L_2)_{J_k}^{[k,m]}= \sum \limits_{\substack{K \subset[1,n^2] \\ |K|=m-k+1}} \det (L_1)_{J_k}^K\cdot \det (L_2)_K^{[k,m]}.
\end{align*}
Since $L_2$ is block diagonal and lower triangular, $\det (L_2)_K^{[k,m]} \neq 0$ if and only if $K= [k,m]$. Thus
\[
    \lambda_k = \det (L_1)_{J_k}^{[k,m]} \det (L_2)_{[k,m]}^{[k,m]}.
\]

We deal with the two determinants that appear in the last expression separately. Set
\[
    L_1'=\begin{bmatrix}
        L & 0\\
        {\bm U} & e_1
    \end{bmatrix}, \text{~where~}L=
    \begin{bmatrix}
        (\mathbbm{1}_{n})_{[2,n]} & 0 & 0 & \cdots & 0  \\
        0 & (\mathbbm{1}_{n})_{[2,n]} & 0 & \cdots & 0  \\
        \vdots & \vdots & \vdots & \ddots & \vdots \\
        0 & 0 & 0 & \cdots &(\mathbbm{1}_{n})_{[2,n]}  \\
    \end{bmatrix}.
\]
Then, up to a sign, we have 
\[
    \det (L_1)_{J_k}^{[k,m]}=\det(L_1')_{[k,m+1]}^{[k,m+1]}.
\]

Consider $k=p(n-1)+i=pn+(i-p)$. If $i-p<1$, the $(k,k)$-entry of $L_1'$ lies in the $(p+1,p)$-block of $L_1'$, which is a zero block. Hence by Laplace expansion along the first $(n-1)^2-k+1$ rows of $(L_1')_{[k,m+1]}^{[k,m+1]}$, up to a sign, we have
\[
    \det((L_1')_{[k,m+1]}^{[k,m+1]})=\det \begin{bmatrix}
        (U^{n-p})^{[n-p+i,n]} &  (U^{n-p-1})^{[1,i]} & U^{n-p-2} e_1 & \cdots & Ue_1 & e_1
    \end{bmatrix}.
\]
Since $k=(p-1)n+(n+i-p)$, we have
\[
    \det (L_2)_{[k,m]}^{[k,m]}=\det Y_{[n-p+i,n]}^{[n-p+i,n]} \cdot (\det Y)^{n-p-1}.
\]

If $i-p\geq 1$, the $(k,k)$-entry of $L'_1$ lies in the $(p+1,p+1)$-block of $L_1'$. Again, up to a sign,
\begin{align*}
    \det(L_1')_{[k,m+1]}^{[k,m+1]}&=\det \begin{bmatrix}
        (U^{n-p-1})^{[i-p,i]} & U^{n-p-2} e_1 & \cdots & U^2 e_1 & U e_1 & e_1
    \end{bmatrix},\\
    \det (L_2)_{[k,m]}^{[k,m]}&=\det Y_{[i-p,n]}^{[i-p,n]} \cdot (\det Y)^{n-p-2},
\end{align*}
which implies the desired expressions. 
\end{proof}

Now let us turn to the calculation of $\lambda_i^{\low}$. Recall that $\b$ (resp. $\b_-$) is the Lie algebra of $B$ (resp. $B_-$). Let $(x,y)\in \mathfrak{b}\oplus\mathfrak{b}_-$, so that $(X = \exp(x), Y = \exp(y)) \in B\times B_-$.  We have
\begin{align*}
    U = XY^{-1} = \exp(x) \exp(y)^{-1} = \exp(x) \exp(-y) = \exp(x-y) + \cdots,
\end{align*}
where $\cdots$ stands for expressions of the entries of $x, y$ of degree at least $2$. 
 Write $u$ for $x-y$ with entries $u_{ij}$, so $U = \mathbbm{1}_n + u + \cdots$. Denote by $\diag(y_1,\ldots,y_n)$ the diagonal of $y$. Thus the $u_{ij}$'s together with the $y_k$'s form a local holomorphic coordinate system on $B\times B_-$ near $(\mathbbm{1}_n,\mathbbm {1}_n)$. We will use these local coordinates to calculate $\lambda_i^\low$. 

Let $e_i$ be the $i$-th column of $\mathbbm{1}_n$. We need the following technical lemma.
\begin{lem} \label{Lem:non-zeroA}
    Let $u$ be as above and take pairwise distinct indices $i_1,\ldots,i_k\in [2,n]$.  The function
    \[
   \det
    \begin{bmatrix}
        e_{i_1} & \cdots & e_{i_k}  &u^{n-k-1} e_1 & \cdots & u e_1 & e_1
    \end{bmatrix}
    \]
    of $u$ is not the zero function.
\end{lem}
\begin{proof}
    Note that there is a permutation $\sigma$ of $[1,n]$ such that
    \[
        \sigma(i_1)=n-k+1, \sigma(i_2)=n-k+2, \ldots, \sigma(i_k)=n; \text{~and~} \sigma(1)=1.
    \]
    Thus there exists a permutation matrix $E_{\sigma}$ such that
    \[
        E_{\sigma}\begin{bmatrix}
        e_{i_1} & \cdots & e_{i_k}  &u^{n-k-1} e_1 & \cdots & u e_1 & e_1
    \end{bmatrix}=\begin{bmatrix}
        e_{n-k+1} & \cdots & e_{n}  &(u')^{n-k-1} e_1 & \cdots & u' e_1 & e_1
    \end{bmatrix},
    \]
    where $u'=E_{\sigma}uE_{\sigma}^{-1}$. Then by taking $u'$ to be the matrix such that $u'e_i=e_{i+1}$ for $i=1,\cdots,n-1$, we see that
    \[
        \det \begin{bmatrix}
        e_{n-k+1} & \cdots & e_{n}  &(u')^{n-k-1} e_1 & \cdots & u' e_1 & e_1
    \end{bmatrix}=\pm 1\neq 0. \tag*{\qedhere}
    \]
\end{proof}

\begin{proof}[Proof of Theorem \ref{prop:lambdalow}]
    Retain the notation in the statement of Theorem \ref{prop:lambdalow}. 
 Since $\varphi_{p(n-1)+i}x_{22} \cdots x_{nn} = \lambda_{p(n-1)+i}$  (see \eqref{eq:lambdai}) and $x_{22}^{\low} = \cdots = x_{nn}^{\low} = 1$, it is clear that $\varphi_{p(n-1)+i}^{\low} = \lambda_{p(n-1)+i}^{\low}$.
    
    We only prove the formula for $\lambda_{p(n-1)+i}^{\low}$ for $i < p+1$. The cases where $i = p+1$ and $i > p+1$ are similar. Notice that, for all $k \in [0, n-p-1]$, we have
\begin{align*}
    \sum \limits_{j=0}^k (-1)^j \binom{k}{j} U^{k-j} e_1 & = U^k \sum \limits_{j=0}^k \binom{k}{j} (-U^{-1})^j e_1 = U^k (\mathbbm{1}_n - U^{-1})^k e_1 \\
    & = \left( \exp (ku)+ \text{higher degree terms} \right) (u + \text{higher degree terms})^k e_1 \\
    & = u^k e_1 + \text{higher degree terms}.
\end{align*}
It is evident that the two sets $$\{U^{n-p-1} e_1, U^{n-p-2} e_1, \ldots, Ue_1, e_1\} ~ \text{and} ~ \{\sum \limits_{j=0}^k (-1)^j \binom{k}{j} U^{k-j} e_1: k = n-p-1, n-p-2, \ldots, 1, 0\}$$ of column vectors differ from each other by right multiplication by a lower triangular matrix with $1$'s on the diagonal.  Hence we have, up to a sign, 
\begin{align*}
    & ~ \det
    \begin{bmatrix}
        U^{n-p} e_{n-p+i} & \cdots & U^{n-p} e_n & U^{n-p-1} e_1 & \cdots & U^{n-p-1} e_i & U^{n-p-2} e_1 & \cdots & Ue_1 & e_1
    \end{bmatrix} \\
    = & ~ \det
    \begin{bmatrix}
        U^{n-p} e_{n-p+i} & \cdots & U^{n-p} e_n & U^{n-p-1} e_i & \cdots & U^{n-p-1} e_1 & U^{n-p-2} e_1 & \cdots & Ue_1 & e_1
    \end{bmatrix} \\
        = & ~ \det
    \begin{bmatrix}
        e_{n-p+i}+\heartsuit & \cdots & e_n+\heartsuit & e_i+\heartsuit & \cdots & e_2+\heartsuit & u^{n-p-1} e_1+\heartsuit & \cdots & u e_1+\heartsuit & e_1+\heartsuit
    \end{bmatrix},
\end{align*}
where $\heartsuit$ stands for the higher degree terms in each expression.  

By Lemma \ref{Lem:non-zeroA},
\begin{align*}
    \det
    \begin{bmatrix}
        e_{n-p+i} & \cdots & e_n & e_i & \cdots & e_2 & u^{n-p-1} e_1 & \cdots & u e_1 & e_1
    \end{bmatrix}
    \neq 0.
\end{align*}
Hence it is the lowest degree term of
\begin{align*}
    & \det
    \begin{bmatrix}
        U^{n-p} e_{n-p+i} & \cdots & U^{n-p} e_n & U^{n-p-1} e_1 & \cdots & U^{n-p-1} e_i & U^{n-p-2} e_1 & \cdots & Ue_1 & e_1
    \end{bmatrix}.
\end{align*}
Now the conclusion follows from Theorem \ref{Thm:expforphi}.
\end{proof}

As above, let $(x,y) \in \b \oplus \b_-$ and $u = x-y$.  We next work out an expression for $\lambda_{p(n-1)+i}^{\low}$ in terms of the minors of the matrix $u$.   Recall the definition of $\Delta_{I,J}(u)$ prior to Theorem \ref{thm:choice-GLn}.

\begin{prop} \label{Pro:LowinMinor}
    For any $k \in [1,n-1]$ and $I_1 \subseteq [1,n]$ with $|I_1| = k$, we have
    \begin{align*}
        \det
        \begin{bmatrix}
            u e_1 & u^2 e_1 & \cdots & u^k e_1
        \end{bmatrix}_{I_1}
        = \sum \limits_{\substack{I_2, \ldots, I_k \subseteq [2,n] \\ |I_i| = k-i+1 ~ \forall i \in [2,k]}} \prod \limits_{i=1}^k \Delta_{I_i, I_{i+1} \sqcup \{1\}}(u).
    \end{align*}
    Here, we use the convention that $I_{k+1} = \emptyset$. For $k=1$, the right hand side is interpreted as $\Delta_{I_1,\{1\}}(u)$.
\end{prop}

\begin{proof}
    We induct on $k$.  The case $k=1$ is immediate. For $k > 1$, by the Binet-Cauchy formula, we have
    \begin{align*}
         & ~ \det
        \begin{bmatrix}
            u e_1 & u^2 e_1 & \cdots & u^k e_1
        \end{bmatrix}_{I_1} 
        = \det (u
        \begin{bmatrix}
            e_1 & u e_1 & \cdots & u^{k-1} e_1
        \end{bmatrix})_{I_1} \\
        = & \sum \limits_{\substack{I_2 \subseteq [1,n] \\ |I_2| = k}} \Delta_{I_1,I_2}(u) \det (
        \begin{bmatrix}
            e_1 & u e_1 & \cdots & u^{k-1} e_1
        \end{bmatrix}_{I_2})
        = \!\!\!\! \sum \limits_{\substack{I_2 \subseteq [2,n] \\ |I_2| = k-1}} \!\! \Delta_{I_1,I_2 \sqcup \{1\}}(u) \det (
        \begin{bmatrix}
            u e_1 & \cdots & u^{k-1} e_1
        \end{bmatrix}_{I_2}) \\
        = & \sum \limits_{\substack{I_2 \subseteq [2,n] \\ |I_2| = k-1}} \Delta_{I_1,I_2 \sqcup \{1\}}(u) \sum \limits_{\substack{I_3, \ldots, I_k \subseteq [2,n] \\ |I_i| = k-i+1 ~ \forall i \in [3,k]}} \prod \limits_{i=2}^k \Delta_{I_i, I_{i+1} \sqcup \{1\}}(u)
        = \sum \limits_{\substack{I_2, \ldots, I_k \subseteq [2,n] \\ |I_i| = k-i+1 ~ \forall i \in [2,k]}} \prod \limits_{i=1}^k \Delta_{I_i, I_{i+1} \sqcup \{1\}}(u),
    \end{align*}
    where the second to last equality follows from induction hypothesis.
\end{proof}

The following lemmas are easy consequences of Theorem \ref{prop:lambdalow} and Proposition \ref{Pro:LowinMinor}.

\begin{cor}\label{cor:lowinminor1}
    Let $(x,y) \in \g_{\rm st}^*$ and $u = x-y$.  Let $i \in [1,n-1], p \in [0,n-2]$ such that $i \le p+1$.  Putting $I_1 = [i+1,n-p+i-1]$, so $k = |I_1| = n-p-1$, up to a sign, we have
    \begin{align*}
        \varphi_{p(n-1)+i}^{\low} (u) = \lambda_{p(n-1)+i}^{\low} (u) = \sum \limits_{\substack{I_2, \ldots, I_k \subseteq [2,n] \\ |I_j| = k-j+1 ~ \forall j \in [2,n-p-1]}} \prod \limits_{j=1}^{n-p-1} \Delta_{I_j, I_{j+1} \sqcup \{1\}}(u).
    \end{align*}
\end{cor}

\begin{cor}\label{cor:lowinminor2}
    Let $(x,y) \in \g_{\rm st}^*$ and $u = x-y$.  Let $i \in [1,n-1], p \in [0,n-2]$ such that $i > p+1$.  Putting $I_1 = [2,i-p-1] \sqcup [i+1,n]$, so $k = |I_1| = n-p-2$, up to a sign, we have
    \begin{align*}
        \varphi_{p(n-1)+i}^{\low} (u) = \lambda_{p(n-1)+i}^{\low} (u) =  \sum \limits_{\substack{I_2, \ldots, I_k \subseteq [2,n] \\ |I_j| = k-j+1 ~ \forall j \in [2,n-p-2]}} \prod \limits_{j=1}^{n-p-2} \Delta_{I_j, I_{j+1} \sqcup \{1\}}(u).
    \end{align*}
\end{cor}

\begin{example} \label{ex:gl4}
    Take $n = 4$.  The two corollaries above say that, up to a sign,
    \begin{align*}
        & \varphi_1^{\low} (u) = \Delta_{[2,4],[1,3]} (u) \varphi_4^{\low} (u) + \Delta_{[2,4],\{1,2,4\}} (u) \varphi_3^{\low} (u) + \Delta_{[2,4],\{1,3,4\}} (u) \varphi_2^{\low} (u), \\
        & \varphi_2^{\low} (u) = \Delta_{[3,4],[1,2]} (u) \Delta_{\{2\},\{1\}} (u) + \Delta_{[3,4],\{1,3\}} (u) \Delta_{\{3\},\{1\}} (u) + \Delta_{[3,4],\{1,4\}} (u) \Delta_{\{4\},\{1\}}, \\
        & \varphi_3^{\low} (u) = \Delta_{\{2,4\},[1,2]} (u) \Delta_{\{2\},\{1\}} (u) + \Delta_{\{2,4\},\{1,3\}} (u) \Delta_{\{3\},\{1\}} (u) + \Delta_{\{2,4\},\{1,4\}} (u) \Delta_{\{4\},\{1\}}, \\
        & \varphi_4^{\low} (u) = \Delta_{[2,3],[1,2]} (u) \Delta_{\{2\},\{1\}} (u) + \Delta_{[2,3],\{1,3\}} (u) \Delta_{\{3\},\{1\}} (u) + \Delta_{[2,3],\{1,4\}} (u) \Delta_{\{4\},\{1\}}, \\
        & \varphi_5^{\low} (u) = \varphi_2^{\low} (u), \\
        & \varphi_6^{\low} (u) = \Delta_{\{4\},\{1\}} (u), \quad \varphi_7^{\low} (u) = \Delta_{\{2\},\{1\}} (u), \quad \varphi_8^{\low} (u) = \Delta_{\{3\},\{1\}} (u), \quad \varphi_9^{\low} (u) = \varphi_6^{\low} (u).
    \end{align*}
    In particular, in terms of the minors of $u$, the function $\varphi_1^{\low}$ only has $9$ terms, which is much simpler than the $86$-term formula in terms of the entries of $u$ mentioned in Example \ref{ex:runningexGLn}.   
\end{example}

To prove the first statement of Theorem \ref{thm:choice-GLn}, we first prove a lemma. Take $(x,y)\in \g_{\rm st}^*$ and set $u=x-y$ with entries $u_{ij}$.  Write $u'$ for $u_{[1,n-1]}^{[1,n-1]}$, so one can apply the functions $\varphi_k^{\low}$ for $\mathfrak{gl}(n-1,\mathbb{C})^{\ast}$ to $u'$.  

\begin{lem} \label{Lem:Ind}
    Assume that $u$ is strictly lower triangular.  Then for $i \le p$, $i \in [1,n-1]$ and $p \in [0,n-2]$, we have, up to a sign,
    \begin{align*}
        \varphi_{p(n-1)+i}^{\low} (u) = \varphi_{(p-1)(n-2)+i}^{\low} (u').
    \end{align*}
    In particular, the function $\varphi_{p(n-1)+i}^{\low} (u)$ is independent of $u_{n1}, \ldots, u_{n,n-1}$.
\end{lem}

\begin{proof}
    Since $u$ is strictly upper triangular, for all $k \in \NN$, we have
    \begin{align*}
        u^k = 
        \begin{bmatrix}
            (u')^k & 0 \\
            u^{[1,n-1]}_{[n,n]}(u')^{k-1} & 0
        \end{bmatrix}.
    \end{align*}

    By Theorem \ref{prop:lambdalow}, up to a sign, we have
    \begin{align*}
        \varphi_{p(n-1)+i}^{\low} (u) = \det
        \begin{bmatrix}
            u e_1 & \cdots & u^{n-p-1} e_1
        \end{bmatrix}_{[i+1,n-p+i-1]}.
    \end{align*}
    Since $i \le p$, we have $n-p+i-1 \le n-1 < n$, so when computing the determinant above, the last row of the matrix does not play a role.  Hence we have, up to a sign,
    \begin{align*}
        & \varphi_{p(n-1)+i}^{\low} (u) = \det
        \begin{bmatrix}
            u' e'_1 & \cdots & (u')^{n-p-1} e'_1
        \end{bmatrix}_{[i+1,n-p+i-1]} \\
        = & \det
        \begin{bmatrix}
            u' e'_1 & \cdots & (u')^{(n-1)-(p-1)-1} e'_1
        \end{bmatrix}_{[i+1,(n-1)-(p-1)+i-1]} = \varphi_{(p-1)(n-2)+i}^{\low} (u'),
    \end{align*}
    where $e'_1$ is the vector $[1 ~ 0 \cdots 0]^T$ in $\CC^{n-1}$ and the last equality holds since  $i \le p = (p-1)+1$.
\end{proof}

\begin{proof}[Proof of the first statement of Theorem \ref{thm:choice-GLn}]
    It suffices to prove that $\{ \varphi_{p(n-1)+i}^\low : i\leq p+1\} \cup\{\overline{c}_i^\low: 0 \le i \le n-1 \}$ consists of independent functions.  By Corollary \ref{cor:lowinminor1}, the entries $u_{11}, \ldots, u_{1n}$ of $u$ do not show up in the functions $\varphi_{p(n-1)+i}^\low$, $i \le p+1$.  Hence, in view of Proposition \ref{lem:degcas}, it suffices to prove that $\{ \varphi_{p(n-1)+i}^\low : i\leq p+1\}$ consists of independent functions.

    Let $\n_-$ stand for the set of all strictly lower triangular $n \times n$ matrices.  We will show, by induction on $n$, that $\{ \bigl( \varphi_{p(n-1)+i}^\low \bigr) \big|_{\n_-} \colon i \in [1,n-1], p \in [0,n-2], i \leq p+1\}$ is independent.

    For the base case $n=2$, we need to show that $\{ (\varphi_1^\low) \big|_{\n_-} = u_{21} \}$ is independent.  This is evident.

    Now assume that $n>2$ and independence has been proved for all strictly smaller values of $n$.  By Lemma \ref{Lem:Ind} and induction hypothesis, for $u \in \mathfrak n_-$, the set $$\{\varphi_{p(n-1)+i}^{\low} (u): i \le p, i \in [1,n-1] ~ \text{and} ~ p \in [0,n-2]\}$$ is independent; and $u_{n1}, \ldots, u_{n,n-1}$ do not appear in these functions.

    Consider those $\varphi_{p(n-1)+i}^{\low}$ that satisfy $i=p+1$.  Up to a sign, these are $$\{\varphi_{p(n-1)+(p+1)}^{\low} (u) = \det 
    \begin{bmatrix}
        u e_1 & \cdots & u^{n-p-1} e_1
    \end{bmatrix}_{[p+2,n]}: p \in [0,n-2]\}.$$  If we can show that the determinant of
    \begin{align*}
        \Jac(u) = \left[
            \frac{\partial \varphi_{p(n-1)+(p+1)}^{\low}}{\partial u_{nj}} (u)
        \right]^{p \in [0,n-2]}_{j \in [1,n-1]}
    \end{align*}
    is a non-zero function, then we are done. By the computation in the proof of Lemma \ref{Lem:Ind}, we have
    \begin{align*}
        u^k e_1 =
        \begin{bmatrix}
            (u')^k e'_1 \\
             u^{[1,n-1]}_{[n,n]} (u')^{k-1} e'_1
        \end{bmatrix}
    \end{align*}
    for all $k \in \NN$.  Thus, representing expressions not involving $u_{n1},\ldots,u_{n,n-1}$ by $\ast$, the $(n,1)$-entry of $u^k e_1$ has the form
    \begin{align*}
        u^{[1,n-1]}_{[n,n]} ((u')^{k-1} e'_1) = [u_{n1} ~ \cdots ~ u_{n,n-1}] [\underbrace{0 ~ \cdots ~ 0}_{k-1} \ast \cdots \ast]^T = \ast u_{nk} + \cdots + \ast u_{n,n-1} =: \bar u_k.
    \end{align*}
    It follows that $\varphi_{p(n-1)+(p+1)}^{\low} (u)$ is a $\CC[u_{ab}: 1 \le b < a \le n-1]$-linear combination of $\bar u_1, \ldots, \bar u_{n-p-1}$. Thus
    \[
        \Jac(u) = \left[
            \frac{\partial \varphi_{p(n-1)+(p+1)}^{\low}}{\partial u_{nj}}
        \right]^{p \in [0,n-2]}_{j \in [1,n-1]} = \left[
            \frac{\partial \bar u_k}{\partial u_{nj}}
        \right]^{k \in [1,n-1]}_{j \in [1,n-1]} \cdot \left[
            \frac{\partial \varphi_{p(n-1)+(p+1)}^{\low}}{\partial \bar u_k} 
        \right]_{k \in [1,n-1]}^{p \in [0,n-2]} =: L_1\cdot L_2.
    \]
    
    Note that $L_1$ is lower triangular. By induction on $k$, one can show that
    \[
    u_{21} u_{32} \cdots u_{k,k-1}
    \]
    is the coefficient for $u_{nk}$ in $\bar u_k$, hence $\det L_1 = \prod \limits_{k=1}^{n-1} u_{21} u_{32} \cdots u_{k,k-1} \neq 0$.
    
    Note also that $L_2$ has the property that its $(k,p)$-entry is zero for $k+p \geq n$.  Hence, up to a sign, we have $\det L_2 = \prod \limits_{p=0}^{n-2} \frac{\partial \varphi_{p(n-1)+(p+1)}^{\low}}{\partial \bar u_{n-p-1}}$.  For $p \in [0,n-2]$, the coefficient for $\bar u_{n-p-1}$ in $\varphi_{p(n-1)+(p+1)}^{\low} (u)$ is
    \begin{align*}
        \det
        \begin{bmatrix}
            u e_1 & \cdots & u^{n-p-2} e_1
        \end{bmatrix}_{[p+2,n-1]},
    \end{align*}
    which is non-zero by Lemma \ref{Lem:non-zeroA}. Thus $\det L_2 \neq 0$. Then we conclude $\det \Jac(u) \neq 0$.
\end{proof}

\Addresses
\end{document}